\documentclass[11pt,a4paper,oneside,centertags,reqno]{amsart}

\usepackage[title]{appendix}

\usepackage{etoolbox}
\robustify{\footnote}

\usepackage[multiple]{footmisc}
\usepackage{txfonts}
\usepackage{exscale}
\usepackage{amsbsy}
\usepackage[colorlinks=true,citecolor=red,pagebackref,hypertexnames=false]{hyperref}
\usepackage{lmodern}
\usepackage{endnotes}
\usepackage{latexsym}
\usepackage{amsmath}
\usepackage{amsfonts}
\usepackage{amssymb}
\usepackage{graphpap}
\usepackage{epsfig}
\usepackage{amscd}
\usepackage{amsthm}
\usepackage{enumerate}
\usepackage{enumitem}
\usepackage{eucal}
\usepackage[francais]{[babel}
\usepackage[deutsch]{[babel}
\usepackage[ansinew]{inputenc}
\usepackage{graphics}
\usepackage{dsfont}

\usepackage[backrefs,non-compressed-cites,initials,nobysame,]{amsrefs}
\usepackage{pdfsync}
\usepackage{color}
\usepackage{mathrsfs}

\usepackage{thmtools}
\usepackage{hyperref}
\usepackage{cleveref}

\declaretheoremstyle[
  headfont=\bfseries,
  bodyfont=\normalfont,
  numbered=no
]{namedstyle}

\declaretheorem[name={Assumption BE}, style=namedstyle]{assBC}

\newcommand{\AssBE}{Assumption~\hyperref[ass:BE]{BE}}

\usepackage{bbm}

\font\got=eufm10 at 11pt

\font\posebni=msam10

\numberwithin{equation}{section}

\renewcommand{\Re}[0]{{\rm Re}\,}
\renewcommand{\Im}[0]{{\rm Im}\,}

\newcommand{\ca}[0]{\mathbbm{1}}

\newcommand{\C}[0]{{\mathbb C}}

\newcommand{\f}[0]{\varphi}

\newcommand{\N}[0]{{\mathbb N}}

\newcommand{\R}[0]{\mathbb{R}}

\newcommand{\leqsim}[0]{\,\text{\posebni \char46}\,}
\newcommand{\geqsim}[0]{\,\text{\posebni \char38}\,}

\newcommand{\cA}[0]{{\mathcal A}}

\newcommand{\cD}[0]{{\mathcal D}}
\newcommand{\cE}[0]{{\mathcal E}}
\newcommand{\cF}[0]{{\mathcal F}}

\newcommand{\cI}[0]{{\mathcal I}}

\newcommand{\cL}[0]{{\mathcal L}}
\newcommand{\cM}[0]{{\mathcal M}}
\newcommand{\cN}[0]{{\mathcal N}}
\newcommand{\cO}[0]{{\mathcal O}}
\newcommand{\cP}[0]{{\mathcal P}}
\newcommand{\cQ}[0]{{\mathcal Q}}
\newcommand{\cR}[0]{{\mathcal R}}

\newcommand{\cV}[0]{{\mathcal V}}
\newcommand{\cW}[0]{{\mathcal W}}

\newcommand{\oL}[0]{{\mathscr L}}
\newcommand{\oP}[0]{{\mathscr P}}
\newcommand{\oV}[0]{{\mathscr V}}
\newcommand{\oW}[0]{{\mathscr W}}
\newcommand{\oN}[0]{{\mathscr N}}

\newcommand{\ovR}[0]{\overline{\rm R}}

\newcommand{\gota}[0]{{\text{\got a}}}
\newcommand{\gotb}[0]{{\text{\got b}}}

\newcommand\bP{\mathbf{P}}
\newcommand\bS{\mathbf{S}}

\newcommand{\mn}[2]{\{ #1 : #2 \}}
\newcommand{\Mn}[2]{\left\{ #1 : #2 \right\}}
\newcommand{\sk}[2]{\left\langle #1 , #2\right\rangle}

\renewcommand{\div}[0]{{\rm div}\,}

\newcommand{\Dom}[0]{{\rm D}}

\newtheorem{theorem}{Theorem}[section]
\newtheorem{defi}[theorem]{Definition}
\newtheorem{lemma}[theorem]{Lemma}
\newtheorem{proposition}[theorem]{Proposition}
\newtheorem{corollary}[theorem]{Corollary}

\theoremstyle{remark}
\newtheorem{remark}[theorem]{Remark}  

\newtheorem*{notation}{Notation}  

\theoremstyle{definition}

\renewcommand\leq[0]{\leqslant}
\renewcommand\geq[0]{\geqslant}
\renewcommand\epsilon[0]{\varepsilon}
\renewcommand\theta[0]{\vartheta}

\newcommand\wrt{\,\text{\rm d}}
\renewcommand\mod[1]{\left\vert{#1}\right\vert}

\newcommand\norm[2]{{\left\Vert{#1}\right\Vert_{#2}}}

\begin{document}

\bibliographystyle{plain}

\title[Bilinear embedding for negative potentials]{Bilinear embedding for divergence-form operators with negative potentials}
\author[Poggio]{Andrea Poggio}

\date{\today}

\address{Andrea Poggio \\Universit\`a degli Studi di Genova\\ Dipartimento di Matematica\\ Via Dodecaneso\\ 35 16146 Genova\\ Italy }
\email{andrea.poggio@edu.unige.it}

\begin{abstract}
Let $\Omega \subseteq \R^d$ be open, $A$ a complex uniformly strictly accretive $d\times d$ matrix-valued function on $\Omega$ with $L^\infty$ coefficients, and $V$ a locally integrable function whose negative part is subcritical. We consider the operator $\oL = -\mathrm{div}(A\nabla) + V$ with mixed boundary conditions on $\Omega$.  

We extend the bilinear embedding of Carbonaro and Dragi\v{c}evi\'c \cite{CD-Potentials}, established for nonnegative potentials under the $p$-ellipticity of the matrices, by introducing a novel condition on the coefficients that reduces to $p$-ellipticity when $V$ is nonnegative. As a consequence, the solution to the parabolic problem $u^\prime(t) + \oL u(t) = f(t)$ with $u(0)=0$ has maximal regularity on $L^p(\Omega)$.  

Moreover, under this new condition, we study mapping properties of the semigroup generated by $-\oL$,  thereby extending classical results for the  Schr\"{o}dinger operator $-\Delta + V$ on $\mathbb{R}^d$.

\end{abstract}
\maketitle

\section{Introduction and statement of the main results}
\label{s: Neumann introduction}
Let $\Omega\subseteq\R^{d}$ be a nonempty open set. Denote by $\cA(\Omega)$ the class of all complex uniformly strictly elliptic $d\times d$ matrix-valued functions on $\Omega$ with $L^{\infty}$ coefficients (in short, elliptic matrices). That is to say, $\cA(\Omega)= \cA_{\lambda,\Lambda}(\Omega)$ is the class of all measurable $A:\Omega\rightarrow \C^{d\times d}$ for which there exist $\lambda$, $\Lambda>0$ such that for almost all $x\in\Omega$ we have
\begin{equation}
\aligned
\label{eq: N ellipticity}
\Re\sk{A(x)\xi}{\xi}
&\geq \lambda|\xi|^2\,,
&\quad\forall\xi\in\C^{d};
\\
\mod{\sk{A(x)\xi}{\sigma}}
&\leq \Lambda \mod{\xi}\mod{\sigma}\,,
&\quad\forall\xi,\sigma\in\C^{d}.
\endaligned
\end{equation}
We denote by $\lambda(A)$ and $\Lambda(A)$ the optimal $\lambda$ and $\Lambda$, respectively.

The main goal of this paper is to extend the bilinear embedding proved in \cite[Theorem~1.4]{CD-Potentials} to Schr\"{o}dinger-type operators with potentials $V$ that may take {\it negative} values, formally given by
$$
\oL u = -{\rm div}(A \nabla u) + Vu, \quad ({\rm on } \,	\, \Omega).
$$
These operators are defined on $L^2(\Omega)$ in the weak sense through a sesquilinear form, where different types of boundary conditions are incorporated through the choice of the form domain; see Section~\ref{ss: op} and Section~\ref{s: boundary} below.

Bilinear embedding theorems have become essential tools in harmonic analysis, with applications ranging from dimension-free bounds for Riesz transforms \cite{DV, Dv-Sch, DraVol, CD-Riesz} to sharp spectral multiplier results \cite{CD-OU, CD-mult}. 
The theory was subsequently expanded to divergence-form operators, with Dragi\v{c}evi\'c and Volberg \cite{Dv-kato} providing the first contribution by establishing a dimension-free bilinear embedding on the {\it whole} space $\R^d$ for Schr\"{o}dinger-type operators associated with {\it real} elliptic matrices and {\it nonnegative} potentials. From this starting point, the theory was progressively generalized in two directions: first, by moving from real to {\it complex} elliptic matrices, and second, by extending the setting from the whole space $\R^d$ to {\it arbitrary} open subsets of $\R^d$. 
A first example of bilinear embedding for complex matrices on $\mathbb{R}^d$ was obtained in \cite{AuHM} as a consequence of conical square function estimates. These techniques differed from those in \cite{Dv-kato} and did not yield the dimension-free property; this was later established by Carbonaro and Dragi\v{c}evi\'c in \cite{CD-DivForm} via the introduction of the notion of $p$-ellipticity for the coefficient matrices.  
They subsequently developed an approximation method to extend the bilinear embedding to arbitrary open domains and showed how the bilinear embedding implies the boundedness of the $H^\infty$-functional calculus and the $L^p$-maximal regularity of $\oL$ \cite{CD-Mixed}. More recently, nonnegative potentials have been incorporated into the complex framework, with bilinear embeddings established under the sole assumption of $p$-ellipticity of the matrices \cite{CD-Potentials}.

To the best of our knowledge, no results are available so far for potentials that may take  negative values. The novelty of this paper lies in the treatment of such potentials: we consider a specific class in which the negative part of $V$ is controlled through a subcritical inequality; see Section~\ref{ss: ssp}. To establish the bilinear embedding for these operators, we introduce a new condition that coincides with standard $p$-ellipticity when the potential is nonnegative and otherwise accounts for its negative part. More precisely, suitable perturbations of the coefficient matrices---reflecting this negative component---are required to remain $p$-elliptic; see Section~\ref{s: new cond}. 
\smallskip

When $V$ is nonnegative,  $p$-ellipticity  has also proven central in the extrapolation of the semigroup $T=(T_t)_{t>0}$ generated by $-\oL$ from $L^2(\Omega)$ to $L^r(\Omega)$. Carbonaro and Dragi\v{c}evi\'c \cite{CD-DivForm, CD-Mixed, CD-Potentials} proved that it guarantees {\it holomorphy} and {\it $L^r$-contractivity} of $T$ for all exponents $r$ satisfying $|1/2-1/r|\leq |1/2-1/p|$. In the case $V=0$ and under additional assumptions on $\Omega$, Egert \cite{Egert20} extended the range of {\it uniform boundedness} at least to
    \begin{equation*}
    |1/2-1/r| \leq 1/d +(1-2/d)|1/2-1/p|,
    \end{equation*}
    which enlarges the sharp range for strictly elliptic matrices \cite{Auscher1, HMMcl, Dav1, Egert20} through the use of structural information on the coefficients. We will show that analogous results continue to hold for Schr\"{o}dinger-type operators with negative potentials under the new perturbed $p$-ellipticity; see Section~\ref{s: sem emb} for precise statements. 
When $A$ is the identity matrix, this condition recovers the same range of  $L^p$-contractivity previously obtained for the classical Schr\"{o}dinger operator $-\Delta + V$ on $\R^d$ \cite{Per, BS90} and on certain complete Riemannian manifolds \cite{AO}.

\subsection{Strongly subcritical potentials}\label{ss: ssp}
Let $\oV$ be a closed subspace of  $W^{1,2}(\Omega)$ containing $W^{1,2}_0(\Omega)$. 

\begin{defi}
A real locally integrable function $V$  is said to be a strongly subcritical potential for $\oV$ if there exist $\alpha \geq 0$ and $\sigma \in [0,1)$ such that
\begin{equation}\label{e : subc ineq}
\int_\Omega V_- |v|^2 \leq \alpha \int_\Omega |\nabla v|^2 
+ \sigma \int_\Omega V_+ |v|^2, 
\qquad \forall v \in \oV.
\end{equation}
We denote by $\cP(\Omega,\oV)$ the class of all strongly subcritical potentials for $\oV$.  
When $\oV$ is clear from the context, we simply write $\cP(\Omega)$.
\end{defi}

For fixed $\alpha \geq 0$ and $\sigma \in [0,1)$ we write 
$\cP_{\alpha,\sigma}(\Omega,\oV)$, or simply $\cP_{\alpha,\sigma}(\Omega)$, for the subclass of $\cP(\Omega)$ in which 
\eqref{e : subc ineq} holds with these constants. Hence,
$$
\cP(\Omega) = \bigcup_{\alpha \geq 0,\, \sigma \in [0,1)} 
\cP_{\alpha,\sigma}(\Omega).
$$
Given $V \in \cP(\Omega,\oV)$, we define $\alpha(V,\oV)$ (or simply $\alpha(V)$ 
if no ambiguity arises) as the smallest admissible $\alpha$ for which 
\eqref{e : subc ineq} holds with some $\sigma \in [0,1)$. The corresponding 
constant $\sigma$ will be denoted by $\sigma(V,\oV)$, or simply $\sigma(V)$.

In \cite{Per, BS90}, the restriction $\alpha \in (0,1)$ is imposed in the definition of strongly subcritical potentials. Here, however, we also allow $\alpha = 0$, since we want $\alpha(V,\oV)=0$ whenever $V$ is nonnegative, so that the new structural conditions introduced later on the coefficients coincide with those in \cite{CD-Potentials} when $V$ is nonnegative; see \eqref{eq: new condit p2} and Section~\ref{s: new cond}.
Indeed, if $V_- = 0$, then clearly $V \in \cP(\Omega,\oV)$ and $\alpha(V,\oV)=0$ for every $\oV$, and conversely, if $\alpha(V,\oV)=0$, then necessarily $V_- = 0$. In fact, suppose that
\begin{equation}\label{eq: suc ineq no a}
\int_\Omega V_- |v|^2 \leq \sigma \int_\Omega V_+ |v|^2, 
\qquad \forall v \in \oV,
\end{equation}
and assume that $S_k := \{ V_- > 1/k \}$ has positive measure for some $k \in \N$. 
Take $K \subseteq S_k$ with $0 < |K| < \infty$ and $\overline{K} \subset \Omega$. Then $V_\pm \in L^1(K)$. Let $(\varphi_\varepsilon)_{\varepsilon>0}$ be a family of mollifiers with $0 \leq \varphi_\varepsilon \leq 1$, and set 
$v_\varepsilon := \mathds{1}_K * \varphi_\varepsilon$. Then, for $\varepsilon$ 
small enough, $v_\varepsilon \in C_c^\infty(\Omega) \subseteq \oV$. By 
\eqref{eq: suc ineq no a} and the dominated convergence theorem,
$$
0 < \frac{1}{k}|K| 
< \int_K V_- 
= \lim_{\varepsilon \to 0} \int_\Omega V_- |v_\varepsilon|^2 
\leq \sigma \lim_{\varepsilon \to 0} \int_\Omega V_+ |v_\varepsilon|^2 
= \int_K V_+ = 0,
$$
which is a contradiction. Hence $|S_k|=0$ for all $k \in \N$, which implies $V_-=0$.

On the other hand, the restriction $\alpha <1$ in \cite{AO,BS90,  Per} stems from the fact that the authors study the classical Schr\"{o}dinger operator $-\Delta + V$. In our setting, we deal instead with the more general Schr\"{o}dinger-type operator $-\mathrm{div}(A \nabla \cdot) + V$, where $A$ is an elliptic matrix-valued function. In this case, rather than assuming $\alpha <1$, we impose an upper bound on $\alpha$ that depends on $A$; see again \eqref{eq: new condit p2}. When $A = I_d$, this  bound precisely reduces  to $\alpha <1$.
\smallskip

It is immediate to observe that if $\oV \subseteq \oW$, then $\cP(\Omega, \oW) \subseteq \cP(\Omega, \oV)$.
This inclusion may be strict. In Section~\ref{s: ex ss pot}, we will provide examples of strongly subcritical potentials and show that, for suitable open subsets $\Omega$,
$$
\cP(\Omega, W^{1,2}(\Omega)) \subsetneq \cP(\Omega, W^{1,2}_D(\Omega)) \subsetneq \cP(\Omega, W^{1,2}_0(\Omega)).
$$
The definition of $W_D^{1,2}(\Omega)$ will be given later in Section~\ref{s: boundary}.
 
\subsection{The operator $-{\rm div}(A\nabla u) + Vu$}\label{ss: op}
Let $\oV$ be  a closed subspace  of $W^{1,2}(\Omega)$ containing $W^{1,2}_{0}(\Omega)$. 
 Suppose that $A \in \cA(\Omega)$ and $V \in \cP_{\alpha,\sigma}(\Omega,\oV)$. 
Consider the sesquilinear form $\gota = \gota_{A,V,\oV}$ defined by
\begin{equation}
\label{e : form sesq}
\aligned
\Dom(\gota) &=\Mn{u\in\oV}
 {\int_\Omega V_+|u|^2 < \infty}, \nonumber 
\\
\gota(u,v)&= \displaystyle \int_{\Omega}\sk{A\nabla u}{\nabla v}_{\C^{d}}+ V u \overline{v}.
\endaligned
\end{equation}
Clearly, $\gota$ is densely defined in $L^2(\Omega)$. Suppose furthermore that
\begin{equation}
\label{eq: new condit p2}
A-\alpha I_d \in \cA(\Omega).
\end{equation}
Then, by using the abbreviation $\gota(u)=\gota(u,u)$, \eqref{e : subc ineq} and \eqref{eq: new condit p2}, for all $u \in \Dom(\gota)$ we have
\begin{eqnarray}
\label{e : below bound Rea}
\Re \gota (u) &=& \int_\Omega \Re\sk{A\nabla u}{\nabla u} +  V_+|u|^2 - V_-|u|^2 \nonumber \\
& \geq&\int_\Omega \Re\sk{(A-\alpha I_d) \nabla u}{\nabla u} + (1-\sigma) V_+|u|^2 \\
& \geq & \int_\Omega \lambda(A-\alpha I_d)|\nabla u|^2 + (1-\sigma) V_+|u|^2.\nonumber
\end{eqnarray}
Therefore, $\gota$ is accretive.

Denote $\gotb =  \gota_{I_d,V_+,\oV}$. We know from \cite[Proposition 4.30]{O} that $\gotb$ is closed. On the other hand, \eqref{e : subc ineq} and \eqref{eq: new condit p2} give 
\begin{equation*}
\aligned
\Re \gota(u) &\geq \min \{\lambda(A-\alpha I_d), (1-\sigma) \} \Re \gotb(u),\\
\Re \gota(u) &\leq \max \{\Lambda(A+\alpha I_d), (1+\sigma) \} \Re \gotb(u),
\endaligned
\end{equation*}
for all $u \in \Dom(\gota)=\Dom(\gotb)$. Hence, $\gota$ is closed.

Given $\phi \in (0,\pi)$ define the sector
$$
\bS_{\phi}=\mn{z\in\C\setminus\{0\}}{|\arg (z)|<\phi}.
$$
Also set $\bS_0=(0,\infty)$. By \eqref{eq: N ellipticity} and \eqref{e : below bound Rea} we also have
\begin{equation}
\label{e : above boun a}
|\Im \gota(u)| \leq \frac{\sqrt{\Lambda^2(A)-\lambda^2(A)}}{\lambda(A-\alpha I_d)}\Re\gota(u), \quad \forall u \in \Dom(\gota),
\end{equation}
which  implies that $\gota$ is sectorial of angle 
$$
\theta_0 := \arctan\left(\frac{\sqrt{\Lambda^2(A)-\lambda^2(A)}}{\lambda(A-\alpha I_d)}\right)  \in(0, \pi/2),
$$
 in the sense of \cite{Kat}, meaning that its numerical range $\text{Nr}(\gota) = \{\gota(u): u\in \Dom(\gota), \, \|u\|_2=1\}$ satisfies
\begin{equation}
\label{eq: sect numer range}
\text{Nr}(\gota) \subseteq \overline{\bS}_{\theta_0}.
\end{equation}
Denote by $\oL=\oL^{A,V,\oV}_{2}$ the unbounded operator on $L^{2}(\Omega)$ associated with the sesquilinear form $\gota$.
That is,
 $$
 \Dom(\oL):=\Mn{u\in\Dom(\gota)}
 {\exists w\in L^2(\Omega):\ 
 \gota(u,v)=\sk{w}{v}_{L^2(\Omega)}\ \forall v\in \Dom(\gota)}
$$
and 
\begin{equation}
\label{eq: ibp}
\sk{\oL u}{v}_{L^2(\Omega)} = \gota(u,v), \quad \forall u\in\Dom(\oL),\quad \forall v\in\Dom(\gota)\,.
\end{equation}
Formally, $\oL$ is given by the expression
$$
 \oL u=-\div(A\nabla u) + V u.
$$
It follows from \eqref{eq: sect numer range} that $-\oL$ is the generator of a strongly continuous semigroup on $L^{2}(\Omega)$
$$
T_t =T^{A,V,\oV}_{t}, \quad t>0,
$$
 which is analytic and contractive in the cone $\bS_{\pi/2-\theta_0}$. 
For details and proofs see \cite[Chapter VI]{Kat} and \cite[Chapters I and IV]{O}.
\medskip

Notice that, by taking $A = I_d$, condition \eqref{eq: new condit p2} reduces to $(1-\alpha)I_d \in \cA(\Omega)$, which is equivalent to requiring $\alpha < 1$, consistently with \cite{AO, Per}.

\subsection{Boundary conditions}\label{s: boundary}
Here we describe certain classes of closed subspace $\oV$ of $W^{1,2}(\Omega)$ containing $W_0^{1,2}(\Omega)$ that satisfy additional conditions which will be assumed throughout the rest of the paper. 
We follow \cite{CD-Potentials}.
\smallskip

We say that the space $\oV \subseteq W^{1,2}(\Omega)$ is {\it invariant} under:
\begin{itemize}
\item the function $p: \C \rightarrow \C$, if $u \in \oV$ implies $p(u):= p \circ u \in \oV$;
\item the family $\oP$ of functions $\C \rightarrow \C$, if its invariant under all $p \in \oP$.
\end{itemize}

Define functions $P, T : \C \rightarrow \C$ by
\begin{equation*}
\aligned
P(\zeta)&= 
\begin{cases}
\zeta; & |\zeta| \leq 1,\\
\zeta/|\zeta|; &|\zeta|\geq 1,
\end{cases}\\
T(\zeta) &= \, (\Re\zeta)_+.
\endaligned
\end{equation*}
Thus $P(\zeta) = \min\{1,|\zeta|\}{\rm sign}\zeta$, where ${\rm sign}$ is defined as \cite[(2.2)]{O}:
$$
{\rm sign}\zeta :=
\begin{cases}
\zeta/|\zeta|; & \zeta \ne 0,\\
0; &\zeta=0.
\end{cases}
$$
Let $\oV$ be a closed subspace of $W^{1,2}(\Omega)$ containing $W_0^{1,2}(\Omega)$. We consider the invariance conditions 
\begin{equation}
\label{eq: inv P}
\oV \text{ is invariant under the function } P,
\end{equation}
\begin{equation}
\label{eq: inv N}
\oV \text{ is invariant under the function } T,
\end{equation}
and assume either only \eqref{eq: inv P} or both \eqref{eq: inv P} and \eqref{eq: inv N}, depending on the context.
In general, \eqref{eq: inv P} does not imply \eqref{eq: inv N}; see \cite[Example~4.3.2]{O}.

It is well known (see \cite[Proposition~4.4\&4.11]{O}, \cite[Section~2,1]{CMR-NumRange} and \cite[Appendix A]{CD-Potentials}) that \eqref{eq: inv P} and \eqref{eq: inv N}  are satisfied in these notable cases:
\begin{enumerate}[label=(\alph*)]
\item \label{i: D} $\oV = W_0^{1,2}(\Omega)$,
\item \label{i: N} $\oV = W^{1,2}(\Omega)$,
\item \label{i: bM} $\oV = \widetilde{W_D}^{1,2}(\Omega)$, the closure in $W^{1,2}(\Omega)$ of $\{u \in W^{1,2}(\Omega): {\rm dist}({\rm supp} \,u, D) >0\}$, where $D$ is a (possibly empty) closed subset of $\partial \Omega$,
\item \label{i: cM} $\oV= W_D^{1,2}(\Omega)$, the closure in $W^{1,2}(\Omega)$ of $\{u_{\vert_{\Omega}}: u\in C^{\infty}_{c}(\R^{d}\setminus D)\}$, where $D$ is a (possibly empty) proper closed subset of $\partial \Omega$.
\end{enumerate}

When $\oV$ falls into any of the special cases \ref{i: D}-\ref{i: cM}, we say that $\oL= \oL^{A,V,\oV}$ is subject to \ref{i: D} {\it Dirichlet}, \ref{i: N} {\it Neumann}, \ref{i: bM} {\it mixed}, or  \ref{i: cM} {\it good mixed boundary conditions}.
\smallskip

When working with a pair of spaces $\oV$ and $\oW$, we will sometimes need to select them appropriately from the spaces listed above. Certain combinations will not be considered. The following assumption is introduced because  \cite[Lemma~19]{CD-Mixed} may fail if $(\oV, \oW)$ does not satisfy it. 
In particular, if one space is of type  \ref{i: cM}, the other must belong either to the same class or to that described in \ref{i: D}. See Remark~\ref{r : lemma19 wrong}. We will formulate and prove the bilinear embedding theorem under the following additional requirement on $\oV$ and $\oW$ described in \ref{i: D}-\ref{i: cM}. The general case, where $\oV$ and $\oW$ are arbitrary combinations of the types listed in \ref{i: D}-\ref{i: cM}, cannot be treated using the heat-flow method of \cite{CD-Mixed}.
\begin{assBC}\label{ass:BE}
We say that the pair $(\oV, \oW)$ satisfies the \AssBE \, if either of the following holds: 
\begin{itemize}
\item $\oV$ and $\oW$ fall into any of the special cases \ref{i: D}-\ref{i: bM}, or 
\item $\oV$ and $\oW$ are of the type described in \ref{i: D} or \ref{i: cM}.
\end{itemize}
\end{assBC}
The combinations  allowed by \AssBE \ are summarized in Table~\ref{t: assBE}.
\begin{remark}
This assumption is imposed only in the bilinear embedding theorem and must be added to the statements of the corresponding theorems in \cite{CD-Mixed, CD-Potentials, Poggio}. We emphasize that all results  in \cite{CD-Mixed, Poggio} derived from the bilinear embedding---such as the boundedness of the $H^\infty$-functional calculus and $L^p$-maximal regularity---remain unaffected by the error, since they follow from applying the bilinear embedding to an operator and its adjoint, which are subject to the same boundary conditions. The fact that this assumption is necessary for the statement of \cite[Lemma~19]{CD-Mixed} does not imply that it is required for the bilinear embedding itself; nevertheless, a new approach must be pursued.
\end{remark}

The space described in \ref{i: bM} is new in the context of bilinear embedding on arbitrary domains. It is introduced to allow combinations of mixed and Neumann boundary conditions, which would otherwise be prohibited  by \AssBE.

\setcounter{table}{0}
\begin{table}[h]
\begin{center}
\begin{tabular}{|l|c|c|c|c|r|}
\hline
 & $W_0^{1,2}(\Omega)$ & $W_D^{1,2}(\Omega)$ & $\widetilde{W_D}^{1,2}(\Omega)$ & $W^{1,2}(\Omega)$ \\
\hline
$W_0^{1,2}(\Omega)$ & \checkmark & \checkmark & \checkmark & \checkmark \\
\hline
$W_D^{1,2}(\Omega)$ & \checkmark & \checkmark & $\times$  & $\times$ \\
\hline
$\widetilde{W_D}^{1,2}(\Omega)$ & \checkmark & $\times$ & \checkmark & \checkmark \\
\hline
$W^{1,2}(\Omega)$ & \checkmark & $\times$ & \checkmark & \checkmark\\ 
\hline
\end{tabular}
\end{center}
\caption{\AssBE: allowed and forbidden pairs of Sobolev spaces; 
the symbol \checkmark\ indicates that the pair is allowed, while $\times$ that it is not.}
\label{t: assBE}
\end{table}

\subsection{The $p$-ellipticity condition} 
We summarize the following notion, which Carbonaro and Dragi\v{c}evi\'c introduced in \cite{CD-DivForm}.

Given $A \in \cA(\Omega) $ and $p \in (1, \infty)$, we say that $A$ is  {\it $p$-elliptic} if $\Delta_p(A) >0$, where
\begin{equation*}
\Delta_p(A):=
\underset{x\in\Omega}{{\rm ess}\inf}\min_{|\xi|=1}\, \Re\sk{A(x)\xi}{\xi+|1-2/p| \overline{\xi}}_{\C^{d}}.
\end{equation*}
Equivalently, $A$ is $p$-elliptic if there exists $C=C(A,p) >0$ such that for a.e. $x \in \Omega$,
\begin{equation}
\label{eq: Delta>0}
\Re\sk{A(x)\xi}{\xi+|1-2/p| \overline{\xi}}_{\C^{d}}
\geq C |\xi|^2\,,
\quad \forall\xi\in\C^{d}.
\end{equation}

Denote by $\cA_p(\Omega)$ the class of all $p$-elliptic matrix functions on $\Omega$.
Clearly, $\cA(\Omega)=\cA_2(\Omega)$.
A bounded matrix function $A$ is real and elliptic if and only if it is $p$-elliptic for all $p>1$ \cite{CD-DivForm}. For further properties of the function $p\mapsto \Delta_{p}(A)$ we also refer the reader to \cite{CD-DivForm}.

At the same time, Dindo\v s and Pipher in \cite{Dindos-Pipher} found a sharp condition which permits proving reverse H\"older inequalities for weak solutions to ${\rm div}(A\nabla u)=0$ with complex $A$. It turned out that this condition was precisely a reformulation of $p$-ellipticity \eqref{eq: Delta>0}.

A condition similar to \eqref{eq: Delta>0}, namely $\Delta_p(A)\geq0$, was formulated in a different manner by Cialdea and Maz'ya in \cite[(2.25)]{CiaMaz}. See \cite[Remark 5.14]{CD-DivForm}.

The $p$-ellipticity proved to be a rather natural condition through several examples where it featured: bilinear embedding \cite[Theorem~1.3]{CD-DivForm}, \cite[Theorem~2]{CD-Mixed}, semigroup contractivity \cite[Theorem~1.3]{CD-DivForm}, bounded $H^\infty$-functional calculus and parabolic maximal regularity \cite[Theorem~3]{CD-Mixed}.

\subsection{The perturbed $p$-ellipticity}\label{s: new cond}
In the spirit of \cite{Poggio}, we aim to introduce a new condition on the coefficients of the operator $\oL^{A,V,\oV}$ that generalizes $p$-ellipticity and plays an analogous role in the context of bilinear embeddings, semigroup contractivity and bounded $H^\infty$-functional calculus on $L^p$.

Let $p>1$ and let $q$ denote its conjugate exponent, i.e., $1/p+1/q=1$. Let $A \in \cA_p(\Omega)$, $\oV$ be a closed subspace of $W^{1,2}(\Omega)$ containing $W_0^{1,2}(\Omega)$ and $V \in \cP(\Omega,\oV)$. We say that
\begin{itemize}
\item $(A,V) \in \widetilde{\cA\cP}_p(\Omega,\oV)$ if 
\begin{equation}
\label{eq: weak new cond}
\Delta_p\left(A - \alpha(V,\oV) \frac{pq}{4} I_d\right) \geq 0;
\end{equation}
\item $(A,V) \in \cA\cP_p(\Omega,\oV)$ if 
\begin{equation}
\label{eq: new cond}
\Delta_p\left(A - \alpha(V,\oV) \frac{pq}{4} I_d\right) > 0,
\end{equation}
that is, $A - \alpha(V,\oV) (pq)/4 \, I_d$ is $p$-elliptic.
\end{itemize}

When $V_-=0$, clearly $V \in \cP(\Omega,\oV)$ and $\alpha(V,\oV)=0$ for all  $\oV$. 
Hence, in this case \eqref{eq: weak new cond} coincides with the weak $p$-ellipticity ($\Delta_p(A) \geq0$), while \eqref{eq: new cond} with $p$-ellipticity, namely,
$$
A \in \cA_p(\Omega) \iff (A,V) \in \cA\cP_p(\Omega,\oV) \, \text{ for all }\, \oV.
$$

Moreover, the class $\cA\cP_p$ retains a lot of properties that the classes $\cA_p$ possess, such as an invariance under conjugation, a decrease with respect to $p$, an invariance under adjointness; see Proposition~\ref{p : basic prop}\ref{i : inv conj},\ref{i : interpol},\ref{i : inv adj}.

We have introduced condition \eqref{eq: new cond} under the standing assumption that $A \in \cA_p(\Omega)$, since by the definition of $\Delta_p$, $p$-ellipticity of $A$ is necessary for \eqref{eq: new cond} to hold.

Finally, observe that \eqref{eq: new cond} coincides with \eqref{eq: new condit p2} when $p=2$. We then set $\cA\cP(\Omega,\oV) = \cA\cP_2(\Omega,\oV)$.
Whenever no confusion arises, we simply write $ \widetilde{\cA\cP}_p(\Omega)$ and $\cA\cP_p(\Omega)$ in place of $ \widetilde{\cA\cP}_p(\Omega,\oV)$ and $ \cA\cP_p(\Omega,\oV)$, respectively.

\subsection{Semigroup properties on $L^p$}\label{s: sem emb}
As a first result, we aim to generalize \cite[Theorem 1.2]{CD-Potentials} through Theorem \ref{t : contract}. Carbonaro and Dragi\v{c}evi\'c proved it by combining a theorem of Nittka \cite[Theorem 4.1]{Nittka}  with \cite[Theorem 4.31]{O}. We will adapt their strategy  to prove Theorem \ref{t : contract}, with the main novelty being the introduction of the new condition of Section~\ref{s: new cond} that extends the one in \cite[Theorem 1.2]{CD-Potentials} (namely, $\Delta_p( e^{i \phi} A) \geq 0$) to ensure the $L^p$-dissipativity of the form. See Section~\ref{ss : Lp contr} for the explanation of terminology and the proof. This approach has been employed in earlier works \cite{CD-DivForm, CD-Mixed, Egert20}, prior to \cite{CD-Potentials}, and was further developed in \cite{Poggio}.

\begin{theorem}
\label{t : contract}
Suppose that $\oV$ satisfies \eqref{eq: inv P} and \eqref{eq: inv N}. Choose $p>1$, $(A,V) \in \cA\cP(\Omega,\oV)$ and $\phi \in \R$ such that $|\phi| <\pi/2 - \theta_0$ and $(e^{i\phi}A, (\cos\phi)V) \in  \widetilde{\cA\cP}_p(\Omega,\oV)$. Then
$$
\left( e^{-t e^{i\phi}\oL} \right)_{t >0}
$$ 
extends to a strongly continuous semigroup of contractions on $L^p(\Omega)$.

If $V_-=0$, the same conclusion holds under milder assumptions on $\oV$, namely, when $\oV$ only satisfies \eqref{eq: inv P}.
\end{theorem}
The next corollary extends \cite[Corollary 1.3]{CD-Potentials} which was in turn a generalization of \cite[Lemma~17]{CD-Mixed}.
\begin{corollary}\label{c: N analytic sem}
Suppose that $\oV$ satisfies \eqref{eq: inv P} and \eqref{eq: inv N}. Choose $p>1$ and $(A,V) \in \cA\cP_p(\Omega,\oV)$.
Then there exists $\theta=\theta(p,A,V,\oV) >0$ such that if $|1-2/r|\leq |1-2/p|$, then $\{T_z \, : \, z \in \bS_\theta\}$ is analytic and contractive in $L^r(\Omega)$.

If $V_-=0$, the same conclusion holds under milder assumptions on $\oV$, namely, when $\oV$ only satisfies \eqref{eq: inv P}.
\end{corollary}

As a consequence of Corollary~\ref{c: N analytic sem}, we obtain the following generalization of \cite[Theorem~1]{Egert20}.  We further assume that $\oV$ satisfies certain embedding properties (see Definition~\ref{d: emb prop}) and is invariant under multiplication by bounded Lipschitz functions, that is, 
\begin{equation}
\label{eq: inv mult lip}
u \in \oV \Longrightarrow \varphi u \in \oV,
\end{equation}
for every bounded Lipschitz function $\varphi: \R^d \rightarrow \R$.
\begin{corollary}
\label{t: eg extr}
Let $d \geq 3$ and assume that $\oV$ has the embedding property and satisfies \eqref{eq: inv P}, \eqref{eq: inv N} and \eqref{eq: inv mult lip}.  If $(A,V) \in \cA\cP_p(\Omega,\oV)$, then for every $\varepsilon >0$ the semigroup $\left(e^{-\varepsilon t} T_t\right)_{t >0}$ generated by $-\oL-\varepsilon$ extrapolates to a strongly continuous semigroup on $L^r(\Omega)$ provided that
$$
|1/2-1/r| \leq 1/d +( 1-2/d)|1/2-1/p|.
$$
This semigroup is bounded holomorphic of angle $\pi/2-\theta_0$. If $\oV$ has the homogeneous embedding property, then the same result also holds for $\varepsilon=0$.
\end{corollary}
Although this result is not central to the main contributions of the present paper, we included it for completeness; its proof essentially follows the method of \cite{Egert20}.

\begin{remark}
\begin{itemize}
\item
The results stated above remain valid under the more general assumption that $V$ satisfies
$$
\int_\Omega V_- |v|^2 \leq \alpha \int_\Omega |\nabla v|^2 + \sigma \int_\Omega V_+ |v|^2 + c(\alpha,\sigma) \int_\Omega |v|^2, \qquad v \in \oV,
$$
for some $\alpha \ge 0$, $\sigma \in [0,1)$, and $c(\alpha,\sigma) \in \R$. In this case, all assertions of Theorem~\ref{t : contract}, Corollary~\ref{c: N analytic sem} and Corollary~\ref{t: eg extr} hold for the operator $\oL+c(\alpha,\sigma)$ instead of $\oL$.
For convenience, we shall always assume $c(\alpha,\sigma)=0$.
\item 
When $A = I_d$, the condition $(I_d, V) \in \widetilde{\cA\cP}_p(\Omega,\oV)$ (resp. $(I_d, V) \in {\cA\cP}_p(\Omega,\oV)$) corresponds to $p$ lying in the  interval $[p_-, p_+]$ (resp. $(p_-, p_+)$), where $p_{\mp} = 2/(1 \pm \sqrt{1 - \alpha(V,\oV)})$. Therefore, Theorem~\ref{t : contract} and Corollary~\ref{c: N analytic sem} generalize \cite[Theorem~1]{BS90} and \cite[Theorem~6]{Per}. 
Similarly, Corollary~\ref{c: N analytic sem} and Corollary~\ref{t: eg extr} should be compared with \cite[Proposition~3.3 \& Theorem~3.4]{AO}, 
where analogous results are obtained for the Schr\"odinger operator $-\Delta + V$ on non-compact complete Riemannian manifolds of homogeneous type.
\end{itemize}
\end{remark}

\subsection{Bilinear embeddings for nonegative potentials}
In case when the potentials are assumed to be nonnegative, in \cite[Theorem 1.4]{CD-Potentials} Carbonaro and Dragi\v{c}evi\'c proved that there exists $C>0$  independent of the dimension $d$ such that
\begin{equation}\label{eq: N bil no lower terms}
\aligned
 \int^{\infty}_{0}\!\int_{\Omega}\sqrt{\mod{\nabla T^{A,V,\oV}_{t}f}^2 + V \mod{T^{A,V,\oV}_{t}f}^2}\sqrt{\mod{\nabla T^{B,W,\oW}_{t}g}^2 +  W \mod{T^{B,W,\oW}_{t}g}^2} & \leq C \norm{f}{p}\norm{g}{q}, 
\endaligned
\end{equation}
for all $A,B \in \cA_p(\Omega)$, $V,W \in L_\text{loc}^1 (\Omega, \R_+)$ and all $f,g \in (L^p \cap L^q)(\Omega)$, where $\oV$ and $\oW$ are two closed subspaces of $W^{1,2}(\Omega)$ satisfying \AssBE \, and 
 $q= p/(p-1)$ is the conjugate exponent of $p$.

Given $V \in \cP(\Omega,\oV)$ and $W \in \cP(\Omega,\oW)$, we extend the bilinear embedding in \eqref{eq: N bil no lower terms} to the semigroups  $(T^{A,V,\oV}_{t})_{t>0}$ and  $(T^{B,W,\oW}_{t})_{t>0}$. 
In accordance with \cites{CD-DivForm, CD-Mixed, CD-Potentials,Poggio}, we need a stronger condition than the one which implies the $L^p$ contractivity of such semigroups. 
In Sections~\ref{s: proof b.e. neg bound} and \ref{s: unb neg pot} we shall prove the following result.

\begin{theorem}\label{t: N bil}
 Suppose that $(\oV,\oW)$ satisfies \AssBE. Choose $p>1$. Let $q$ be its conjugate exponent, i.e., $1/p+1/q=1$.  Assume that $(A,V) \in \cA\cP_p(\Omega,\oV)$ and $(B,W) \in \cA\cP_p(\Omega,\oW)$. 

There exists $C>0$ such that for any $f,g\in (L^{p}\cap L^{q})(\Omega)$ we have
\begin{equation}
\label{eq: N bil}
\aligned
 \displaystyle\int^{\infty}_{0}\!\int_{\Omega}\sqrt{\mod{\nabla T^{A,V,\oV}_{t}f}^2 +\Bigl| V\Bigr| \mod{T^{A,V,\oV}_{t}f}^2}\sqrt{\mod{\nabla T^{B,W,\oW}_{t}g}^2 +  \Bigl|W\Bigr| \mod{T^{B,W,\oW}_{t}g}^2} \leq C \norm{f}{p}\norm{g}{q}.
\endaligned
\end{equation}
We may choose $C>0$ which depends on $p, A, B, \alpha(V,\oV), \alpha(W,\oW)$, but not on the dimension $d$.
\end{theorem}

This result incorporates several earlier theorems as special cases, including:
\begin{itemize}
\item $V=W$ nonnegative, $\Omega = \R^d$, $A, B$ equal and real \cite[Theorem 1]{Dv-kato}
\item $V=W=0$, $\Omega = \R^d$ \cite[Theorem 1.1]{CD-DivForm}
\item $V=W=0$ \cite[Theorem 2]{CD-Mixed}
\item $V, W$ nonnegative \cite[Theorem 1.4]{CD-Potentials}.
\end{itemize}

Recently, bilinear inequalities of this type have been also proven for perturbed divergence-form operators \cite[Theorem~3]{Poggio} and divergence-form operators subject to dynamical boundary conditions \cite[Theorem~1.4]{BER}.

\subsection{Maximal regularity and functional calculus}  
In case when $V=0$, let $A \in \cA_p(\Omega)$ and let $-\oL_p^A$ denote the generator of $(T^{A,\oV}_{t})_{t>0}$ on $L^{p}(\Omega)$. Then $\oL_p^A$ admits a bounded holomorphic functional calculus of angle $\theta<\pi/2$ and has parabolic maximal regularity \cite[Theorem 3]{CD-Mixed}. 

Following the same argument of \cite[Theorem 3]{CD-Mixed}, by means of 
\begin{itemize}
\item elementary properties of the classes $\cA\cP_p(\Omega)$ (see Proposition~\ref{p : basic prop}\ref{i : inv small rot},\ref{i : inv adj})
\item a well-known sufficient condition for bounded holomorphic functional calculus \cite[Theorem~4.6 and Example~4.8]{CDMY}
\item the Dore-Venni theorem \cite{DoreVenni,PrussSohr}
\item  Theorem~\ref{t: N bil} applied with $B=A^{*}$, $W=V$ and $\oW=\oV$
\end{itemize} 
we can deduce the following result; see Section~\ref{s: max funct} for the explanation of terminology and the proof.
\begin{theorem}\label{t: N principal}
Suppose that $\oV$ falls into any of the special cases \ref{i: D}-\ref{i: cM} of Section~\ref{s: boundary}. Assume that $p>1$ and $(A,V) \in \cA\cP_{p}(\Omega,\oV)$. Let $-\oL_p$ be the generator of $(T_{t})_{t>0}$ on $L^{p}(\Omega)$. Then $\oL_p$ admits a bounded holomorphic functional calculus of angle $\theta <\pi/2$. As a consequence, $\oL_{p}$ has parabolic maximal regularity.
\end{theorem}

After its introduction by Carbonaro and Dragi\v{c}evi\'c, the technique of deriving a bounded $H^\infty$-functional calculus from this type of bilinear embedding has also appeared in subsequent works, such as \cite{Poggio} and \cite{BER}.

Recent results regarding the holomorphic functional calculus for the operator $\oL_p$ have been obtained by Egert \cite{Egert} and Bechtel \cite{Bechtel}. In \cite{Egert} the author considered elliptic systems of second order in divergence-form with bounded and complex coefficients and subject to mixed boundary conditions on bounded and connected open sets $\Omega$ whose boundary is Lipschitz regular around the Neumann part $\overline{\partial\Omega \setminus D}$. In \cite[Theorem 1.3]{Egert} he provided the optimal interval of $p$'s for the bounded $H^\infty$-calculus on $L^p$. More recently, Bechtel \cite[Proposition~3.6]{Bechtel} improved the aforementioned \cite[Theorem~1.3]{Egert}   by only assuming that $\Omega$ is open and locally uniform near $\partial\Omega \setminus D$; see \cite[Section~2.1]{Bechtel} for the definition. In both cases, the bounded $H^\infty$-calculus of $\oL_p$ in $L^p$ was exploited to establish $L^p$-estimates for the square root of $\oL_p$ \cite[Theorem~1.2 and Theorem~1.4]{Egert} and \cite[Theorem~1.2]{Bechtel}. 

In generality that we consider, the domain $\Omega$ may be completely irregular and/or unbounded. Therefore, as explained in \cite[Section~1.7]{Poggio} and \cite[Section~1.5]{CD-Mixed}, we only deduce that our interval of $p$'s for the bounded $H^\infty$-calculus on $L^p$ is contained in those obtained by Egert and Betchel.

\subsection{Notation}
Given two quantities $X$ and $Y$, we adopt the convention whereby $X \leqsim Y$ means that there exists an absolute constant $C>0$ such that $X \leq C Y$. If both $X \leqsim Y$ and $Y \leqsim X$, then we write $X \sim Y$. If  $\{\alpha_1, \dots, \alpha_n\}$ is a set of parameters, then $C(\alpha_1, \dots, \alpha_n)$ denotes a constant depending only on $\alpha_1,\dots,\alpha_n$. When $X \leq C(\alpha_1, \dots, \alpha_n) Y$, we will often write $X \leqsim_{\alpha_1, \dots, \alpha_n} Y$.

If $z=(z_1, \dots, z_d) \in \C^d$ and $w$ is likewise, we write
$$
\sk{z}{w}_{\C^d} = \sum_{j =1}^d z_j \overline{w}_j
$$
and $|z|^2 = \sk{z}{z}_{\C^d}$. When the dimension is obvious, we sometimes omit the index $\C^d$ and only write $\sk{z}{w}$. When both $z$ and $w$ belong to $\R^d$ , we sometimes emphasize this by writing $\sk{z}{w}_{\R^d}$. This should not be confused with the standard pairing
$$
\sk{\varphi}{\psi} = \int_{\Omega} \varphi \overline{\psi},
$$
where $\varphi$, $\psi$ are complex functions on $\Omega$ such that the above integral makes sense. All the integrals in this paper are with respect to the Lebesgue measure.

Unless stated otherwise, for every $r \in [1,\infty]$ we denote by $r^\prime$ its conjugate exponent, i.e., $1/r + 1/r^\prime = 1$. 
When working with a fixed exponent $p$, we set $q$ to be its conjugate exponent, so as to simplify the notation in the definition and subsequent use of the associated Bellman function introduced in \eqref{eq: N Bellman}.

For $p,r \in [1,\infty]$, $\|\cdot\|_{p - r}$ denotes the operator norm from $L^p$ to $L^r$.

\subsection{Organization of the paper}
Here is the summary of each section.
\begin{itemize}
\item In Section~\ref{s : dom form} we illustrate invariance properties of the spaces described in Section~\ref{s: boundary} and we explain the necessity of \AssBE \, for the validity of \cite[Lemma~19]{CD-Mixed}.
\item In Section~\ref{s: Neumann gen conv} we summarize some of the main notions needed in the paper and we describe the heat-flow method that we will use to prove the bilinear embedding.
\item In Section~\ref{s : Lp contr} we prove the results on contractivity and analyticity of semigroups announced in Section \ref{s: sem emb}.
\item In Section~\ref{s: CR} we prove a chain rule in order to apply the heat-flow method. 
\item In Section~\ref{s: conv bell}, we establish a stronger convexity property of the Bellman function $\cQ$ than that obtained in \cite[Theorem~5.2]{CD-DivForm} and \cite[Theorem~3.1]{CD-Potentials}, under the new (and stronger) condition \eqref{eq: new cond}.
\item In Section~\ref{s: proof b.e. neg bound} we prove the bilinear embedding for potentials with bounded negative part.
\item In Section~\ref{s: unb neg pot}  we prove the bilinear embedding in the general case.
\item In Section~\ref{s: max funct} we prove Theorem \ref{t: N principal}.
\item In Section~\ref{s: ex ss pot} we provide some examples of strongly subcritical potentials.
\end{itemize}

\section{The domain of the form}\label{s : dom form}
In this section, we present some invariant properties of the closed subspaces $\oV$ of $W^{1,2}(\Omega)$ described in Section~\ref{s: boundary}. We also explain the necessity of \AssBE \, for the validity of \cite[Lemma~19]{CD-Mixed}.

Denote by $\cL$ the set of all Lipischitz functions $\Phi: \C \rightarrow \C$ with $\Phi(0)=0$. If the Lipschitz constant of  $\Phi$ is $1$, $\Phi$ is said to be a {\it normal contraction}. We denote by $\oN$ the set of all normal contractions.
\begin{proposition}
\label{p: invariance under L}
Let $\oV$ be a closed subspace of $W^{1,2}(\Omega)$ containing $W_0^{1,2}(\Omega)$. Then $\oV$ satisfies \eqref{eq: inv P} and \eqref{eq: inv N} if and only if it is invariant under the (whole) class $\cL$.
\end{proposition}
\begin{proof}
In view of \cite[Proposition~A.1.]{CD-Potentials} and the fact that $\oN \subseteq \cL$, it suffices to prove that $\oV$ is invariant under $\cL$ if it is invariant under $\oN$.

Let $\Phi \in \cL$ and $u \in \oV$. Then $\Phi/{\rm Lip}(\Phi) \in \cN$. Hence, the invariance of $\oV$ under $\oN$ gives 
\begin{alignat*}{2}
\Phi(u) = {\rm Lip}(\Phi) \cdot \frac{\Phi}{{\rm Lip}(\Phi)}(u) \in \oV.
\tag*{\qedhere}
\end{alignat*}
\end{proof}

We would like to obtain an analogous invariant result in the multivariable setting. We proceed much as in \cite[Lemma~4]{Egert20}. See also \cite[Lemma~19]{CD-Mixed}.
\begin{lemma}
\label{l: sbsq V}
Let $\oV$ be a closed subspace of $W^{1,2}(\Omega)$ containing $W_0^{1,2}(\Omega)$. Let $(u_n)_{n\in \N} \subseteq \oV$ and $u \in L^2(\Omega)$ such that
\begin{itemize}
\item $(u_n)_{n\in\N}$ is bounded in $\oV$,
\item $u_n \rightarrow u$ in $L^2(\Omega)$ as $n\rightarrow \infty$.
\end{itemize}
Then $u \in \oV$.
\end{lemma}
\begin{proof}
Since $(\Phi_n(u))_{n\in\N}$ is bounded in $\oV$, it admits a subsequence with weak limit $u_\infty \in \oV$. Then $u = u_\infty \in \oV$, as $u$ is the strong limit of  $(u_n)_{n\in\N}$ in $L^2(\Omega)$.
\end{proof}

\begin{proposition}\label{p: 2var Lip}
Suppose that $\oV$ falls into any of the special cases \ref{i: D}-\ref{i: bM} and $\oW=W^{1,2}(\Omega)$. Let $u \in \oV$, $v \in \oW$ and $\Phi : \C^2 \rightarrow \C$ be a Lipschitz function such that $\Phi(0,\eta)=0$ for all $\eta \in \C$. Then $\Phi \circ (u,v) \in \oV$.

Furthermore, if $\oV$ is of the type described in \ref{i: cM} the same conclusion holds provided that $\oW$ is of the type described either in \ref{i: D} or in \ref{i: cM}.
\end{proposition}
\begin{proof}
First case:  \framebox{$\oV$  falls into any of the special cases \ref{i: D}-\ref{i: bM}}.

If $\oV=W^{1,2}(\Omega)$, the assertion follows from \cite[Corollary~2.7]{ARKR}. Moreover, we have
\begin{equation}
\label{eq: stima lip}
\| \Phi \circ (u,v) \|_{1,2} \leqsim {\rm Lip}(\Phi)\left( \|u\|_{1,2} +\|v\|_{1,2}\right).
\end{equation}

Suppose that $\oV= \widetilde{W_D}^{1,2}(\Omega)$. Let $(u_n)_{n\in \N}$ be a sequence in $W^{1,2}(\Omega)$ with \\$ {\rm dist}({\rm supp} \, u_n, D)>0$ converging to $u$ in $W^{1,2}(\Omega)$.
We set $w_n := \Phi \circ (u_n,v)$. From the previous case $w_n \in W^{1,2}(\Omega)$ and thanks to $\Phi(0,\cdot)=0$ we have ${\rm dist}({\rm supp}\, w_n, D)>0$. Thus $w_n \in  \widetilde{W_D}^{1,2}(\Omega)$. Estimate \eqref{eq: stima lip} shows that $(w_n)_n$ is bounded in $ \widetilde{W_D}^{1,2}(\Omega)$. 
On the other hand, we have  $\| w_n - \Phi \circ (u,v) \|_2 \leq {\rm Lip}(\Phi) \| u_n - u\|_2$, so that $w_n \rightarrow \Phi \circ (u,v)$ strongly in $L^2(\Omega)$ as $n \rightarrow \infty$. Thus, Lemma~\ref{l: sbsq V} gives $\Phi \circ (u,v) \in  \widetilde{W_D}^{1,2}(\Omega)$ as required.

When $\oV= W_0^{1,2}(\Omega)$ the proof is similar to the previous one. 

Second case: \\ \framebox{$\oV$ is of the type described in \ref{i: cM} and $\oW$ of the type described in \ref{i: D} and \ref{i: cM}.}

Suppose that $\oV = W_{D}^{1,2}(\Omega)$ and $\oW= W_{D^\prime}^{1,2}(\Omega)$ with $D, D^\prime$ being (possibly empty)  closed subsets of $\partial \Omega$.  Let $(u_n)_{n\in \N} \subseteq C^\infty_c(\R^d \setminus D) $ and $(v_n)_{n\in \N} \subseteq C^\infty_c(\R^d \setminus D^\prime)$ such that ${u_n}_{\vert_{\Omega}}$ and ${v_n}_{\vert_{\Omega}}$ converge to $u$ and $v$ in $W^{1,2}(\Omega)$ as $n \rightarrow \infty$, respectively. We set $w_n := \Phi \circ (u_n,v_n)$. Then $w_n$ is Lipschitz continuous on $\R^d$ and has compact support vanishing in a neighborhood of $D$ since $\Phi(0,\cdot)=0$. We conclude $w_n \in W_D^{1,2}(\Omega)$ because the required approximations in $C_c^\infty(\R^d \setminus D)$ can explicitly be constructed by convolution with smooth, compactly supported kernels. At this point, we repeat the same final argument of the first case where we made use of \eqref{eq: stima lip} and Lemma~\ref{l: sbsq V}. 
\end{proof}
\smallskip

Next proposition illustrates the necessity of imposing additional assumptions on $\oW$ in the case when $\oV=W_{D}^{1,2}(\Omega)$. In fact, in general $\Phi \circ (u,v) \notin W_D^{1,2}(\Omega)$ whenever $u \in W_D^{1,2}(\Omega)$ and $v \in \widetilde{W_{D^\prime}}^{1,2}(\Omega)$.
\begin{proposition}
\label{p : lemma19 wrong}
Let $\Omega=\{(x,y) \in \R^2 : 0 < |x| <1, 0<y<1\}$ and $\Phi: \C^2 \rightarrow \C$ be a Lipschitz function such that $\Phi(0,\cdot)=0$. Suppose that there exist $\zeta_0, \eta_0 \in \C$ such that
\begin{equation}
\label{eq: PHI- e PHI+}
\Phi(\zeta_0,\eta_0) \ne \Phi(\zeta_0, 0).
\end{equation}
Then there exist $D, D^\prime \subseteq \partial \Omega$ (possibly empty) closed, $u \in W_D^{1,2}(\Omega)$ and $v \in \widetilde{W_{D^\prime}}^{1,2}(\Omega)$ such that $\Phi \circ (u,v) \notin W_D^{1,2}(\Omega)$.
\end{proposition}
\begin{proof}
Take $D=\emptyset$ and $D^\prime =\{ (-1,y) : 0 \leq y \leq1\}$ and define
$$
u:=\zeta_0, \qquad  \qquad v:=
\begin{cases}
\eta_0, &{\rm if}\,\, x>0,\\
0, &{\rm if}\,\, x<0.
\end{cases}
$$
By construction $u \in W_\emptyset^{1,2}(\Omega)$ and $v \in \widetilde{W_{D^\prime}}^{1,2}(\Omega)$ and
$$
\Phi \circ (u,v) =
\begin{cases}
\Phi(\zeta_0,\eta_0), &{\rm if}\,\, x>0,\\
\Phi(\zeta_0, 0), &{\rm if}\,\, x<0.
\end{cases}
$$
Therefore, from \eqref{eq: PHI- e PHI+} we deduce that  for sufficiently small $\widetilde{\varepsilon} >0$ no function $\phi \in C^\infty_c(\R^2)$ can satisfy $\|\Phi\circ(u,v) - \phi_{\vert_{\Omega}} \|_{1,2} <\widetilde{\varepsilon}$. Hence, $\Phi(u,v) \notin  W_\emptyset^{1,2}(\Omega)$.
\end{proof}
In order to provide a counterexample with $D \ne \emptyset$, we may keep  $D^\prime$ and $v$ as before and take $D=D^\prime$ and $u=(\zeta_0\psi \otimes \mathbf{1})_{\vert_{\Omega}}$, where $\psi$ is a smooth function on $[-1,1]$ such that $\psi=0$ on $[-1,-2/3]$, $\psi=1$ on $[-1/3,1]$ and $0\leq \psi \leq 1$ otherwise. Set $\Omega^\prime = \{(x,y) \in \Omega : x > -1/3\}$. Then, by the same previous argument, we infer that for sufficiently small $\widetilde{\varepsilon} >0$ no function $\phi \in C^\infty_c(\R^2)$ can satisfy $\|\Phi \circ(u,v) - \phi_{\vert_{\Omega^\prime}} \|_{W^{1,2}(\Omega^\prime)} <\widetilde{\varepsilon}$. Hence, no function  $\phi \in C^\infty_c(\R^2 \setminus D)$ can satisfy $\|\Phi \circ(u,v) - \phi_{\vert_{\Omega}} \|_{W^{1,2}(\Omega)} <\widetilde{\varepsilon}$.

\begin{remark}
\label{r : lemma19 wrong}
Proposition~\ref{p : lemma19 wrong} shows the necessity of \AssBE \, for the validity of \cite[Lemma~19]{CD-Mixed}. See Section~\ref{s: seq Rnnu} for the definition of the sequence $( \cR_{n,\nu})_{n,\nu}$.

Let $p>2$ and denote by $q$ its conjugate exponent. From \cite[Theorem~16(i),(v)]{CD-Mixed}  the function $\partial_\zeta \cR_{n,\nu}$ is Lipschitz continuous on $\C^2$ and $\partial_\zeta\cR_{n,\nu}(0,\cdot)=0$, for all $n \in \N$ and $\nu\in(0,1)$. Moreover, a trivial computation shows that
$$
\aligned
\partial_\zeta \cQ(1, 2^{1/q}) &= \frac{p}{2} +\delta \cdot 2^{(2-q)/2} >\frac{p}{2} +\delta  = \partial_\zeta \cQ(1, 0), \\
\partial_\zeta \cP_n(1, 2^{1/q}) &= \frac{p+\varepsilon}{2}n^{-\varepsilon}\left[K_{p+\varepsilon} + (1+2^{2/q})^{(p+\varepsilon-2)/2} \right]\\
 &>  \frac{p+\varepsilon}{2}n^{-\varepsilon}\left[K_{p+\varepsilon} +1 \right] = \partial_\zeta\cP_n(1,0),
\endaligned
$$
for all $n \geq (1+2^{2/q})^{1/2}$. Therefore, since both $\cQ$ and $\cP_n$ are continuous on $\C^2$, for all $n \geq (1+2^{2/q})^{1/2}$ there exists $\nu_0(n)$ such that for any $\nu \in (0, \nu_0(n))$ we have
\begin{equation*}
\aligned
\partial_\zeta \cR_{n,\nu}(1, 2^{1/q}) &= \partial_\zeta(\cQ * \varphi_\nu)(1, 2^{1/q}) +C_1 \nu^{q-2} \partial_\zeta(\cP_n * \varphi_\nu)(1, 2^{1/q})\\
&>  \partial_\zeta(\cQ * \varphi_\nu)(1, 0) +C_1 \nu^{q-2} \partial_\zeta(\cP_n * \varphi_\nu)(1, 0) = \partial_\zeta \cR_{n,\nu}(1, 0).
\endaligned
\end{equation*}
Hence, Proposition~\ref{p : lemma19 wrong} implies that there exist $D,D^\prime \subseteq \partial\Omega$ closed, $u \in W_D^{1,2}(\Omega)$ and $v \in \widetilde{W_{D^\prime}}^{1,2}(\Omega)$ such that $\partial_\zeta \cR_{n,\nu}(u,v) \notin W_D^{1,2}(\Omega)$ for all  $n \geq (1+2^{2/q})^{1/2}$ and  $\nu \in (0, \nu_0(n))$.
\end{remark}

\section{Heat-flow monotonicity and generalized convexity}\label{s: Neumann gen conv} 
\subsection{Real form of complex operators}  We explicitly identify $\C^{d}$ with $\R^{2d}$ as follows. 
For each $d\in\N_{+}$ consider the operator
$\cV_{d}:\C^{d}\rightarrow\R^{d}\times\R^{d}$, defined by 
$$
\cV_{d}(\xi_{1}+i\xi_{2})=
(\xi_{1},\xi_{2}),\quad \xi_{1},\xi_{2}\in\R^{d}.
$$
Let $k, d \in \N_+$. We define another identification operator
$$
\cW_{k,d}:\underbrace{\C^{d}\times\cdots\times\C^{d}}_{k-{\rm times}}\longrightarrow \underbrace{\R^{2d}\times\cdots\times\R^{2d}}_{k-{\rm times}},
$$
 by the rule
$$
\cW_{k,d}(\xi^{1},\dots, \xi^{k})
=\left(\cV_{d}(\xi^{1}),\dots, \cV_{d}(\xi^{k})\right),\quad \xi^{j}\in\C^{d}, \, j=1,\dots,k.
$$
When $k=2$, we set $\cW_d =\cW_{2,d}$.

If $A\in\C^{d\times d}$ we shall frequently use its real form:
$$
\cM(A)=\cV_{d}A\cV_{d}^{-1}=\left[
\begin{array}{rr}
\Re A  & -\Im A\\
\Im A  & \Re A
\end{array}
\right]\,.
$$

\subsection{Convexity with respect to complex matrices}\label{s: GeH}
Let $d, k \in \N_+$ and let $\Phi:\C^{k} \rightarrow \R$ be of class $C^{2}$. 
We associate the function $\Phi$ on $\C^k$ with the following function on $\R^{2k}$:
\begin{equation}
\label{e : realis of compl fun}
\Phi_{\cW} := \Phi \circ \cW_{k,1}^{-1}.
\end{equation}

Choose and, respectively, denote
$$
A_1, \dots, A_k \in \C^{d \times d} \qquad \mathbf{A}:=(A_1,\dots, A_k).
$$
Let $\omega \in \C^k$ and $\Xi \in \C^{kd}$. Denote, respectively, by $D^{2}\Phi(\omega)$ and $D \Phi (\omega)$ the Hessian matrix and the gradient of the function $\Phi_{\cW}=\Phi \circ \cW_{k,1}^{-1} : \R^{2k} \rightarrow \R$ calculated at the point $\cW_{k,1}(\omega) \in \R^{2k}$. 
In accordance with \cite{CD-DivForm,CD-Mixed} we define the {\it generalized Hessian form of $\Phi$ with respect to $\mathbf{A}$} as
$$
H^{\mathbf{A}}_{\Phi}[\omega; \Xi]=
 \sk{\left(D^{2}\Phi(\omega)\otimes I_{\R^{d}}\right)\cW_{k,d}(\Xi)}{\left(\cM(A_1)\oplus \cdots \oplus \cM(A_k)\right)\cW_{k,d}(\Xi)}_{\R^{2kd}},
$$
where $\cM(A_1)\oplus \cdots \oplus \cM(A_k)$ is the $2kd \times 2kd$ block diagonal real matrix with the $2d\times2d$ blocks $ \cM(A_1) ,\dots, \cM(A_k)$ along the main diagonal and  $\otimes$ denotes the Kronecker product of matrices (see, for example, \cite{CD-DivForm}).

\begin{defi}{\cite{CD-DivForm, CD-Mixed}}
We say that  $\Phi$ is $\mathbf{A}$-{\it convex} in $\C^{k}$ if $H^{\mathbf{A}}_{\Phi}[\omega; \Xi]$ is nonnegative for all $\omega\in \C^{k}$, $\Xi \in \C^{kd}$. 
\end{defi}

We maintain the same notation when instead of matrices we consider matrix-valued {\it functions} $A_1, \dots, A_k \in L^{\infty}(\Omega;\C^{d\times d})$; in this case however we require that all the conditions are satisfied for a.e. $x\in\Omega$. 

\subsection{The Bellman function of Nazarov and Treil}
We want to study the monotonicity of the flow
\begin{equation*}\label{eq: N flow}
\cE(t)=\int_{\Omega}\cQ(T^{A,V,\oV}_{t}f,T^{B,W,\oW}_{t}g)
\end{equation*}
associated with a particular explicit {\it Bellman function} $\cQ$ invented by Nazarov and Treil \cite{NT}. Here we use a simplified variant introduced in \cite{Dv-kato} which comprises only two variables:
\begin{equation}\label{eq: N Bellman}
\cQ(\zeta,\eta)=
|\zeta|^p+|\eta|^{q}+\delta
\begin{cases}
 |\zeta|^2|\eta|^{2-q},& |\zeta|^p\leqslant |\eta|^q;\\
 (2/p)\,|\zeta|^{p}+\left(
 2/q-1\right)|\eta|^{q},&|\zeta|^p\geqslant |\eta|^q\,,
\end{cases}
\end{equation}
where $p\geq2$, $q=p/(p-1)$, $\zeta,\eta\in\C$ and $\delta>0$ is a positive parameter that will be fixed later. It was noted in \cite[p. 3195]{CD-DivForm} that $\cQ\in C^{1}(\C^{2})\cap C^{2}(\C^{2}\setminus\Upsilon)$, where
$$
\Upsilon=\{\eta=0\}\cup\{|\zeta|^p=|\eta|^q\}\,,
$$
and that for $(\zeta,\eta)\in\C\times\C$ we have
\begin{equation}
\label{eq: N 5}
\aligned
0\leqslant \cQ(\zeta,\eta) & \leqsim_{p,\delta}\,\left(|\zeta|^p+|\eta|^q\right), \\
|(\partial_{\zeta}\cQ)(\zeta,\eta)| & \leqsim_{p,\delta}\, \max\{|\zeta|^{p-1},|\eta|\},\\
|(\partial_{\eta}\cQ)(\zeta,\eta)| & \leqsim_{p,\delta}\, |\eta|^{q-1}\,,
\endaligned
\end{equation}
where $\partial_\zeta=\left(\partial_{\zeta_1}-i\partial_{\zeta_2}\right)/2$ and $\partial_\eta=\left(\partial_{\eta_1}-i\partial_{\eta_2}\right)/2$.

In \cite{CD-DivForm} the authors established the $(A,B)$-convexity of the Bellman function $\cQ$ under the assumption that the matrices $A$ and $B$ are $p$-elliptic. We present  the result as stated in \cite[Theorem~3.1]{CD-Potentials} which also includes a lower bound for the first-order derivatives of $\cQ$.
\begin{theorem}{\cite[Theorem~3.1]{CD-Potentials}}
\label{t: thm CDPot}
Choose $p \geq 2$ and $A,B \in \cA_p(\Omega)$. Then there exist a continuous function $\tau : \C^2 \rightarrow [0,+\infty)$ such that $\tau^{-1}=1/\tau$ is locally integrable on $\C^2 \setminus \{(0,0)\}$, and $\delta \in (0,1)$ such that $\cQ=\cQ_{p,\delta}$ as in \eqref{eq: N Bellman} admits the following properties:
\begin{enumerate}[label=\textnormal{(i)}]
\item
for any $\omega =(\zeta,\eta) \in \C^2 \setminus \Upsilon$, $X,Y \in \C^d$, and  a.e. $x \in \Omega$, we have
$$
H_{\cQ}^{(A(x),B(x))}[\omega, (X,Y)] \geqsim \tau |X|^2 + \tau^{-1}|Y|^2;
$$ 
\item for any $\omega =(\zeta,\eta) \in \C^2$, we have
$$
(\partial_\zeta\cQ)(\zeta,\eta) \cdot \zeta \geqsim \tau |\zeta|^2 \quad {\rm and} \quad (\partial_\eta\cQ)(\zeta,\eta) \cdot \eta \geqsim \tau^{-1} |\eta|^2.
$$
\end{enumerate}
The implied constants depend on $p,A,B$, but not on the dimension $d$.

We may take $\tau(\zeta,\eta)= \max\{|\zeta|^{p-2},|\eta|^{2-q}\}$.
\end{theorem}

\begin{remark}
\label{r : CD delta}
A careful examination of the proof of \cite[Theorem~3.1]{CD-Potentials} shows that the conclusion of the previous theorem holds for all parameters smaller than the specific $\delta$ given above, provided that $\tau$ is chosen accordingly, depending on $\delta$. More precisely, there exists $\delta_0 \in (0,1)$ such that for all $\delta \in (0,\delta_0)$ there exist $C=C(\delta) >0$ and $\tau = \tau_\delta : \C^2 \setminus \Upsilon \rightarrow (0, +\infty)$ such that, a.e. $x \in \Omega$,
$$
H_{\cQ}^{(A(x),B(x))}[\omega, (X,Y)] \geq 2 \delta C(\delta) \left( \tau |X|^2 + \tau^{-1}|Y|^2\right),
$$
for all $\omega \in \C^2 \setminus \Upsilon$ and $X,Y \in \C^d$.

Here $\tau$ may be chosen as
\begin{equation}
\label{e: def tauuu}
\tau(\zeta,\eta) = \rho(\zeta,\eta) :=
\begin{cases}
(p-1)|\zeta|^{p-2}, &|\zeta|^p \geq|\eta|^q >0,\\
D |\eta|^{2-q}, &|\zeta|^p <|\eta|^q,
\end{cases}
\end{equation}
where $D=D(\delta)$  is a positive constant depending on $\delta, p, \lambda(A), B$ such that
\begin{equation}
\label{eq: cons delta 0}
 C(\delta)  D^{-1}(\delta) \rightarrow +\infty, \quad {\rm as } \,\, \delta \searrow 0.
\end{equation}
\end{remark}

\subsection{Heat-flow monotonicity}\label{s: hfmon}
We describe now the method we will apply to prove the dimension-free bilinear embedding. The idea is studying the monotonicity of certain functionals associated with semigroups, exploiting the convexity with respect to complex matrices of specific functions \cite{CD-DivForm,CD-mult,CD-OU, CD-Mixed, CD-Potentials, Poggio, BER}. The main passages will be presented at a formal level in what follows. Their rigorous justification is beyond the scope of this exposition and will be provided later.
\medskip

Let $\Omega \subseteq \R^d$, $A,B \in \cA(\Omega)$, $\oV,\oW$ of the type described in Sect.~\ref{s: boundary} and $V \in \cP(\Omega,\oV)$ and $W\in \cP(\Omega,\oW)$. Denote
\begin{equation}
\label{eq: beta}
\sigma := \max\{ \sigma(V,\oV), \sigma(W,\oW)\}.
\end{equation}

Let $\Phi : \C^2 \rightarrow \R_+$ be of class $C^1$. Given $f,g \in L^2(\Omega)$, define the function
$$
\cE(t)=\int_{\Omega}\Phi\left(T^{A,V,\oV}_{t}f, T^{B,W,\oW}_{t}g\right),\quad t>0.
$$
\begin{enumerate}
    \item Suppose that we can differentiate and interchange derivative and integral. Then a calculation (see \cite{CD-DivForm}) shows that
    \begin{equation*}
    \aligned
        -\cE^\prime(t) = 2\, \Re\int_\Omega& \biggl[(\partial_\zeta \Phi)\left(T^{A,V,\oV}_{t}f, T^{B,W,\oW}_{t}g\right) \oL^{A,V} T_t^{A,V,\oV}f \\
        &+ (\partial_\eta \Phi)\left(T^{A,V,\oV}_{t}f, T^{B,W,\oW}_{t}g\right) \oL^{B,W} T_t^{B,W,\oW}g\biggr].
        \endaligned
    \end{equation*}
    \item 
Set
$$
(u,v)=\left(T_t^{A,V,\oV}f, T_t^{B,W,\oW}g\right).
$$
Suppose that we can split the operators $\oL^{A,V}, \oL^{B,W}$ as
    $$
    \aligned
    \oL^{A,V} &= \oL^{A,0}+ V_+ -V_-,\\
    \oL^{B,W} &= \oL^{B,0}+ W_+ -W_-.
    \endaligned
    $$
    Then
    \begin{equation}
	\label{eq: E tre}
     -\cE^\prime(t) = I_1+I_2-I_3,
    \end{equation}
    where
    $$
    \aligned
    I_1 &=  2\Re \int_\Omega (\partial_\zeta \Phi)(u,v) \oL^{A,0} u + (\partial_\eta \Phi)(u,v) \oL^{B,0} v , \\
    I_2 &= 2\int_\Omega  V_+\Re \biggl[(\partial_\zeta \Phi)(u,v)  u  \biggr]+ W_+ \Re\biggl[(\partial_\eta \Phi)(u,v) v \biggr], \\
    I_3 &=2  \int_\Omega V_- \Re \biggl[(\partial_\zeta \Phi)(u,v) u \biggr] + W_- \Re \biggl[(\partial_\eta \Phi)(u,v) v \biggr]. \\
    \endaligned    
    $$
    \item 
    Suppose that $\Phi \in C^2$ and that $(\partial_\zeta\Phi)(u,v)$ and $(\partial_\eta \Phi)(u,v)$ belong to the form domain $\Dom(\gota_{A,V\oV})$ and $\Dom(\gota_{B,W\oW})$, respectively. Then we can integrate by parts in the sense of \eqref{eq: ibp} on $I_1$ and by means of another calculation (see \cite{CD-DivForm}), we get
    \begin{equation}
	\label{eq: I1 hess}
    I_1= \int_\Omega H_{\Phi}^{(A,B)}[(u,v); (\nabla u, \nabla v)].
    \end{equation}
    \item 
   Suppose that   there exist $k \in \N$,  $\{\varphi_j, \psi_j : \C^2 \rightarrow \C : j=1,\cdots,k\}$ and $\mu \in (\sigma,1]$   
such that  $\varphi_j(u,v) \in \oV$ and $\psi_j(u,v)  \in\oW$ for all $j  \in\{1,\cdots,k\}$  and
    \begin{equation}
\label{eq: prop phij psij}
    \aligned
    2\Re [(\partial_\zeta \Phi)(u,v) u] &\leq \sum_{j=1}^k|\varphi_j(u,v)|^2  \leq \frac{2}{\mu} \Re [(\partial_\zeta \Phi)(u,v) u],\\
     2\Re [(\partial_\eta \Phi)(u,v)  v] &\leq \sum_{j=1}^k|\psi_j(u,v)|^2 \leq \frac{2}{\mu} \Re [(\partial_\eta \Phi)(u,v)  v].
    \endaligned
    \end{equation}
    Then, by means of the subcritical inequality \eqref{e : subc ineq}, we get 
    $$
    I_3 \leq  \frac{\sigma}{\mu} I_2 +   \sum_{j=1}^k  \int_\Omega  \alpha(V,\oV) |\nabla [\varphi_j(u,v)]|^2 + \alpha(W,\oW)  |\nabla[\psi_j(u,v)]|^2.
    $$
\end{enumerate}
Hence, it follows from \eqref{eq: E tre} and \eqref{eq: I1 hess} that if 
\begin{equation}
\label{e : weak gen conv wrt pot}
H_{\Phi}^{(A,B)}[(u,v); (\nabla u, \nabla v)] \geq   \sum_{j=1}^k  \left(\alpha(V,\oV) |\nabla [\varphi_j(u,v)]|^2 + \alpha(W,\oW) |\nabla[\psi_j(u,v)]|^2\right),
\end{equation}
then the function $\cE$ is nonincreasing on $(0,+\infty)$.
Moreover, if a stronger inequality than \eqref{e : weak gen conv wrt pot} holds,  that is,
\begin{equation}
\label{e : weak gen conv wrt pot stronger}
\aligned
H_{\Phi}^{(A,B)}[(u,v); (\nabla u, \nabla v)] \geq&\,\, \tau(u,v)|\nabla u|^2+\tau^{-1}(u,v)|\nabla v|^2\\
&+    \sum_{j=1}^k  \left(\alpha(V,\oV) |\nabla [\varphi_j(u,v)]|^2 + \alpha(W,\oW) |\nabla[\psi_j(u,v)]|^2\right),
\endaligned
\end{equation}
for some positive function $\tau$ on $\C^2$, this formal method with $\Phi=\cQ$ can be used for proving bilinear inequalities in the spirit of \cite{CD-DivForm,CD-mult,CD-OU, CD-Mixed, CD-Potentials}, \cite{Poggio} and \cite{BER}.

Justifying item (i), and in particular item (iii), was the main goal of \cite{CD-Mixed}, on which we shall rely. As explained in \cite{CD-Potentials}, item (ii) follows easily when the potentials  $V$ and $W$ are bounded. Dealing with unbounded potentials requires greater care. In \cite{CD-Potentials} the decomposition in item (ii) was not established for unbounded nonnegative potentials;  instead,  the authors deduced the bilinear embedding for such potentials via a truncation argument. As we will show in Section~\ref{s: proof b.e. neg bound}, attempting to follow the same approach reveals that the negative part of the potentials can be truncated, whereas the positive part cannot. Therefore, in Section~\ref{ss: proof of 58} we will justify the decomposition in item (ii) for potentials with bounded negative part and possibly unbounded positive part, and subsequently adapt the truncation argument to remove the boundedness assumption on the negative part. A further novelty of this method concerns item (iv), namely, finding functions $\varphi_j$ and $\psi_j$ such that  
$$
\aligned
\varphi_j\left(T_t^{A,V,\oV}f, T_t^{B,W,\oW}g\right) &\in \oV, \\
\psi_j\left(T_t^{A,V,\oV}f, T_t^{B,W,\oW}g\right) &\in \oW,
\endaligned
$$
and satisfying \eqref{eq: prop phij psij} and \eqref{e : weak gen conv wrt pot stronger} for the Bellman function $\cQ$.

The candidate functions $\varphi_j$ and $\psi_j$
 are given by the following lemma.
\begin{lemma}
\label{l : first ord est Bell}
    Choose $p>2$ and $\mu \in (0,1)$. There exists  $\delta>0$, sufficiently small and depending on $\mu$, such that for $\cQ=\cQ_{p,\delta}$ we have
    $$
    \aligned
    2\,\Re [\partial_\zeta \cQ(\zeta,\eta) \zeta] &= p |\zeta|^p + 2\delta \left|\zeta \max\{|\zeta|^{p/2-1},|\eta|^{1-q/2}\} \right|^2, \\
   2 \, \Re [\partial_\eta \cQ(\zeta,\eta) \eta] & \leq [q+(2-q)\delta] |\eta|^q \leq \frac{ 2}{\mu}\Re [\partial_\eta \cQ(\zeta,\eta) \eta],
    \endaligned
    $$
    for all $\zeta,\eta \in \C$.
\end{lemma}
\begin{proof}
    The first equality holds for all $\delta >0$. 

   On the other hand, an easy calculation shows that 
    $$
    \aligned
    2\Re [\partial_\eta \cQ(\zeta,\eta) \eta] & \leq [q+(2-q)\delta] |\eta|^q, \\
    2\Re [\partial_\eta \cQ(\zeta,\eta) \eta] & \geq q |\eta|^q,
    \endaligned
    $$
    for all $\zeta, \eta \in \C$ and all $\delta >0$. Therefore, it suffices to prove that there exists $\delta >0$ such that
    $$
    \mu[ q+ (2-q)\delta] <  q.
    $$
    Since $\mu \in (0,1)$, this is verified for $\delta$ sufficiently small.
\end{proof}

In view of Lemma~\ref{l : first ord est Bell}, estimate \eqref{e : weak gen conv wrt pot stronger} for $\Phi=\cQ$ turns in
\begin{equation}\label{eq: strong belm conv f}
\aligned
 H_{\cQ}^{(A,B)}[(u,v)&; (\nabla u, \nabla v)]\\
 \geq&\,\,\tau(u,v)|\nabla u|^2+\tau^{-1}(u,v)|\nabla v|^2\\
&+ \alpha(V,\oV) \biggl( p|\nabla(|u|^{p/2-1}u)|^2 +2\delta |\nabla(u \max\{|u|^{p/2-1},|v|^{q/2-1}\})|^2\biggr)\\
&+ \alpha(W,\oW) [q+(2-q)\delta] |\nabla(|v|^{q/2-1}v) |^2.
\endaligned
\end{equation}

	In the same spirit as \cite{CD-DivForm, CD-Mixed, CD-Potentials}, we aim to establish a pointwise estimate of $H^{(A,B)}_{\cQ}$ on $\C^2 \times \C^{2d}$ in a such way that ensures the validity of \eqref{eq: strong belm conv f} for $u=T^{A,V,\oV}_tf$ and $v=T_t^{B,W,\oW}g$, where $f,g \in L^p \cap L^q$. To this end, we first need to show that the functions $|u|^{p/2-1}u$ and $u \max\{|u|^{p/2-1},|v|^{q/2-1}\}$ belong to $\oV$ and $|v|^{q/2-1}v$ to $\oW$ and that their gradients can be computed using the chain rule.
Regarding the terms $|u|^{p/2-1}u$ and $|v|^{q/2-1}v$, we can rely on \cite{CD-Potentials, Egert20, SV} for this purpose. However, justifying the chain rule for $u \max\{|u|^{p/2-1},|v|^{q/2-1}\}$ requires more effort. This is the focus of Section~\ref{s: CR}.

Subsequently, under the assumption that $(A,V) \in \cA\cP_p(\Omega,\oV)$ and $(B,W) \in \cA\cP_p(\Omega,\oW)$  we establish the desired lower pointwise estimate of $H_{\cQ}^{(A,B)}$  in Section~\ref{s: conv bell}, which will imply \eqref{eq: strong belm conv f}.

\section{$L^p$ contractivity and analyticity of  $(T^{A,V,\oV}_{t})_{t>0}$} \label{s : Lp contr}
In this section we prove all the results stated in Section~\ref{s: sem emb}. We begin by establishing some elementary properties of the class $\cA\cP_p(\Omega)$, which will be used later on. These properties closely resemble those of the class $\cA_p(\Omega)$ \cite{CD-DivForm}.
\subsection{Basic property of the perturbed $p$-ellipticity}
The following lemma is a consequence of the inequality
$$
|\Delta_p(A) -\Delta_p(B)| \leq \frac{\|A-B \|_\infty}{\min\{p,q\}},
$$
showed in \cite[p. 3204]{CD-DivForm} for all $A,B \in L^\infty(\Omega; \C^{d,d})$.
\begin{lemma}\label{eq: Lip pell}
Let $k \in \N$, $U \subseteq \R^k$ be open and $f: U \rightarrow \C$ be (Lipschitz) continuous. Then $U \ni \omega \mapsto \Delta_p(A-f(\omega)B)$ is (Lipschitz) continuous for all $p \in (1,\infty)$ and $A,B \in L^\infty(\Omega;\C^{d,d})$. 
\end{lemma}
For next proposition it is useful recalling the $\R$-linear operator $\cI_p$ which appears in the original definition of $\Delta_p$ \cite{CD-DivForm}. For every $p \in   (1, \infty)$ define the map $\cI_p: \C^d \rightarrow \C^d$  by
\begin{equation*}
\cI_{p}(\xi)=\xi + (1-2/p) \overline{\xi}, \quad \xi \in \C^d.
\end{equation*}
Then \cite{CD-DivForm},
\begin{equation*}
\Delta_p(A) =
\underset{x\in\Omega}{{\rm ess}\inf}\min_{|\xi|=1}\, \Re\sk{A(x)\xi}{\cI_p\xi}_{\C^{d}},
\end{equation*}
and, equivalently, $A$ is $p$-elliptic if there exists $C=C(A,p) >0$ such that for a.e. $x \in \Omega$,
\begin{equation*}
\Re\sk{A(x)\xi}{\cI_p\xi}_{\C^{d}}
\geq C |\xi|^2\,,
\quad \forall\xi\in\C^{d}.
\end{equation*}
\begin{proposition}
\label{p : basic prop}
Let $p \in (1, \infty)$, $A \in \cA_p(\Omega)$ and $V \in \cP(\Omega)$ such that $(A,V) \in \cA\cP_{p}(\Omega)$.
Then 
\begin{enumerate}[label=\textnormal{(\roman*)}]
\item
\label{i : inv conj}
$(A,V) \in \cA\cP_{q}(\Omega)$, where $q$ is the conjugate exponent of $p$;
\item
\label{i : interpol}
 $(A,V) \in \cA\cP_{r}(\Omega)$ for all exponents $r$ satisfying $|1/2-1/r|\leq |1/2-1/p|$;
\item
\label{i : end-open cond}
there exists $\varepsilon >0$ such that $(A,V) \in \cA\cP_{p+\varepsilon}(\Omega)$;
\item
\label{i : inv small rot}
 there exists $\theta \in (0, \pi/2)$ such that $(e^{i\phi}A, V  \cos\phi) \in \cA\cP_p(\Omega)$
 for all $\phi \in [-\theta,\theta]$;
\item
\label{i : inv small sub}
 there exists $\varepsilon >0$ such that $A -\mu(pq/4)I_d \in \cA_p(\Omega)$ for all $\mu \in [\alpha(V)-\varepsilon, \alpha(V) + \varepsilon]$;
\item
\label{i : inv adj}
$(A^*,V) \in  \cA\cP_p(\Omega)$.
\end{enumerate}
\end{proposition}
\begin{proof}
Item \ref{i : inv conj} follows by the identity $\cA_p(\Omega)=\cA_{q}(\Omega)$ \cite[Proposition~5.8]{CD-DivForm}.


Since $I_d$ is $r$-elliptic for all $r \in (1,\infty)$ and $rr^\prime \leq pq$ for  all exponents $r$ satisfying $|1/2-1/r|\leq |1/2-1/p|$,
we have
$$
\Re\sk{\left(A-\alpha\frac{rr^\prime}{4}I_d\right)\xi}{\cI_r\xi} \geq \Re\sk{\left(A-\alpha\frac{pq}{4}I_d\right)\xi}{\cI_r\xi} \geq \Delta_r\left(A-\alpha\frac{pq}{4}I_d\right) |\xi|^2,
$$
for all $\xi \in \C^d$. Then item \ref{i : interpol} follows from the $r$-ellipticity of $A-(\alpha pq/4 ) I_d$, which is a consequence of the fact that $\{ \cA_p(\Omega) : p \in [2,\infty)\}$ is a decreasing chain of matrix classes \cite[Corollary~5.16]{CD-DivForm}.

Item \ref{i : end-open cond} follows from the continuity of $p \mapsto \Delta_p$ \cite[Corollary~5.16]{CD-DivForm} and  $p \mapsto \alpha pq/4$, together with Lemma~\ref{eq: Lip pell}.

Since $\alpha(V \cos\phi) \leq \alpha(V)$ for all $\phi \in [-\pi/2,\pi/2]$ and $I_d$ is $p$-elliptic, in order to prove item \ref{i : inv small rot} it suffices to show that for all $A,B \in L^\infty(\Omega;\C^{d,d})$ such that $A-B \in \cA_p(\Omega)$ there exists $\theta \in (0,\pi/2)$ for which
$$
e^{i\phi}A-B \in \cA_p(\Omega),
$$
for all $\phi \in [-\theta,\theta]$. This holds true since $\phi \mapsto \Delta_p(e^{i\phi}A-B)$ is Lipschitz continuous on $(-\pi/2,\pi/2)$ by Lemma~\ref{eq: Lip pell}.

Item \ref{i : inv small sub} is a consequence of the Lipschitzianity of $\mu \mapsto \Delta_p(A-\mu I)$ on $\R$, which is guaranteed by Lemma~\ref{eq: Lip pell}.

Item \ref{i : inv adj} follows by \cite[Corollary~5.17]{CD-DivForm}.
\end{proof}

\subsection{Proof of Theorem \ref{t : contract}}\label{ss : Lp contr}

Before proving Theorem~\ref{t : contract}, we begin by recalling \cite[Lemma~B.6]{CD-Potentials} and presenting a corollary. 
\begin{lemma}{\cite[Lemma~B.6]{CD-Potentials}}
\label{t: aus Nittka}
Let $u \in W^{1,2}(\Omega)$ and $p \in (1,\infty)$. The function $|u|^{p-2}u$ belongs to $W^{1,2}(\Omega)$ if and only if $|u|^{p-2}u \in L^2(\Omega)$ and $|u|^{p-2} \nabla u \in L^2(\Omega;\C^d)$. In this case,
$$
\nabla(|u|^{p-2}u) = \frac{p}{2} |u|^{p-2} {\rm sign} u \cdot \cI_p( {\rm sign}\overline{u} \cdot \nabla u) \mathds{1}_{\{u \ne0\}}.
$$
Consequently,
$$
\big|\nabla(|u|^{p-2}u)\big| \sim |u|^{p-2} |\nabla u| \mathds{1}_{\{u \ne0\}}.
$$
\end{lemma}

\begin{corollary}
\label{c : note andrea}
Suppose that $\oV$ satisfies \eqref{eq: inv P} and \eqref{eq: inv N}. 
Let $u \in \oV$ and $p \in (1,\infty)$ such that $|u|^{p-2}u \in W^{1,2}(\Omega)$. Then $u \in L^p(\Omega)$, $|u|^{p/2-1}u \in \oV$ and
\begin{equation}
\label{e : stima grad p}
\big|\nabla(|u|^{\frac{p}{2}-1}u)\big|^2 = \frac{p}{2}|u|^{p-2}\left(\frac{p}{2} |\Re( {\rm sign }\overline{u} \cdot\nabla u)|^2 + \frac{2}{p}  |\Im( {\rm sign }\overline{u} \cdot\nabla u)|^2\right) \mathds{1}_{\{u \ne0\}}. 
\end{equation}
\end{corollary}
\begin{proof}
Let us prove first that $u \in L^p(\Omega)$, $|u|^{p/2-1}u \in W^{1,2}(\Omega)$ and \eqref{e : stima grad p}. By assumption $|u|^{p-1} \in L^2(\Omega)$, hence $|u|^p = |u| \cdot |u|^{p-1} \in L^1(\Omega)$, namely $|u|^{p/2-1}u \in L^2(\Omega)$. By Lemma~\ref{t: aus Nittka},
$$
|u|^{p-2} |\nabla u| \sim | \nabla( |u|^{p-2}u)| \in L^2(\Omega),
$$
which implies that $|u|^{p-2}|\nabla u|^2 \in L^1(\Omega)$. Therefore, $|u|^{p/2-1}\nabla u \in L^2(\Omega;\C^d)$. We conclude by applying Lemma~\ref{t: aus Nittka} with $p/2+1$ instead of $p$. 

Finally, we prove that $|u|^{p/2-1}u \in \oV$. For all $n \in \N$ the function $\Phi_n(\zeta)=\zeta(|\zeta|^{p/2-1}\wedge n)$ is Lipschitz continuous with $\Phi_n(0)=0$. Therefore, from Proposition~\ref{p: invariance under L} we have $u ( |u|^{p/2-1} \wedge n ) \in \oV$ with gradient given by \cite[(10)]{Egert20}. Hence, 
\begin{equation}
\label{eq: bound n}
\| u ( |u|^{p/2-1} \wedge n ) \|_{\oV} \leqsim \| |u|^{p/2-1}u \|_{W^{1,2}(\Omega)} < \infty,
\end{equation}
for all $n \in \N$. On the other hand, Lebesgue's dominated convergence theorem and the fact that $u \in L^p$ give 
\begin{equation}
\label{eq: conv norme}
\| u ( |u|^{p/2-1} \wedge n )\|_2 \rightarrow \| |u|^{p/2-1}u\|_2,
\end{equation}
and  $u ( |u|^{p/2-1} \wedge n )  \rightarrow |u|^{p/2-1} u$ in $\cD^\prime(\Omega)$, as $n \rightarrow \infty$. From the latter convergence, the density of $C_c^\infty(\Omega)$ in $L^2(\Omega)$ and \eqref{eq: bound n} we deduce that  
\begin{equation}
\label{eq: weak conv tronc}
u ( |u|^{p/2-1} \wedge n )  \rightharpoonup |u|^{p/2-1} u
\end{equation}
in $L^2(\Omega)$, as $n \rightarrow \infty$. By combining \eqref{eq: conv norme} and \eqref{eq: weak conv tronc} we obtain   $u ( |u|^{p/2-1} \wedge n )  \rightarrow |u|^{p/2-1} u$ strongly in $L^2(\Omega)$, as $n \rightarrow \infty$.  Thus, \eqref{eq: bound n} and Lemma~\ref{l: sbsq V} yield $ |u|^{p/2-1} u \in \oV$.
\end{proof}

Let $(A,V) \in \cA\cP(\Omega)$. We prove now Theorem \ref{t : contract} and Corollary \ref{c: N analytic sem}.

Let $(\Omega, \mu)$ be a measure space, $\gotb$ a sesquilinear form defined on the domain $\Dom(\gotb) \subset L^2=L^2(\Omega)$ and $1<p<\infty$. Denote 
$$
\Dom_p(\gotb) := \{ u \in \Dom(\gotb) : |u|^{p-2} u \in \Dom(\gotb) \}.
$$
We say that $\gotb$ is {\it $L^p$-dissipative} if 
$$
\Re \gotb(u, |u|^{p-2} u) \ge 0 \quad \forall u \in \Dom_p(\gotb).
$$
The notion of $L^p$-dissipativity of sesquilinear forms  was introduced by Cialdea and Maz'ya in \cite{CiaMaz} for forms defined on $C_c^1(\Omega)$. Then it was extended by Carbonaro and Dragi\v{c}evi\'c in \cite[Definition 7.1]{CD-DivForm}. 

In order to prove the $L^p$-contractivity of $(T_t^{A,V,\oV})_{t>0}$, we follow the proof of the implication $(a) \Rightarrow (b)$ in \cite[Theorem~1.3]{CD-DivForm} for which the following theorem due to Nittka \cite[Theorem 4.1]{Nittka} is essential. We reproduce it in the form it appeared in \cite[Theorem 2.2]{CD-Potentials}.
\begin{theorem}[Nittka]\label{t: Nittka}
Let $(\Omega, \mu)$ be a measure space. Suppose that the sesquilinear form $\gota$ on $L^2=L^2(\Omega,\mu)$ is densely defined, accretive, continuous and closed. Let $\oL$ be the operator associated with $\gota$.

Take $p \in (1,\infty)$ and define $B^p := \{u \in L^2 \cap L^p : \|u\|_p \leq 1 \}$. Let $\bP_{B^p}$ be the orthogonal projection $L^2 \rightarrow B^p$. Then the following assertions are equivalent:
\begin{itemize}
\item $\| {\rm exp}(-t\oL) f\|_p \leq \|f\|_p$ for all $f \in L^2 \cap L^p$ and all $t \geq 0$;
\item $\Dom(\gota)$ is invariant under $\bP_{B^p}$ and $\gota$ is $L^p$-dissipative.
\end{itemize}
\end{theorem}
 
\begin{proof}[Proof of Theorem \ref{t : contract}]
We will use Nittka's invariance criterion (Theorem \ref{t: Nittka}). Under our assumptions on $\phi$, the sesquilinear form $\gotb := e^{i\phi}\gota$ is densely defined, closed and sectorial. It is well-known that a sectorial form is accretive and continuous; see for example \cite[Proposition 1.8]{O}. Therefore, it falls into the framework of Nittka's criterion. The operator associated with $\gotb$ is $e^{i\phi} \oL^{A,V}$. 

The invariance of $\Dom(\gotb) = \Dom(\gota_{A,V,\oV}) = \Dom(\gota_{A,V_+,\oV})$ under $\bP_{B^p}$ was established in \cite[Theorem~1.2]{CD-Potentials}, assuming only condition \eqref{eq: inv P} on $\oV$. Thus, it remains to prove the $L^p$-dissipativity of $\gotb$. To this end, we will make use of the fact that $\oV$ also satisfies condition \eqref{eq: inv N}.

Let $u \in \Dom_p(\gotb)$. By \cite[(2.3)]{CD-Potentials}, applied with $B = e^{i\phi}A$, we get 
\begin{equation}
\label{eq: 1 nittka dis}
\aligned
\Re \gotb(u, &|u|^{p-2} u) \\
\geq &  \int_{\Omega}\frac{p}{2}|u|^{p-2} \Re\sk{e^{i \phi}A ( {\rm sign } \overline{u} \cdot \nabla u)}{\cI_p({\rm sign } \overline{u} \cdot \nabla u)} + (\cos\phi) V_+|u|^p  \\
&- \int_\Omega (\cos\phi) V_- \left| |u|^{\frac{p}{2}-1} u \right|^2.
\endaligned 
\end{equation}
If $V_-=0$, the assumption $(e^{i \phi}A, V \cos\phi) \in \widetilde{\cA\cP}_p(\Omega,\oV)$ is equivalent to the condition $\Delta_p(e^{i \phi}A) \geq 0$. Hence, the integrand in the right-hand side term of \eqref{eq: 1 nittka dis} is nonnegative and we conclude. Notice that in this case we did not need the assumption \eqref{eq: inv N} on $\oV$.

Suppose now that $V_-\ne 0$. We would like to use in \eqref{eq: 1 nittka dis} the subcritical inequality \eqref{e : subc ineq} applied with the potential $(\cos\phi)V$ and $v = |u|^{p/2-1}u$. To this purpose, we can rely on Corollary~\ref{c : note andrea} as it guarantees that $|u|^{p/2-1}u \in \oV$. We highlight that only in this step we are making use of the fact that $\oV$ also satisfies \eqref{eq: inv N}.

So, the subcritical inequality \eqref{e : subc ineq} gives
$$
\aligned
\Re \gotb(u, &|u|^{p-2} u) \\
\geq   \int_{\Omega}&\frac{p}{2}|u|^{p-2} \Re\sk{e^{i \phi}A ( {\rm sign } \overline{u} \cdot \nabla u)}{\cI_p({\rm sign } \overline{u} \cdot \nabla u)} - \alpha( V \cos\phi) |\nabla(|u|^{\frac{p}{2}-1}u)|^2\\
&+ (1-\sigma)(\cos\phi) V_+|u|^p.
\endaligned 
$$
By combining  \eqref{e : stima grad p} with the fact that $pq/4 \geq 1$ for all $p \in (1,\infty)$, we get
\begin{equation*}
\aligned
\big|\nabla(|u|^{\frac{p}{2}-1}u)\big|^2 &\leq \frac{p}{2} \cdot \frac{pq}{4} |u|^{p-2}\left(\frac{2}{q} |\Re( {\rm sign }\overline{u}\cdot \nabla u)|^2 + \frac{2}{p}  |\Im( {\rm sign }\overline{u}\cdot \nabla u)|^2\right) \\
& = \frac{p}{2} \cdot \frac{pq}{4}|u|^{p-2} \Re\sk{ {\rm sign }\overline{u} \cdot \nabla u}{\cI_p({\rm sign } \overline{u} \cdot \nabla u)}.
\endaligned
\end{equation*}
Hence, we may continue as
$$
\aligned
\Re \gotb(u, &|u|^{p-2} u) \\
\geq \int_\Omega &\frac{p}{2}  |u|^{p-2} \Re \sk{\left(e^{i \phi}A- \alpha(V \cos\phi)\frac{pq}{4} I_d\right) ( {\rm sign } \overline{u} \cdot \nabla u)}{\cI_p({\rm sign } \overline{u} \cdot \nabla u)} \\
&+ ( 1-\sigma) (\cos\phi)V_+ |u|^p.
\endaligned
$$
Now, the integrand is nonnegative since $(e^{i \phi}A, V \cos\phi) \in \widetilde{\cA\cP}_p(\Omega,\oV)$ and $\sigma \in [0,1)$.
\end{proof}

\begin{proof}[Proof of Corollary \ref{c: N analytic sem}]
By Proposition~\ref{p : basic prop}\ref{i : inv conj},\ref{i : interpol},\ref{i : inv small rot} there exists $\theta=\theta(p,A,V) >0$ such that 
$$
\Delta_r( e^{i\phi} A - \alpha(V \cos\phi)(rr^\prime/4) I_d) >0
$$
for all $\phi \in [-\theta,\theta]$ and all $r$ satisfying $|1/2-1/r|\leq |1/2-1/p|$. The contractivity part now follows from Theorem \ref{t : contract} and the relation
$$
T_{t e^{i\phi}} = \textrm{exp}\left( -te^{i\phi} \oL \right),
$$
whereupon analyticity is a consequence of a standard argument \cite[Chapter II, Theorem 4.6]{EN}.
\end{proof}

\subsection{$L^p$-estimates of $\oL^{A,V,\oV}$}
We now establish some $L^p$-estimates for the operator $\oL^{A,V}$. Besides being of independent interest, these results serve as auxiliary tools for proving Corollary~\ref{t: eg extr}, and for applying the chain rule in Proposition~\ref{p : fucnt a tratti in sob}.
\medskip

Let $r \in (1,\infty)$. Define $F_r : \C \rightarrow \R_+$ by
$$
F_r(\zeta) = |\zeta|^r, \quad \zeta \in \C.
$$
From \cite[Lemma~5.6]{CD-DivForm} applied with $A=I_d$ we have
$$
\aligned
 H_{F_r}^{I_d}[\zeta;X] &=\frac{r^2}{2} |\zeta|^{r-2} \Re\sk{{\rm sign}\overline{\zeta} \cdot X}{\cI_r({\rm sign}\overline{\zeta}\cdot X)}, \\
& = \frac{r^2}{2} |\zeta|^{r-2} \left(\frac{2}{r^\prime} |\Re( {\rm sign }\overline{\zeta}\cdot X)|^2 + \frac{2}{r}  |\Im( {\rm sign }\overline{\zeta}\cdot X)|^2\right),
\endaligned
$$
for all $\zeta \in \C \setminus\{0\}$ and $X \in \C^d$. Therefore, since $rr^\prime/4 \geq 1$ for all $r \in (1,\infty)$,
\begin{equation}
\label{e : est hess Fp wrt I}
\frac{r^\prime}{4}  H_{F_r}^{I_d}[\zeta;X] \geq  \frac{r}{2}|\zeta|^{r-2}\left(\frac{r}{2} |\Re( {\rm sign }\overline{\zeta} \cdot X)|^2 + \frac{2}{r}  |\Im( {\rm sign }\overline{\zeta} \cdot X)|^2\right),
\end{equation}
for all $\zeta \in \C\setminus \{0\}$ and $X \in \C^d$. In particular, from \eqref{e : stima grad p} we deduce that
\begin{equation}
\label{eq: grad p leq HpI 2}
\big|\nabla(|u|^{\frac{r}{2}-1}u)\big|^2 \leq \frac{r^\prime}{4}H_{F_r}^{I_d}[u;\nabla u] \mathds{1}_{\{u \ne 0\}},
\end{equation}
for all $u \in W^{1,2}(\Omega)$ such that $|u|^{\frac{r}{2}-1}u \in W^{1,2}(\Omega)$.
\medskip

The next result was already proved by Egert in \cite[Proposition~11]{Egert20} for the case $p \ge 2$ and $V=0$. Here we extend it to the general setting.
\begin{proposition}
\label{t : p grad est}
Suppose that $\oV$ satisfies \eqref{eq: inv P} and \eqref{eq: inv N}. 
Let $p \in (1,\infty)$ and $(A,V) \in \cA\cP_p(\Omega,\oV)$. If $u \in \Dom(\oL^{A,V}) \cap L^p(\Omega)$ is such that $\oL^{A,V}u \in L^p(\Omega)$, then $|u|^{p/2-1} u \in \oV$ and
\begin{equation}
\label{eq: 399}
\|\nabla(|u|^{\frac{p}{2}-1}u) \|_2^2 \leqsim \|\oL^{A,V} u \|_p \, \|u \|_p^{p-1}.
\end{equation}

In particular, $|u|^{p-2}|\nabla u|^2 \mathds{1}_{\{u \ne 0\}} \in L^1(\Omega)$.

Furthermore,
$$
\int_\Omega V_- |u|^p \leq \alpha \int_\Omega \frac{q}{4}H_{F_p}^{I_d}[u; \nabla u] \mathds{1}_{\{u \ne 0\}} + \sigma \int_\Omega V_+ |u|^p,
$$
where $q =p/(p-1)$.
\end{proposition}
\begin{proof}
Let $u \in \Dom(\oL^{A,V}) \subseteq \Dom(\gota_{A,V,\oV})$ be such that $u, \oL^{A,V}u \in L^p(\Omega)$. Then, by Lebesgue's dominated convergence theorem we deduce that
\begin{equation}\label{e : first est tr}
\Re \int_\Omega \oL^{A,V}(u) \cdot  \overline{u}  |u|^{p-2}=\lim_{n \rightarrow \infty} \Re \int_\Omega \oL^{A,V}(u) \cdot  \overline{u} \left( |u|^{p-2} \wedge n \right).
\end{equation}
Proposition~\ref{p: invariance under L} and the fact that $u \in \Dom(\gota_{A,V,\oV})$ give $u \left( |u|^{p-2} \wedge n \right) \in  \Dom(\gota_{A,V,\oV})$. Therefore, \eqref{eq: ibp} yields
\begin{equation}
\label{eq: ibp pf}
\aligned
\int_\Omega \oL^{A,V}(u) \cdot \overline{u} \left( |u|^{p-2} \wedge n \right) =&\, \int_\Omega \sk{A \nabla u}{\nabla\left[u\left( |u|^{p-2} \wedge n \right) \right]} +V_+ |u|^2  \left( |u|^{p-2} \wedge n \right)\\
&- \int_\Omega V_- \left| u( |u|^{p/2-1}\wedge \sqrt{n})\right|^2.
\endaligned 
\end{equation}
Again from Proposition~\ref{p: invariance under L} we have $u(|u|^{p/2-1} \wedge \sqrt{n}) \in \oV$. Hence
\begin{equation}
\label{eq: sub ineq tronc}
\int_\Omega V_- \left| u( |u|^{p/2-1}\wedge \sqrt{n})\right|^2 \leq \alpha \int_\Omega \bigl|\nabla\bigl[u\left(|u|^{p/2-1}\wedge \sqrt{n}\right)\bigr]\bigr|^2 + \sigma \int_\Omega V_+ |u|^2  \left( |u|^{p-2} \wedge n \right).
\end{equation}

From \cite[Lemma~5.2]{SV}, the identity $\nabla |u| = \Re\left({\rm sign}\overline{u} \cdot \nabla u\right) \mathds{1}_{\{u \ne 0\}}$ and the fact that $\nabla u =0$ almost everywhere on $\{u=0\}$, we obtain 
\begin{equation}
\label{eq: chain rule pf}
\aligned
\nabla\left[u\left(|u|^{p-2}\wedge n\right)\right] &= \left(|u|^{p-2}\wedge n\right)\left(\nabla u + (p-2) \mathds{1}_{\{|u|^{p-2}<n\}}\,{\rm sign}u\nabla |u| \right)\\
&=  n \nabla u\mathds{1}_{\{|u|^{p-2}\geq n\}} + |u|^{p-2} \left(\nabla u + (p-2) \,{\rm sign}u\nabla |u| \right)\mathds{1}_{\{|u|^{p-2}<n\}} \\
&= n \nabla u \mathds{1}_{\{|u|^{p-2}\geq n\}} + \frac{p}{2} |u|^{p-2} {\rm sign} u \cdot \cI_p( {\rm sign}\overline{u} \cdot \nabla u) \mathds{1}_{\{|u|^{p-2}<n, u \ne 0\}},
\endaligned
\end{equation}
and
\begin{equation}
\label{eq: chain rule pf2}
\aligned
\bigl|\nabla\bigl[u&\left(|u|^{p/2-1}\wedge \sqrt{n}\right)\bigr]\bigr|^2 \\
=&\,\, n |\nabla u|^2 \mathds{1}_{\{|u|^{p-2}\geq n\}} + |u|^{p-2}\left|\nabla u +\left(\frac{p}{2}-1\right) {\rm sign} u \nabla |u|\right|^2 \mathds{1}_{\{|u|^{p-2}<n\}}\\
=& \,\, n |\nabla u|^2 \mathds{1}_{\{|u|^{p-2}\geq n\}} \\
&+ \frac{p}{2}|u|^{p-2}\left(\frac{p}{2} |\Re( {\rm sign }\overline{u} \cdot\nabla u)|^2 
+ \frac{2}{p}  |\Im( {\rm sign }\overline{u} \cdot\nabla u)|^2\right) \mathds{1}_{\{|u|^{p-2}<n, u \ne 0\}}\\
\leq&\,\,  n |\nabla u|^2 \mathds{1}_{\{|u|^{p-2}\geq n\}} + \frac{q}{4}H_{F_p}^{I_d}[u;\nabla u]\mathds{1}_{\{|u|^{p-2}<n, u \ne 0\}},
\endaligned
\end{equation}
where in the last inequality we used \eqref{e : est hess Fp wrt I}.

From Proposition~\ref{p : basic prop}\ref{i : interpol} it follows that $(A,V) \in \cA\cP(\Omega,\oV)$, that is, $A-\alpha I_d$ is elliptic. Therefore, by combining \cite[Lemma~5.6]{CD-DivForm}, \eqref{eq: ibp pf}, \eqref{eq: sub ineq tronc} \eqref{eq: chain rule pf} and \eqref{eq: chain rule pf2}, we obtain
\begin{equation}
\label{e : second est tr}
\aligned
\Re \int_\Omega \oL^{A,V}(u) &\cdot  \overline{u} \left( |u|^{p-2} \wedge n \right) \\
\geq & \,\, n \int_{\{|u|^{p-2} \geq n\}} \Re\sk{A\nabla u}{\nabla u} -\alpha  |\nabla u|^2 \\
&+ p^{-1} \int_{\{|u|^{p-2} < n, u \ne 0 \}} H_{F_p}^A[u;\nabla u] - \alpha \frac{pq}{4}H_{F_p}^{I_d}[u;\nabla u]  \\
&+(1-\sigma) \int_\Omega V_+ |u|^2 (|u|^{p-2} \wedge n) \\
 \geq &\,\, \int_{\{|u|^{p-2} < n, u \ne 0 \}} p^{-1} H_{F_p}^{A-\alpha (pq/4)I_d}[u;\nabla u]  + (1-\sigma)  V_+ |u|^p.
\endaligned
\end{equation}

 Finally,  $(A,V) \in \cA\cP_p(\Omega,\oV)$, \cite[Corollary~5.10]{CD-DivForm}, Fatou's Lemma,  \eqref{e : first est tr} and \eqref{e : second est tr} give
\begin{equation}
\label{eq: 400}
\Re \int_\Omega \oL^{A,V}(u) \cdot  \overline{u}  |u|^{p-2} \geqsim \int_\Omega |u|^{p-2} |\nabla u|^2 \mathds{1}_{\{u \ne0\}} + V_+|u|^p.
\end{equation}
By combining \eqref{eq: 400} with the assumptions on $u$ and Lemma~\ref{t: aus Nittka} applied with $p/2+1$ instead of $p$, we obtain $|u|^{p/2-1}u \in W^{1,2}(\Omega)$ and 
\begin{equation}
\label{eq. 401}
|\nabla(|u|^{p/2-1}u)|^2 \sim |u|^{p-2}|\nabla u|^2 \mathds{1}_{\{u \ne 0\}}.
\end{equation}
In order to deduce that $|u|^{p/2-1}u\in \oV$ we may argue as in the proof of Corollary~\ref{c : note andrea}, while \eqref{eq: 399} follows from \eqref{eq: 400},  \eqref{eq. 401} and  H\"{o}lder inequality.

The final statement is a consequence of \eqref{eq: grad p leq HpI 2}, which holds as $|u|^{p/2-1} u \in W^{1,2}(\Omega)$.
\end{proof}

\subsection{Egert's extrapolation: proof of Corollary~\ref{t: eg extr}}
Let $p \in (1,\infty)$. We have shown that the semigroup $(T_t^{A,V,\oV})_{t>0}$ extrapolates to $L^r(\Omega)$ for all exponents  $r$ satisfying $|1/2-1/r|\leq |1/2-1/p|$, whenever $(A,V) \in \cA\cP_p(\Omega,\oV)$ and $\oV$ satisfies \eqref{eq: inv P} and \eqref{eq: inv N}; see  Corollary~\ref{c: N analytic sem}.
In this section,  we extend the extrapolation range  under additional assumptions on $\oV$. We adapt the argument of Egert in \cite{Egert20}: the key idea is to combine the extrapolation of the semigroup on $L^p(\Omega)$, already at our disposal,  with ultracontractivity techniques which rely on $L^2$ off-diagonal bounds for the semigroup. The additional assumptions on $\oV$ are required precisely to ensure that these techniques can be applied.  The argument from \cite{Egert20} carries over to our setting thanks to elliptic inequalities still satisfied by the underlying sesquilinear form, together with Proposition~\ref{t : p grad est}, which generalizes \cite[Proposition~11]{Egert20}. For this reason, we shall not reproduce all details of the proofs but rather highlight the key properties that allow us to adapt the results of \cite[Section~3.3. \& Section~5]{Egert20} to the present framework. For a deeper understanding of the method and its technical underpinnings, we refer the reader to the original exposition in \cite{Egert20}. 
\medskip

 For $d \geq 3$, let $2^* := 2d/(d-2)$ denote the Sobolev conjugate of $2$.  
 In the same spirit as \cite{Egert20}, we introduce the following definition.
\begin{defi}\label{d: emb prop}
Let $d \geq 3$ and $\oV$ be a closed subspace of $W^{1,2}(\Omega)$ containing $W_0^{1,2}(\Omega)$. If $\|v\|_{2^*} \leqsim \|v\|_{1,2}$ holds for all $v \in \oV$, then $\oV$ has the {\it embedding property}. It has the {\it homogeneous embedding property} if $\|v\|_{2^*} \leqsim \|\nabla v\|_2$ for all $v \in \oV$.
\end{defi}
In order to apply Davies' perturbation method, we will also require invariance under multiplication by bounded Lipschitz functions, that is, \eqref{eq: inv mult lip}.

Since we shall work with a single triple $(A,V,\oV)$, we  simplify the notation by writing
$$
\aligned
T_t := T_t^{A,V,\oV},\qquad \oL:= \oL^{A,V,\oV}.
\endaligned
$$

Recall that $\theta_0$ is the angle defined in page \pageref{e : above boun a}. The proof of Corollary~\ref{t: eg extr} follows step by step  \cite[Section~3.3 \& Section~5]{Egert20}. The first ingredient are $L^2$ off-diagonal estimates for the semigroup. Thanks to \eqref{eq: inv mult lip}, for any bounded Lipschitz function $\varphi$ with $\|\nabla \varphi\|_\infty \leq 1$ and  any $\varrho >0$, we may define the perturbed sesquilinear form 
$$
\gotb(u,v) := \gota( e^{\varrho \varphi}u, e^{-\varrho \varphi}v), \qquad u,v \in \Dom(\gota).
$$
 Hence, in view of 
 the elliptic inequalities
\begin{equation}
\label{eq: ell ineq pot}
\begin{array}{rcll}
|\gota(u,u)| &\leqsim& \|\nabla u\|_2^2 + \|V_+^{1/2}u\|_2^2, & \qquad u \in \Dom(\gota),\\
\Re \gota(u,u) &\geq &c \left( \|\nabla u\|_2^2 + \|V_+^{1/2}u\|_2^2\right), &\qquad u \in \Dom(\gota),
\end{array}
\end{equation}
we can apply Davies' perturbation method  and prove the following result by choosing appropriate function $\varphi$ and parameter $\varrho$.  For more details, see, for example,  the proof of \cite[Proposition~7]{Egert20}.
\begin{proposition}\label{p : 2 offdiag}
Let $\oV$ satisfy \eqref{eq: inv mult lip}, $(A,V) \in \cA\cP(\Omega,\oV)$ and $\psi \in [0, \pi/2 -\theta_0)$. For all measurable sets $E, F \subseteq \Omega$, all $z \in \bS_\psi$, and all $f \in L^2(\Omega)$ with support in $E$ it follows
$$ 
\|T_z f \|_{L^2(F)} \leq e^{-\frac{d(E,F)}{4C|z|}} \|f \|_{L^2(E)},
$$
where $C= \Lambda + (\Lambda^2 \cos(\omega))/(c \cos(\psi+\theta_0))$. Here $c$ is the constant which appears in \eqref{eq: ell ineq pot}.
\end{proposition}
\medskip

Once the $L^2$ off-diagonal bounds are established, we can develop the ultracontractivity techniques of \cite[Section~5]{Egert20}.
\begin{defi}
Let $\psi \in [0,\pi)$. Given $p, r \in (1,\infty)$ with $p \leq r$, a family of operators $( S_z )_{z \in \bS_\psi} \subset \cL(L^2(\Omega))$ is said to be \emph{$p \to r$ bounded} if
\[
\|S_zf\|_r \leq C |z|^{\frac{d}{2r} - \frac{d}{2p}} \|f\|_p
\]
holds for some constant $C$ and all $z \in \bS_\psi$ and all $f \in L^p(\Omega) \cap L^2(\Omega)$.
\end{defi}

We now reproduce, in our setting, the analogues of  \cite[Lemma~15 \& Lemma~16]{Egert20}.
\begin{lemma}
\label{l: 15 E}
Assume $d \geq 3$ and that $\oV$ has the embedding property. Suppose that $(T_t)_{t>0}$
is $p \to p$ bounded and let $\varepsilon > 0$. If $p < 2$, then the shifted semigroup 
$(e^{-\varepsilon t} T_t)_{t>0}$ is $p \to 2$ bounded, and if $p > 2$, then it is $2 \to p$ bounded. 
If $\oV$ has the homogeneous embedding property, then the conclusion also holds for $\varepsilon = 0$.
\end{lemma}

\begin{proof}
The proof follows the lines of \cite[Lemma~15]{Egert20}. In particular, we can argue in the same way thanks to the elliptic inequality 
$$
\Re\gota(u,u) \geqsim \|\nabla u\|_2^2,  \qquad u \in \Dom(\gota).
$$
which follows from \eqref{eq: ell ineq pot}.
\end{proof}

\begin{lemma}
\label{l: 16 E}
Let $\varepsilon \geq 0$ and  $\oV$ satisfy \eqref{eq: inv mult lip}. Suppose either $p < 2$ and that $(e^{-\varepsilon t} T_t)_{t>0}$ 
is $p \to 2$ bounded, or suppose $p > 2$ and that it is $2 \to p$ bounded. 
Then for every $\psi \in \big[0, \tfrac{\pi}{2} - \theta_0 \big)$ and every $r$ between $2$ and $p$, 
the holomorphic extension $( e^{-\varepsilon z} T_z )_{z \in \bS_\psi}$ is $r \to r$ bounded.
\end{lemma}
\begin{proof}
In the proof of \cite[Lemma~16]{Egert20}, the peculiarity that the semigroup is generated 
by a divergence-form operator with an elliptic coefficient matrix is used solely 
to apply \cite[Proposition~7]{Egert20}. Apart from this, the argument applies 
to any semigroup satisfying the assumptions of the present lemma. 
Therefore, the claim follows from the proof of \cite[Lemma~16]{Egert20} together 
with Proposition~\ref{p : 2 offdiag}, which takes the role of \cite[Proposition~7]{Egert20}.
\end{proof}

The following proposition is modeled after \cite[Proposition~18]{Egert20}. 
\begin{proposition}
Assume $d \geq 3$ and that $\oV$ has the embedding property and satisfies \eqref{eq: inv P}, \eqref{eq: inv N} and \eqref{eq: inv mult lip}.
Let $p > 2$ and assume that $(A,V) \in \cA\cP_p(\Omega,\oV)$.
Then for every $\psi \in \big[0, \tfrac{\pi}{2} - \theta_0 \big)$ and every $\varepsilon > 0$ 
the semigroup $( e^{-\varepsilon z} T_z )_{z \in \bS_\psi}$ is $r \to r$ bounded for $r \in \big(2, \tfrac{dp}{d-2} \big)$.  
If $\oV$ has the homogeneous embedding property, then the same result holds for $\varepsilon = 0$.
\end{proposition}
\label{p: fin prop eg}
\begin{proof}
The result follows by adapting the proof of \cite[Proposition~18]{Egert20}, replacing
\begin{itemize}
\item \cite[Proposition~11]{Egert20} with Proposition~\ref{t : p grad est};
\item \cite[Theorem~2]{Egert20} with Corollary~\ref{c: N analytic sem};
\item \cite[Lemma~15]{Egert20} with Lemma~\ref{l: 15 E};
\item \cite[Lemma~16]{Egert20} with Lemma~\ref{l: 16 E}.\qedhere
\end{itemize}
\end{proof}
\medskip

\begin{proof}[Proof of Corollary~\ref{t: eg extr}]
By duality and Proposition~\ref{p : basic prop}\ref{i : inv conj},\ref{i : inv adj}, it suffices to consider the case $p, r > 2$. For 
$q$ satisfying
\[
\left|\frac{1}{2} - \frac{1}{r} \right|< \frac{1}{d} + \left(1 - \frac{2}{d}\right) \left|\frac{1}{2} - \frac{1}{p}\right|,
\]
the conclusion follows from Proposition~\ref{p: fin prop eg} and the holomorphy of the semigroup on $L^2(\Omega)$. 
Indeed, we can apply Stein interpolation to the restriction of $T_z$ to any ray $[0,\infty) e^{\pm i \psi}$ for $\psi \in [0, \pi/2 - \theta_0)$. The reader can refer to \cite[Theorem~10.8]{ACSVV} for this argument.

The endpoint case for $r$ is immediate, since the perturbed $p$-ellipticity is an open-ended condition; see Proposition~\ref{p : basic prop}\ref{i : end-open cond}.
\end{proof}

\section{Chain rule}\label{s: CR}
Let $p\ge2$ and denote by $q$ its conjugate exponent.
As explained at the very end of Section~\ref{s: hfmon} we would like to apply the chain rule to compute the gradient of $u \max\{|u|^{p/2-1}, |v|^{1-q/2} \}$. 
In general, justifying the chain rule is not a trivial problem. For real-valued functions belonging to $W^{1,2}(\Omega)$, it is known that the chain rule holds for composition with Lipschitz functions, see \cite[Theorem~2.1.11]{Ziemer}. However, this does not hold with the same generality for complex-valued or vector-valued functions, see \cite{Leonichain}.

In \cite{ARKR}  mapping theorems for Sobolev spaces of vector-valued functions are provided.  Given two Banach spaces $X \ne \{0\}$ and $Y$,  it has been proved that each Lipschitz continuous mapping $\Phi : X \rightarrow Y$ gives rise to a mapping $u \mapsto F \circ u$ from $W^{1,p}(\Omega,X)$ to $W^{1,p}(\Omega,Y)$ if and only if $Y$ has the Radon-Nikod\'ym Property. Moreover, if in addition $\Phi$ is one-sided Gateaux differentiable, no condition on the space is needed and a chain rule can even be proved.

We recall that a function $\Phi : X \rightarrow Y$ is said to be {\it one-sided Gateaux differentiable at $x$} if the right-hand limit
$$
D_v^+\Phi(x) := \lim_{t \rightarrow 0^+} \frac{1}{t}\left(\Phi(x+tv) - \Phi(x) \right)
$$
exists for every direction $v \in X$. In this case, the left-hand limit
$$
D_v^-\Phi(x) := \lim_{t \rightarrow 0^+} \frac{1}{-t}\left(\Phi(x-v) - \Phi(x) \right)
$$
exists as well and is given by 
\begin{equation}
\label{eq: p mm}
D^-_v\Phi(x) = -D^+_{-v}\Phi(x).
\end{equation}
We say that $\Phi$ is one-sided Gateaux differentiable on $X$ if it is one-sided Gateaux differentiable at $x$ for all $x \in X$.

As  special case of \cite[Theorem~4.2]{ARKR} we have the following.
\begin{theorem}\label{t: chain rule G}
Let $1 \leq p \leq \infty$ and $u \in W^{1,p}(\Omega, \R^2)$. Suppose that $\Phi : \R^2 \rightarrow \R$ is Lipschitz continuous and one-sided Gateaux differentiable, and assume furthermore that $\Omega$ is bounded or $\Phi(0)=0$. Then $\Phi \circ u \in W^{1,p}(\Omega)$ and we have the chain rule
$$
\partial_j (\Phi \circ u) = D^+_{\partial_ju} \Phi(u) = D^-_{\partial_ju}\Phi(u).
$$
\end{theorem}

\subsection{Chain rule in the heat-flow method}
In this section we will justify the chain rule for  $u \max\{|u|^{p/2-1}, |v|^{1-q/2} \}$ by means of Theorem~\ref{t: chain rule G}.
\smallskip

Let $p \geq 2$, $q=p/(p-1)$, $\delta \in (0,1)$ and $n \in \N$. Define the functions \\$\Phi_\delta,\Phi_{\delta,n} : (-\delta/2,+\infty)^2 \rightarrow [0,+\infty)$ by 
$$
\aligned
\Phi_{\delta,n}(s,t)&=\max\{ (s+\delta)^{p/2-1}, (t+\delta)^{1-q/2}\} \wedge n,\\
\Phi_n(s,t) &= \max\{ s^{p/2-1}, t^{1-q/2}\} \wedge n.
\endaligned
$$
Clearly, after defining $\varphi,\psi : (-\delta/2, +\infty) \rightarrow [0,+\infty)$ as
$$
\aligned
\varphi(s) &= (s+\delta)^{p/2-1},\\
\psi(t) &= (t+\delta)^{1-q/2},
\endaligned
$$
we can rewrite $\Phi_{\delta,n}$ as
$$
\aligned
\Phi_{\delta,n}(s,t) &= n \wedge
\begin{cases}
\varphi(s), &{\rm if} \,\, (s+\delta)^p \geq (t+\delta)^{q},\\
\psi(t), &{\rm if} \,\, (s+\delta)^p \leq (t+\delta)^{q},
\end{cases}\\
&= n \wedge
\begin{cases}
\varphi(s), &{\rm if} \,\, g(s) \geq t,\\
\psi(t), &{\rm if} \,\, g(s) \leq t,
\end{cases}
\endaligned
$$
with $g: ( (\delta/2)^{1/p}-\delta, +\infty) \rightarrow (-\delta/2, +\infty)$ being defined by
$$
g(s) = (s+\delta)^{p-1}-\delta.
$$
Finally, for all $(s,t) \in {\rm graf}(g)$ define
$$
\aligned
\Pi_{0}(s,t) &= \{(x,y) \in \R^2 :  y =  g^\prime(s)x \},\\
\Pi_{\pm}(s,t) &= \{(x,y) \in \R^2 : \pm y > \pm  g^\prime(s)x \}.
\endaligned
$$
\begin{lemma}\label{l: gat diff}
Let $p \geq 2$, $q=p/(p-1)$, $\delta \in (0,1)$ and $n \in \N$. Then $\Phi_{\delta,n}$ is Lipschitz continuous and one-sided Gateaux differentiable. In particular,
\begin{itemize}
\item if $n \ne (s+\delta)^p > (t+\delta)^q$,
$$
D^+_v\Phi_{\delta,n}(s,t) = \partial_s \Phi_{\delta,n}(s,t) \cdot v_1,
$$
for all $v=(v_1,v_2) \in \R^2$;
\item if $(s+\delta)^p < (t+\delta)^q \ne n$,
$$
D^+_v\Phi_{\delta,n}(s,t) = \partial_t \Phi_{\delta,n}(s,t) \cdot v_2,
$$
for all $v=(v_1,v_2) \in \R^2$;
\item if $n \ne (s+\delta)^p = (t+\delta)^q$,
$$
D^+_v\Phi_{\delta,n}(s,t) =
\begin{cases}
\varphi^\prime(s) \mathds{1}_{\{(s+\delta)^{p/2-1}<n\}} \cdot v_1, &{\rm if} \,\, v \in \overline{\Pi_-(s,t)},\\ 
\psi^\prime(t) \mathds{1}_{\{(t+\delta)^{1-q/2}<n\}} \cdot v_2, &{\rm if} \,\, v \in \overline{\Pi_+(s,t)};
\end{cases}
$$
\item if $n = (s+\delta)^p > (t+\delta)^q$,
$$
D^+_v\Phi_{\delta,n}(s,t) =
\begin{cases}
 \varphi^\prime(s) \cdot v_1, &{\rm if} \,\, v_1 < 0,\\ 
0, &{\rm if} \,\, v_1 \geq0;
\end{cases}
$$
\item if $  (s+\delta)^p < (t+\delta)^q =n$,
$$
D^+_v\Phi_{\delta,n}(s,t) =
\begin{cases}
 \psi^\prime(t) \cdot v_2, &{\rm if} \,\, v_2 < 0,\\ 
0, &{\rm if} \,\, v_2 \geq 0;
\end{cases}
$$
\item if $n = (s+\delta)^p = (t+\delta)^q$,
$$
D^+_v\Phi_{\delta,n}(s,t) =
\begin{cases}
\varphi^\prime(s) \cdot v_1, &{\rm if} \,\, v \in \overline{\Pi_-(s,t)},v_1,v_2 <0,\\ 
\psi^\prime(t) \cdot v_2, &{\rm if} \,\, v \in \Pi_+(s,t), v_1,v_2 <0,\\
0, &{\rm if}\,\,{\rm either} \,\, v_1 \geq 0 \,\,{\rm or}\,\, v_2 \geq 0.
\end{cases}
$$
\end{itemize}
\end{lemma}
\begin{proof}
Set
$$
\aligned
\Xi_{n,\delta} =& \, \{ (s,t) \in (-\delta/2,+\infty)^2 : n=(s+\delta)^p > (t+\delta)^q \} \\
&\cup \{ (s,t) \in (-\delta/2,+\infty)^2 : (s+\delta)^p < (t+\delta)^q = n \}\\
& \cup \{ (s,t) \in (-\delta/2,+\infty)^2 :  (s+\delta)^p = (t+\delta)^q \leq n\}.
\endaligned
$$
Clearly, $\Phi_{\delta,n} \in C^1\left( (-\delta/2,+\infty)^2 \setminus \Xi_n\right)$ with bounded first order derivatives. Moreover, for every two different elements in $(-\delta/2,+\infty)^2  $ the connecting line segment intersects $\Xi_n$ at most finitely times. Therefore, $\Phi_{\delta,n}$ is Lipschitz on  $(-\delta/2,+\infty)^2$.

In order to verify that $\Phi_{\delta,n}$ is one-sided Gateaux differentiable, we have to show that
$$
D^+_v \Phi_{\delta,n}(s,t) := \lim_{h \rightarrow 0^+}\frac{\Phi_{\delta,n}((s,t)+hv) -\Phi_{\delta,n}(s,t)}{h}
$$ 
exists for all $(s,t) \in (-\delta/2,+\infty)^2$ and $v \in \R^2$. We will only consider the case when $(s+\delta)^p =(t+\delta)^q \ne n$. The other cases are simpler and will not be written down here. 

Since $g$ is convex, for all $(x,y) \in (s,t) +\overline{\Pi_-(s,t)}$, with $y \ne t$,  we have $g(x) >y$. Therefore, for all $v \in \overline{\Pi_-(s,t)}$ we get
$$
(s,t)+h v \in \{(x,y) \in \R^2 : y <g(x)\},
$$
for any $h >0$. Hence,  
$$
\lim_{h \rightarrow 0^+} \frac{\Phi_{\delta,n}((s,t)+hv)-\Phi(s,t)}{h} = \lim_{h \rightarrow 0^+} \frac{\varphi(s+hv_1)-\varphi(s)}{h} = \varphi^\prime(s) \cdot v_1.
$$
On the other hand, if $v \in \Pi_+(s,t)$ there exists $h_0 \in (0,1)$ such that 
$$
(s,t)+h v \in \{(x,y) \in \R^2 : y >g(x)\},
$$
for all $h \in (0,h_0)$.  Thus, 
$$
\lim_{h \rightarrow 0^+} \frac{\Phi_{\delta,n}((s,t)+hv)-\Phi(s,t)}{h} = \lim_{h \rightarrow 0^+} \frac{\psi(t+hv_2)-\psi(t)}{h} = \psi^\prime(t) \cdot v_2.
$$
We conclude by observing that 
$$
\varphi^\prime(s) v_1 = \psi^\prime(t) v_2,
$$
for all $(s+\delta)^p=(t+\delta)^q$ and $v \in \Pi_0(s,t)$.
\end{proof}
\begin{lemma}\label{l: uvn in sob}
Assume that $\oV$ and $\oW$ are as in the statement of Proposition~\ref{p: 2var Lip}.
Let $p \geq 2$, $q=p/(p-1)$ and $n \in \N$. Suppose that $u \in \oV$ and $v \in \oW$ are such that
$$
  |v|^{q-2}  |\nabla v|^2 \mathds{1}_{\{v \ne 0 \}}\in L^1(\Omega).
$$
Then 
\begin{enumerate}[label=\textnormal{(\roman*)}]
\item \label{i: uvn in sob} $u \left(\max\{ |u|^{p/2-1}, |v|^{1-q/2}\} \wedge n \right)  \in\oV$;
\item \label{i: eq qo curve n} almost everywhere on $\{n^{2p/(p-2)}>|u|^p=|v|^q$\} we have
$$
\aligned
\left(\frac{p}{2}-1\right)\frac{u}{|u|} \Re\left(\frac{\overline{u}}{|u|} \nabla u \right)   \mathds{1}_{\{0 <|u|^{p/2-1} < n\}}=  \left(1-\frac{q}{2}\right)\frac{u}{|v|} \Re\left(\frac{\overline{v}}{|v|} \nabla v \right)  \mathds{1}_{\{0 <|v|^{1-q/2} < n\}};
\endaligned
$$
\item \label{i: grad uvn}
$$
\aligned
\nabla\left[u \left(\max\{ |u|^{p/2-1}, |v|^{1-q/2}\} \wedge n \right) \right] = z_n,
\endaligned
$$
where
\begin{equation}
\label{eq: def zn}
\aligned
z_n:= &\, \left( \max\{|u|^{p/2-1},|v|^{1-q/2}\} \wedge n\right) \nabla u \\
&+
\begin{cases}
\left(\frac{p}{2}-1\right)|u|^{p/2-1} \frac{u}{|u|} \Re\left(\frac{\overline{u}}{|u|} \nabla u \right)  \mathds{1}_{\{0 <|u|^{p/2-1} < n\}}, &{\rm if }\,\, |u|^p \geq |v|^q,\\
  \left(1-\frac{q}{2}\right)|v|^{1-q/2}  \frac{u}{|v|} \Re\left(\frac{\overline{v}}{|v|} \nabla v \right)  \mathds{1}_{\{0 <|v|^{1-q/2} < n\}}, &{\rm if }\,\, |u|^p \leq |v|^q.
\end{cases}
\endaligned
\end{equation}
\end{enumerate}
\end{lemma}
\begin{proof}
Fix $n \in \N$. For all $\delta \in (0,1)$ define  $\Psi, \Psi_{\delta}: \R^2 \times \R^2 \rightarrow \R^2$ by
$$
\aligned
\Psi(\zeta,\eta) &= \zeta\cdot  \Phi_{n}(|\zeta|,|\eta|), \\
\Psi_{\delta}(\zeta,\eta) &= \zeta\cdot  \Phi_{\delta,n}(|\zeta|,|\eta|).
\endaligned
$$
By combining Lemma~\ref{l: gat diff} with Theorem~\ref{t: chain rule G}, we obtain $\Phi_{\delta,n}(|u|,|v|) \in W^{1,2}_{\rm loc}(\Omega)$ and
\begin{equation}
\label{eq: d+ d-}
\partial_j[\Phi_{\delta,n}(|u|,|v|)] = D^+_{(\partial_j|u|,\partial_j|v|)} \Phi_{\delta,n}(|u|,|v|) = D^-_{(\partial_j|u|,\partial_j|v|)} \Phi_{\delta,n}(|u|,|v|),
\end{equation}
for all $j \in \{1,\cdots, d\}$ and $\delta \in (0,1)$. In particular, from \eqref{eq: p mm}, \eqref{eq: d+ d-} and Lemma~\ref{l: gat diff}, for all $\delta \in (0,1)$ we deduce that
\begin{equation}
\label{eq: grad+ grad-}
\aligned
\left(\frac{p}{2}-1\right) &(|u|+\delta)^{p/2-2}\nabla |u| \mathds{1}_{\{(|u|+\delta)^{p/2-1}<n\}}=\left(1-\frac{q}{2}\right) (|v|+\delta)^{-q/2}\nabla |v| \mathds{1}_{\{(|v|+\delta)^{1-q/2}<n\}}
\endaligned
\end{equation}
almost everywhere on $\{n^{2p/(p-2)} >(|u|+\delta)^p=(|v|+\delta)^q \}$ and
\begin{equation}
\label{eq: grad Phidn}
\aligned
\nabla&[\Phi_{\delta,n}(|u|,|v|)] \\
&= 
\begin{cases}
\left(\frac{p}{2}-1\right)\frac{(|u|+\delta)^{p/2-1}}{|u|+\delta} \nabla|u| \mathds{1}_{\{ (|u|+\delta)^{p/2-1} <n\}}, &{\rm if}\,\, (|u|+\delta)^p \geq (|v|+\delta)^q,\\
\left(1-\frac{q}{2}\right)\frac{(|v|+\delta)^{1-q/2}}{|v|+\delta} \nabla |v| \mathds{1}_{\{(|v|+\delta)^{1-q/2} <n\}}, &{\rm if} \,\,(|u|+\delta)^p \leq (|v|+\delta)^q.
\end{cases}
\endaligned
\end{equation}
Item \ref{i: eq qo curve n} follows now  by the identity
$$
\nabla |u| = \Re \left( \frac{\overline{u}}{|u|} \nabla u \right) \mathds{1}_{\{u \ne 0\}}
$$
and by sending $\delta \rightarrow 0$ in \eqref{eq: grad+ grad-}.

Let us now prove item~\ref{i: uvn in sob}, that is, $\Psi(u,v) \in \oV$. The function $\Psi_\delta$ is Lipschitz continuous and $\Psi_{\delta}(0)=0$ for all $\delta \in (0,1)$. Therefore, $\Psi_{\delta}(u,v) \in \oV$ by Proposition~\ref{p: 2var Lip}. Moreover, for all $\delta \in (0,1)$ we have
\begin{equation}
\label{eq: domination delta}
|\Psi_{\delta}(u,v)| \leq n |u| \in L^2(\Omega).
\end{equation}
Thus, from Lebesgue's dominated convergence theorem we obtain
\begin{equation}
\label{eq: conv norme delta}
\|\Psi_\delta(u,v) \|_{L^2} \rightarrow \| \Psi(u,v) \|_{L^2},
\end{equation} 
and 
\begin{equation}
\label{eq: Psid in Psi}
\Psi_\delta(u,v) \rightarrow \Psi(u,v) \in L^2(\Omega)
\end{equation}
in $\cD^\prime(\Omega)$, as $\delta \rightarrow 0$. From \eqref{eq: Psid in Psi}, the density of $C_c^\infty(\Omega)$ in $L^2(\Omega)$ and \eqref{eq: domination delta} we deduce that  
\begin{equation}
\label{eq: weak conv tronc delta}
\Psi_\delta(u,v)  \rightharpoonup \Psi(u,v)
\end{equation}
in $L^2(\Omega)$, as $\delta \rightarrow 0$.  By combining \eqref{eq: conv norme delta} and \eqref{eq: weak conv tronc delta} we get 
\begin{equation}
\label{eq: strong conv tronc delta}
\Psi_\delta(u,v)  \rightarrow \Psi(u,v)
\end{equation}
strongly in $L^2(\Omega)$, as $\delta \rightarrow 0$. Moreover, \eqref{eq: grad Phidn} and the product rule imply the existence of a positive constant $C$, not depending on $\delta$, such that
\begin{equation}
\label{eq: domination grad delta}
|\nabla[\Psi_\delta(u,v)]| \leq C \left( |\nabla u| + |v|^{q/2-1}|\nabla v| \mathds{1}_{\{v \ne 0\}} \right),
\end{equation}
which belongs to $L^2(\Omega)$ by the assumptions on $u$ and $v$. Thus, 
\begin{equation}
\label{eq: bound delta}
\| \Psi_\delta(u,v) \|_{\oV} \leqsim \| u \|_{W^{1,2}} + \| |v|^{q/2-1} |\nabla v|  \mathds{1}_{\{v \ne 0\}} \|_{L^2} < \infty,
\end{equation}
for all $\delta \in (0,1)$. Therefore, \eqref{eq: strong conv tronc delta}, \eqref{eq: bound delta} and Lemma~\ref{l: sbsq V} yield $ \Psi(u,v) \in \oV$.

Finally, we prove the chain rule in \ref{i: grad uvn}. From \eqref{eq: domination grad delta} and Lebesgue's dominated convergence theorem we get
\begin{equation}
\label{eq: graPsid in zn}
\nabla[\Psi_\delta(u,v)] \rightarrow z_n \in L^2(\Omega)
\end{equation}
in $\cD^\prime(\Omega)$, as $\delta \rightarrow 0$. Therefore, by combining \eqref{eq: Psid in Psi} and \eqref{eq: graPsid in zn} we infer that $\nabla[\Psi(u,v)] = z_n$.
\end{proof}
\begin{proposition}
\label{p : fucnt a tratti in sob}
Assume that $\oV$ and $\oW$ are as in the statement of Proposition~\ref{p: 2var Lip}.
Let $p \geq 2$, $q=p/(p-1)$. Suppose that $u \in \oV$ and $v \in \oW$ are such that
$$
u \in L^p(\Omega), \qquad v \in L^q(\Omega), \qquad \left(|u|^{p-2} + |v|^{2-q}\right) |\nabla u|^2 +  |v|^{q-2}  |\nabla v|^2 \mathds{1}_{\{v \ne 0\}} \in L^1(\Omega).
$$
Then
\begin{enumerate}[label=\textnormal{(\roman*)}]
 \item $u \max\{ |u|^{p/2-1}, |v|^{1-q/2}\} \in \oV$;
\item \label{i: eq qo curve} almost everywhere on $\{|u|^p=|v|^q$\} we have
$$
\aligned
  \left(\frac{p}{2}-1\right)\frac{u}{|u|} \Re\left(\frac{\overline{u}}{|u|} \nabla u \right) \mathds{1}_{\{u \ne 0\}}=    \left(1-\frac{q}{2}\right)\frac{u}{|v|} \Re\left(\frac{\overline{v}}{|v|} \nabla v \right) \mathds{1}_{\{v \ne 0\}};
\endaligned
$$
\item
$$
\aligned
\nabla&\left[u \max\{ |u|^{p/2-1}, |v|^{1-q/2}\} \right] \\
&= 
\begin{cases}
|u|^{p/2-1} \left( \nabla u + \left(\frac{p}{2}-1\right)\frac{u}{|u|} \Re\left(\frac{\overline{u}}{|u|} \nabla u \right) \mathds{1}_{\{u \ne 0\}} \right), &{\rm if }\,\, |u|^p \geq |v|^q,\\
|v|^{1-q/2} \left( \nabla u +  \left(1-\frac{q}{2}\right)\frac{u}{|v|} \Re\left(\frac{\overline{v}}{|v|} \nabla v \right) \mathds{1}_{\{v \ne 0\}} \right), &{\rm if }\,\, |u|^p \leq |v|^q.
\end{cases}
\endaligned
$$
\end{enumerate}
\end{proposition}
\begin{proof}
Item~\ref{i: eq qo curve} follows by Lemma~\ref{l: uvn in sob}\ref{i: eq qo curve n}.

Let $n \in \N$. Recall the definition of $z_n$ in \eqref{eq: def zn}  and define
$$
\aligned
w &= u \max\{ |u|^{p/2-1}, |v|^{1-q/2}\}, \\
w_n &= u \left(\max\{ |u|^{p/2-1}, |v|^{1-q/2}\} \wedge n \right), \\
z &=\begin{cases}
|u|^{p/2-1} \left( \nabla u + \left(\frac{p}{2}-1\right)\frac{u}{|u|} \Re\left(\frac{\overline{u}}{|u|} \nabla u \right) \mathds{1}_{\{u \ne 0\}} \right), &{\rm if }\,\, |u|^p \geq |v|^q,\\
|v|^{1-q/2} \left( \nabla u +  \left(1-\frac{q}{2}\right)\frac{u}{|v|} \Re\left(\frac{\overline{v}}{|v|} \nabla v \right) \mathds{1}_{\{v \ne 0\}} \right), &{\rm if }\,\, |u|^p \leq |v|^q.
\end{cases}
\endaligned
$$
By Lemma~\ref{l: uvn in sob}\ref{i: uvn in sob}, \ref{i: grad uvn}, $w_n \in \oV$ and $\nabla w_n =z_n$ for all $n \in \N$.
Moreover, Lemma~\ref{l: uvn in sob}\ref{i: grad uvn} and the assumptions on $u$ and $v$ give
\begin{equation}
\label{eq: domin wn e gradwn}
\aligned
|w_n| &\leq |w| \leq \max\{ |u|^{p/2}, |v|^{q/2}\} \in L^2(\Omega),\\
|\nabla w_n| &\leq |z| \leq \left(|u|^{p/2-1}+|v|^{1-q/2} \right)|\nabla u|\mathds{1}_{\{u \ne 0\}} + |v|^{q/2-1}|\nabla v| \mathds{1}_{\{v\ne 0\}} \in L^2(\Omega),
\endaligned
\end{equation}
for all $n \in \N$. Thus, as in the proof of  Lemma~\ref{l: uvn in sob}, we can prove that $w_n \rightarrow w$ strongly in $L^2(\Omega)$ and $(w_n)_{n \in \N}$ is bounded in $\oV$. Therefore, Lemma~\ref{l: sbsq V} yields $ \Psi(u,v) \in \oV$. 

Moreover, from \eqref{eq: domin wn e gradwn} and Lebesgue's dominated convergence theorem we have
\begin{equation*}
\aligned
w_n &\rightarrow w \in L^2(\Omega), \\
\nabla w_n &\rightarrow z \in L^2(\Omega),
\endaligned
\end{equation*}
in $\cD^\prime(\Omega)$, as $n \rightarrow \infty$. Hence, $\nabla w = z$.
\end{proof}

 In Section~\ref{s: conv bell} we will provide a pointwise lower estimate of the generalized Hessian of $\cQ$. To this purpose, it is useful to introduce the following functions.
   
 For any $(X,Y) \in \C^d \times \C^d$ we define the nonnegative functions $b_p[\,\cdot\,; (X,Y)]$ on $\C^2 \setminus \{(\zeta,\eta) \in \C^2 \, : \, \eta=0\} $, $g_p[\,\cdot\,;X]$ on $\C$ and $h_p[\,\cdot\,; (X,Y)]$ on $\C^2$ by the following rules
\begin{equation}
\label{e : def gp e bp}
\aligned
b_p[(\zeta,\eta); (X,Y)] =&\,|\eta|^{2-q}|X|^2 + \left( 1-\frac{q}{2}\right)^2 |\zeta|^2|\eta|^{-q}||\Re( e^{-i\arg\eta}Y)|^2  \\
&+ 2 \left( 1-\frac{q}{2}\right) |\zeta||\eta|^{1-q}\sk{ \Re (e^{-i\arg\zeta}X)}{ \Re (e^{-i\arg\eta}Y)}, \\
g_p[\zeta; X] =& \, \frac{p}{2} |\zeta|^{p-2} \left( \frac{p}{2} |\Re( e^{-i\arg\zeta}X)|^2 + \frac{2}{p} |\Im( e^{-i\arg\zeta}X)|^2 \right),\\
h_p[(\zeta,\eta); (X,Y)] =&
\begin{cases}
    b_p[(\zeta,\eta);(X,Y)] & \, |\zeta|^p < |\eta|^q,\\
    g_p[\zeta; X] &\, |\zeta|^p \geq|\eta|^q.
\end{cases}
\endaligned
\end{equation}
For any $X,Y \in \C^d$, the function $b_p[\,\cdot\,;(X,Y)]$ is continuous in $\C^2 \setminus \{\eta =0\}$, while $g_p[\,\cdot\,;X]$ is continuous in all $\C$. 

The definitions of $g_p$ and $b_p$ are motivated by the fact that 
\begin{equation}
\label{e : hp equal Fp}
|\nabla[u \max\{|u|^{p/2-1},|v|^{1-q/2}\}]|^2 =
\begin{cases}
 b_p[(u,v);(\nabla u, \nabla v)] \mathds{1}_{\{v\ne 0\}} & \, |u|^p \leq |v|^q,\\
    g_p[u; \nabla u] &\, |u|^p \geq|v|^q;
\end{cases}
\end{equation}
for any $u, v$ as in the assumption of Proposition~\ref{p : fucnt a tratti in sob}.

\section{Generalized convexity of the Bellman function}\label{s: conv bell}
As explained in Section~\ref{s: hfmon}, at a certain point, we need the estimate \eqref{eq: strong belm conv f} in order to apply the heat-flow method and prove the bilinear inequality \eqref{eq: N bil}. From \eqref{eq: grad p leq HpI 2}, applied with $r=p$ and $r=q=p/(p-1)$, and \eqref{e : hp equal Fp}, it follows that to deduce \eqref{eq: strong belm conv f} it suffices to establish a lower bound on the generalized Hessian of $\cQ$ on $\C^2 \times \C^{2d}$ of the form
$$
\aligned
H_{\cQ}^{(A,B)}[(\zeta,\eta); (X,Y)]  \geq& \,  \tau(\zeta,\eta) |X|^2 + \tau(\zeta,\eta)^{-1}|Y|^2 \\
&+ \mu \left( \frac{pq}{4}H_{F_p}^{I_d}[\zeta,X] +  2 \delta  h_p[(\zeta,\eta);(X,Y)] \right) \\
&+  \gamma [ q+  (2-q)\delta ] \frac{p}{4}H_{F_q}^{I_d}[\eta,Y],
\endaligned
$$
whenever $A-\mu (pq/4)I_d, B-\gamma (pq/4)I_d \in \cA_p(\Omega)$, for some $\mu, \gamma >0$. 

We will prove this result in Theorem~\ref{t : alpha gen strict conv Q}, by choosing the parameter $\delta$ in the definition of the Bellman function $\cQ$ to be sufficiently small and by relying on \cite[Theorem~3.1]{CD-Potentials}. We follow, and adequately modify, the proof of \cite[Theorem~3.1]{CD-Potentials}, which was in turn modeled after the proofs of \cite[Theorem~3]{Dv-Sch}, \cite[Theorem~15]{CD-mult}, \cite[Theorem~5.2]{CD-OU} and \cite[Theorem~5.2]{CD-DivForm}. See also \cite[Theorem~16]{Poggio}.
\smallskip

To enhance clarity, we begin with a lemma.
Clearly, for $|\zeta|^p \leq |\eta|^q$ we have
\begin{equation}
\label{e : hp less Kp}
b_p[\omega;(X,Y)] \leq K_q[\eta; (X,Y)],
\end{equation}
for all $X,Y \in \C^d$, where
\begin{equation*}
K_q[\eta; (X,Y)] := 
|\eta|^{2-q}|X|^2 + 2 \left( 1-\frac{q}{2}\right) |X||Y|  + \left( 1-\frac{q}{2}\right)^2 |\eta|^{q-2}|Y|^2.
\end{equation*}

\begin{lemma}
\label{l : aux lemma conv Bell}
Let  $p \geq 2$, $q=p/(p-1)$ and $\mu,\gamma >0$. For all $\delta \in (0,1)$ consider the positive constants $C(\delta)$ and $D(\delta)$ of Remark~\ref{r : CD delta}. Then there exists $\delta_0 \in (0,1)$ such that for all $\delta \in (0,\delta_0)$ we have $\widetilde{C}(\delta) >0$ such that
\begin{equation*}
\aligned
2 C(\delta) \biggl(& D(\delta)|\eta|^{2-q}|X|^2 + D(\delta)^{-1} |\eta|^{q-2} |Y|^2 \biggr) +  \frac{pq}{4} H_{F_2 \otimes F_{2-q}}^{(\mu I_d,  \gamma I_d)}[(\zeta,\eta);(X,Y)]\\
\geq&\, \widetilde{C}(\delta)  \biggl(|\eta|^{2-q}|X|^2 +  |\eta|^{q-2} |Y|^2 \biggr),\\
&+ \gamma (2-q)\frac{p}{4}H_{F_q}^{I_d}[\eta;Y] + 2\mu K_q[\eta;(X,Y)],
\endaligned
\end{equation*}
for all $|\zeta|^p < |\eta|^q$ and $X,Y \in \C^d$.
\end{lemma}
\begin{proof}
For all $\delta \in (0,1)$, $|\zeta|^p < |\eta|^q$ and $X,Y \in \C^d$, denote 
$$
\aligned
L_\delta[\omega; (X,Y)] =&\, 2 C(\delta) \biggl( D(\delta)|\eta|^{2-q}|X|^2 + D(\delta)^{-1} |\eta|^{q-2} |Y|^2 \biggr) \\
&+  \frac{pq}{4} H_{F_2 \otimes F_{2-q}}^{(\mu I_d,  \gamma I_d)}[(\zeta,\eta);(X,Y)]\\
&- \gamma(2-q)\frac{p}{4}H_{F_q}^{I_d}[\eta;Y] - 2\mu K_q[\omega;(X,Y)].
\endaligned
$$ 
We want to prove the existence of $\delta_0\in (0,1)$ such that for all $\delta \in (0,\delta_0)$ we have $\widetilde{C}(\delta)>0$ such that
\begin{equation}
\label{e : est of L}
L_\delta[\omega; (X,Y)] \geq \widetilde{C}(\delta)  \biggl(|\eta|^{2-q}|X|^2 +  |\eta|^{q-2} |Y|^2 \biggr),
\end{equation}
for any $\omega = (\zeta, \eta) \in \{|\zeta|^p < |\eta|^q\}$ and $X,Y \in \C^d$.

By combining \cite[Corollary~5.12]{CD-DivForm} applied with $A=\mu I_d$ and $B=\gamma I_d$ and the inequality
$$
H_{F_q}^{I_d}[\eta,Y] \leqsim_q  |\eta|^{q-2}|Y|^2, \quad \eta \in \C\setminus \{0\}, Y \in \C^d,
$$ 
we obtain,  for some constant $\Gamma = \Gamma(p, \mu, \gamma) \in \R$,
$$
\aligned
L_\delta[\omega;(&X,Y)] \\
\geq 2 \cdot \Biggl[& \left(C(\delta)D(\delta)+ \mu \left(\frac{pq}{4} -1\right)\right) |\eta|^{2-q}|X|^2 - \max\{\mu,\gamma\}(2-q)\left( \frac{pq}{2} + 1 \right) |X||Y| \\
&+ \left(C(\delta)D(\delta)^{-1} -  \Gamma  \right) |\eta|^{q-2}|Y|^2 \Biggr],
\endaligned
$$
for all $\omega = (\zeta, \eta) \in \{|\zeta|^p < |\eta|^q\}$ and $X,Y \in \C^d$. From \eqref{eq: cons delta 0} and the fact that $(pq)/4 \geq 1$, there exists $\delta_0 \in (0, 1)$ such that
$$
\aligned
4 \left(C(\delta)D(\delta)+ \mu \left(\frac{pq}{4} -1\right)\right) \left(C(\delta)D(\delta)^{-1} -  \Gamma  \right)  > \left[  \max\{\mu,\gamma\}(2-q)\left( \frac{pq}{2} + 1 \right)\right]^2,
\endaligned
$$
for all $\delta \in (0,\delta_0)$. Therefore, \cite[Corollary~3.4]{CD-Potentials} applied for all $\delta \in (0,\delta_0)$ implies the existence of $\widetilde{C}(\delta) >0$ for which \eqref{e : est of L} holds.
\end{proof}
\begin{theorem}
\label{t : alpha gen strict conv Q}
Choose $p \geq 2$ and $A, B \in \cA_p(\Omega)$. Suppose that there exist $\mu, \gamma  >0$ such that  $A - \mu (pq/4)I_d, B - \gamma (pq/4)I_d \in \cA_p(\Omega)$. Then there exists a continuous function $\tau : \C^2 \rightarrow [0,+\infty)$ such that $\tau^{-1}=1/\tau$ is locally integrable on $\C^2 \setminus \{(0,0)\}$, and $\delta \in (0,1)$ such that $\cQ=\cQ_{p,\delta}$ as in \eqref{eq: N Bellman} admits the following property:
\begin{itemize}
\item
there exists $\widetilde{C}>0$ such that for any $\omega =(\zeta,\eta) \in \C^2 \setminus \Upsilon$, $X,Y \in \C^d$, and  a.e. $x \in \Omega$, we have
$$
\aligned
H_{\cQ}^{(A(x),B(x))}[\omega; (X,Y)]  \geq& \,\widetilde{C} ( \tau |X|^2 + \tau^{-1}|Y|^2) \\
&+ \mu \left( \frac{pq}{4}H_{F_p}^{I_d}[\zeta,X] +  2 \delta  h_p[\omega;(X,Y)] \right) \\
&+  \gamma [ q+  (2-q)\delta ] \frac{p}{4}H_{F_q}^{I_d}[\eta,Y].
\endaligned
$$
\end{itemize}
The implied constant depends on $A,B, p$, $\mu$ and $\gamma$, but not on the dimension $d$.

We may take $\tau(\zeta,\eta)= \max\{|\zeta|^{p-2},|\eta|^{2-q}\}$. 
\end{theorem}
\begin{proof}
Since $A- \mu(pq/4)I_d, \, B- \gamma(pq/4)I_d \in \cA_p(\Omega)$, by \cite[Theorem~3.1]{CD-Potentials} and Remark~\ref{r : CD delta} there exists $\delta_0 \in (0,1)$ such that for all $\delta \in (0,\delta_0)$ we have $C=C(\delta)>0$ 
 such that, a.e. $x \in \Omega$,
\begin{equation}
\label{e : hess AB hess II}
\aligned
H_{\cQ}^{(A,B)}[\omega; (X,Y)] \geq 2 \delta C (\rho |X|^2 + \rho^{-1} |Y|^2) + \frac{pq}{4} H_\cQ^{(\mu I_d,\gamma I_d)}[\omega; (X,Y)],
\endaligned
\end{equation}
for all $\omega=(\zeta,\eta) \in \C^2 \setminus \Upsilon$ and $X,Y \in \C^d$. Here $C$ is the constant of Remark~\ref{r : CD delta} and $\rho$ is the function defined in \eqref{e: def tauuu}.

First assume that $|\zeta|^p >|\eta|^q >0$. Then, by \eqref{e : est hess Fp wrt I} we have for all $\delta \in (0,1)$
$$
\aligned
\frac{pq}{4} H_{\cQ}^{(\mu  I_d,  \gamma I_d)}[&(\zeta,\eta);(X,Y)] \\
  =& \, \mu(p+2\delta)\frac{q}{4} H_{F_p}^{I_d}[\zeta,X] + \gamma [q+ \delta(2-q) ]\frac{p}{4}H_{F_q}^{I_d}[\eta,Y] \\
\geq & \,\mu\frac{pq}{4} H_{F_p}^{I_d}[\zeta,X]  + 2\delta \mu \frac{p}{2}|\zeta|^{p-2}\left( \frac{p}{2} |\Re( e^{-i\arg\zeta}X)|^2 + \frac{2}{p} |\Im( e^{-i\arg\zeta}X)|^2 \right)\\
&+ \gamma[q+ \delta(2-q) ]\frac{p}{4}H_{F_q}^{I_d}[\eta,Y].
\endaligned
$$
Hence, from \eqref{e : hess AB hess II} and the definitions of $\rho$ and $h_p$, in this region we obtain
$$
\aligned
H_{\cQ}^{(A,B)}[(\zeta,\eta); (X,Y)] \geq& \,\,  2\delta C (p-1)^{-1} \left(|\zeta|^{p-2}|X|^2 + |\zeta|^{2-p}|Y|^2\right) \\
&+ \mu \left( \frac{pq}{4}H_{F_p}^{I_d}[\zeta,X] +  2 \delta  h_p[\omega;(X,Y)] \right) \\
&+  \gamma [ q+  (2-q)\delta ] \frac{p}{4}H_{F_q}^{I_d}[\eta,Y],
\endaligned
$$
for all $\delta \in (0,\delta_0)$.

Suppose now that $|\zeta|^p <|\eta|^q$. Then $\rho(\zeta,\eta) = D(\delta) |\eta|^{2-q}$, where  $D(\delta)$ is the constant of Remark~\ref{r : CD delta}. In view of \eqref{e : hp less Kp} it suffices to prove the inequality with $K_q$ instead of $h_p$. From the definition of the Bellman function $\cQ$ we have
$$
\aligned
\ H_{\cQ}^{(\mu I_d,  \gamma I_d)}&[(\zeta,\eta);(X,Y)] = \mu   H_{F_p}^{I_d}[\zeta;X]+  \gamma   H_{F_q}^{I_d}[\eta;Y] + \delta H_{F_2 \otimes F_{2-q}}^{(\mu  I_d,  \gamma I_d)}[(\zeta,\eta);(X,Y)].
\endaligned
$$
Therefore, we conclude by combining \eqref{e : hess AB hess II} and Lemma~\ref{l : aux lemma conv Bell}.
\end{proof}

\subsection{Regularization of $\cQ$}
Denote by $*$ the convolution in $\R^4$ and let $(\varphi_\nu)_{\nu >0}$ be a nonnegative, smooth, and compactly supported approximation of the identity on $\R^4$. Explicitly, $\varphi_\nu(y) = \nu^{-4} \varphi(y/\nu)$, where $\varphi$ is smooth, nonnegative, radial, of integral $1$, and supported in the closed unit ball in $\R^4$. If $\Phi : \C^2 \rightarrow \R$, define $\Phi * \varphi_\nu = (\Phi_\cW * \varphi_\nu) \circ \cW_{2,1} : \C^2 \rightarrow \R$. Explicitly, for $\omega \in \C^2$, 
\begin{equation}
\label{e : compl convol}
\aligned
(\Phi * \varphi_\nu)(\omega) &= \int_{\R^4} \Phi_\cW(\cW_{2,1}(\omega)-z) \varphi_\nu(z) \wrt z \\
&= \int_{\R^4} \Phi( \omega - \cW^{-1}_{2,1}(z)) \varphi_\nu(z) \wrt z.
\endaligned
\end{equation}
The following theorem is modeled after \cite[Corollary~5.5]{CD-DivForm}. See also \cite[Theorem~4]{Dv-Sch}, \cite[Theorem~3.5]{CD-Potentials} and \cite[Corollary~21]{Poggio}.
\begin{theorem}
\label{t : gen conv reg bell func}
Choose $p \geq 2$ and $A, B \in \cA_p(\Omega)$. Suppose that there exist $\mu, \gamma  >0$ such that  $A - \mu (pq/4)I_d, B - \gamma (pq/4)I_d \in \cA_p(\Omega)$. Let $\delta \in (0,1)$, $\widetilde{C}>0$ and function $\tau : \C^2 \rightarrow (0, \infty)$ be as in Theorem~\ref{t : alpha gen strict conv Q}. Then for $\cQ = \cQ_{p,\delta}$ and any $\omega=(\zeta,\eta) \in \C^2$ we have, for a.e. $x \in \Omega$ and every $(X,Y) \in \C^d \times \C^d$,
$$
\aligned
H_{\cQ * \varphi_\nu}^{(A(x),B(x))}[\omega; (X,Y)]  \geq&\,  \widetilde{C} \left( (\tau * \varphi_\nu)(\omega) |X|^2 + (\tau^{-1} * \varphi_\nu)(\omega) |Y|^2 \right) \\
&+ \mu \frac{pq}{4}\left(H_{F_p }^{I_d}[\cdot;X]  *\varphi_\nu\right) (\zeta)	\\	
&+  \gamma [ q+  (2-q)\delta ] \frac{p}{4} \left( H_{F_q}^{I_d}[\cdot;Y] *\varphi_\nu \right) (\eta) \\
&+ 2 \delta \mu (h_p[\, \cdot \, ;(X,Y)] * \varphi_\nu)(\omega),
\endaligned
$$
with the implied constant depending on $A, B, p$, $\mu$ and $\gamma$, but not on the dimension $d$.
\end{theorem}

\begin{proof}
See the proof of \cite[Theorem~3.5]{CD-Potentials}.
\end{proof}
\smallskip

The next two results are auxiliary for Proposition~\ref{p : point conv hp reg} below.
Let $F \in C^1(\R^n; \R)$. Define 
$$
\aligned
E_{\pm}&:=\{y \in \R^n \, : \, \pm F(y) >0\},\\
E_{0}&:=\{y \in \R^n \, : \,  F(y) =0\}, \\
\widetilde{E}_0 &:=\{ y \in E_0 \, : \, D F(y) \ne 0\}.
\endaligned
$$
\begin{lemma}
\label{p: conv dom reg}
Let $y_0 \in\widetilde{E}_0$ and $\varphi \in L^1(\R^n)$. Then,
$$
\lim_{\nu \rightarrow 0} \int_{E_{\pm}} \varphi(y_0 -y) \wrt y = c_{\pm}(y_0),
$$
with $c_+(y_0)+c_-(y_0)=\int_{\R^n}\varphi$. 

More precisely,
$$
c_{\pm}(y_0) = \int_{\Pi_{\pm}(y_0)} \varphi(y_0-y) \wrt y =  \int_{\widetilde{\Pi}_{\mp}(y_0)} \varphi(y) \wrt y,
$$
where
$$
\aligned
\Pi_{\pm}(y_0)& = \{ y \in \R^n \, : \,\pm D F(y_0) \cdot (y-y_0) >0\},\\
\widetilde{\Pi}_{\pm}(y_0) &=  \{ y \in \R^n \, : \,\pm D F(y_0) \cdot y >0\}.
\endaligned
$$
\end{lemma}
\begin{proof}
Observe that
\begin{equation}
\label{eq: a ch var}
 \int_{E_+} \varphi_\nu(y_0-y) \wrt y =  \int_{\R^n} \varphi(-y) \ca_{E_+}(y_0+\nu y) \wrt y.
\end{equation}
For all $y \in \R^n$, define
$$
g_{y_0,y}(\nu) := F(y_0+\nu y), \qquad \nu \in \R.
$$
Notice that $g_{y_0}(0)=0$ implies 
$$
\lim_{\nu \rightarrow 0} \ca_{E_+}(y_0+\nu y) =
\begin{cases}
1, &\text{ if }\, g^\prime_{y_0,y}(0) >0,\\ 
0, &\text{ if } \, g^\prime_{y_0,y}(0) <0.
\end{cases}
$$
But $g^{\prime}_{y_0}(\nu) = D F(y_0+\nu y) \cdot y$. Hence,
$$
\lim_{\nu \rightarrow 0} \ca_{E_+}(y_0+\nu y) =
\begin{cases}
1, &\text{ if } \, D F(y_0) \cdot y >0,\\ 
0, &\text{ if } \, D F(y_0) \cdot y <0,
\end{cases}
$$
and, since $|\{y \in \R^n \, : \,D F(y_0) \cdot y = 0\}| =0$, Lebesgue's dominated convergence theorem and \eqref{eq: a ch var} give
$$
\aligned
\lim_{\nu \rightarrow 0} \int_{E_+} \varphi_\nu(y_0-y) \wrt y &= \int_{\widetilde{\Pi}_{+}(y_0)} \varphi(-y) \wrt y \\
&= \int_{\widetilde{\Pi}_{-}(y_0)} \varphi(y) \wrt y\\
& = \int_{\Pi_+(y_0)} \varphi(y_0-y) \wrt y = c_+(y).
\endaligned
$$
Analogously, $\lim_{\nu \rightarrow 0} \int_{E_+} \varphi_\nu(y_0-y) \wrt y = c_-(y_0)$. Clearly, $c_+(y_0)+c_-(y_0)=\int_{\R^n} \varphi$.
\end{proof}

\begin{corollary}
\label{c: conv dom reg}
Let $g_+,g_-$ be  functions on $\R^n$ which are continuous at $y$ for all $y \in \widetilde{E}_0$. Define
$$
g(y) =
\begin{cases}
g_+(y), &\text{ if } y \in \overline{E_+},\\
g_-(y), &\text{ if } y \in E_-.
\end{cases}
$$
 Suppose that $|E_0|=0$. Let $\varphi \in L^1(\R^n)$ be such that $\varphi \geq 0$, $\int_{\R^n}\varphi=1$, ${\rm supp}\varphi \subseteq B(0,1)$. Then for all $y_0 \in \widetilde{E}_0$  there exists $c_{\pm}(y_0) \geq 0$ such that
\begin{enumerate}
\item $c_+(y_0) + c_-(y_0)=1$,
\item $\lim_{\nu \rightarrow 0}(g *\varphi_\nu)(y_0) = c_+(y_0)g_+(y_0)+ c_-(y_0)g_-(y_0)$.
\end{enumerate}

Moreover, if $\varphi$ is radial then $c_+(y_0)=c_-(y_0)=1/2$.
\end{corollary}
\begin{proof}
Fix $y_0 \in \widetilde{E}_0$. 
As $|E_0|=0$, we have
$$
\aligned
(g *\varphi_\nu)(y_0) =&\, \int_{\R^n} \varphi_\nu(y_0-y) g(y)\wrt y \\
 =& \, \int_{E_+}  \varphi_\nu(y_0-y) g_+(y)\wrt y + \int_{E_-}  \varphi_\nu(y_0-y) g_-(y)\wrt y \\
 = & \, \int_{E_+}  \varphi_\nu(y_0-y) [g_+(y)-g_+(y_0)]\wrt y + g_+(y_0) \int_{E_+}  \varphi_\nu(y_0-y)\wrt y \\
&+ \int_{E_-}  \varphi_\nu(y_0-y) [g_-(y)-g_-(y_0)]\wrt y  + g_-(y_0) \int_{E_-}  \varphi_\nu(y_0-y) \wrt y.
\endaligned
$$
Since
$$
\int_{\pm E} \varphi_\nu(y_0-y) \wrt y \leq \int_{\R^n}\varphi_\nu(y_0-y) \wrt y= 1,
$$
 $g_{\pm}$ are continuous at $y_0$ and ${\rm supp}\varphi_\nu \subseteq B(0,\nu)$, we obtain
$$
\left| \int_{E_{\pm}} \varphi_\nu(y_0-y) [g_{\pm}(y)-g_{\pm}(y_0)]\wrt y \right| \leq \sup_{|y-y_0|<\nu}|g_{\pm}(y)-g_{\pm}(y_0)| \rightarrow 0, \quad \text{ as } \nu \rightarrow 0.
$$
On the other hand, Lemma~\ref{p: conv dom reg} yields
$$
g_{\pm}(y_0) \int_{E_{\pm}}  \varphi_\nu(y_0-y)\wrt y \rightarrow  g_{\pm}(y_0)c_{\pm}(y_0),  \quad \text{ as } \nu \rightarrow 0.
$$

Now suppose that $\varphi$ is radial. Therefore, from Lemma~\ref{p: conv dom reg} we have
\begin{alignat*}{2}
c_{\pm}(y_0) = \int_{\tilde{\Pi}_{\mp}(y_0)} \varphi(y)  \wrt y = \frac{1}{2} \int_{\R^n} \varphi(y) \wrt y = \frac{1}{2}.
\tag*{\qedhere}
\end{alignat*}
\end{proof}

\begin{proposition}\label{p : point conv hp reg}
Let $p \geq2$. Then for all $X,Y \in \C^d$ the following hold:
\begin{enumerate}
    \item for any $\omega=(\zeta,\eta) \in \C^2 \setminus \{|\zeta|^p = |\eta|^q\}$,
    $$
    \lim_{\nu \rightarrow 0} (h_p[\, \cdot \, ;(X,Y)] * \varphi_\nu)(\omega)  = h_p[\omega; (X,Y)];
    $$
    \item for any $\omega=(\zeta,\eta) \in\{|\zeta|^p = |\eta|^q\} \setminus \{(0,0)\}$,
    $$
    \lim_{\nu \rightarrow 0} (h_p[\, \cdot \, ;(X,Y)] * \varphi_\nu)(\omega) = \frac{1}{2} \left(b_p[\omega; (X,Y)]+ g_p[\zeta; X]\right).
    $$
\end{enumerate}
\end{proposition}
\begin{proof}
The first assertion follows by the continuity of $h_p[\, \cdot\,;(X,Y)]$ in $\C^2 \setminus \{|\zeta|^p = |\eta|^q\}$ for all $X,Y \in \C^d$. 

The second one follows by applying Corollary~\ref{c: conv dom reg} with $g_+=(g_p \otimes \textbf{1})_{\cW}$, $g_-=(b_p)_{\cW}$ and $F = (F_p \otimes \textbf{1} - \textbf{1} \otimes F_q)_{\cW}$. Notice that in this case $\widetilde{E}_0 = E_0 \setminus\{0\}$, which gives $|E_0|=|\widetilde{E}_0|=0$ because $\nabla F(\omega) \ne 0$ for all $\omega \in 
\widetilde{E}_0$.
\end{proof}

\section{Proof of the bilinear embedding for potentials with bounded negative part} \label{s: proof b.e. neg bound}
Suppose that $(\oV,\oW)$ satisfies \AssBE. Take $p >1$, $A,B \in \cA_p(\Omega)$ and $V \in \cP(\Omega,\oV),$ and $W \in \cP(\Omega,\oW)$ such that $(A,V) \in \cA\cP_p(\Omega,\oV)$ and $(B,W) \in \cA\cP_p(\Omega,\oW)$. Denote by $q$ the conjugate exponent of $p$. It is enough to consider the case $p \geq 2$.

For $f,g \in (L^p \cap L^q)(\Omega)$ define
\begin{equation}
\label{eq: def flow}
\cE(t)= \int_\Omega \cQ\left( T_t^{A,V,\oV}f, T_t^{B,W,\oV}g \right) \, , \,\, t>0.
\end{equation}

In the case when $V_-=W_-=0$, in \cite{CD-Potentials} the authors supposed the potentials to be bounded in order to study the monotonicity of the flow $\cE$.  
In this case the operator $\oL^{A,V}$ turns out to be the sum of the second-order operator $\oL^{A,0}$ and the multiplication operator $V$. More precisely, $\Dom(\oL^{A,V})=\Dom(\oL^{A,0})$ and for $u \in \Dom(\oL^{A,V})$ we have
$$
\oL^{A,V}u = \oL^{A,0}u + Vu.
$$
The same holds for $B, W$. See \cite[Section~3.3]{CD-Potentials}.

Once the bilinear inequality was proved for bounded potentials, they deduced the general one by approximating the unbounded potentials with their truncations. The key point to apply that argument was the uniform sectoriality (in $n \in \N$) of the operators $\oL^{A,V \wedge n}$, in the sense of \cite[(3.15)]{CD-Potentials}. See \cite[Section~3.4]{CD-Potentials}. 

Let us try to repeat this approach by cutting both the negative and the positive part of the potentials. Let $A \in \cA(\Omega)$ and $\{V_j\}_{j \in J}$ be a family of potentials. Denote by $\oL_j$ the operator associated with the sesquilinear form $\gota_j:=\gota_{A,V_j}$. Uniform sectoriality of such family of operators is guaranteed provided that \eqref{e : above boun a} holds for all $j \in J$ with a constant not depending on the parameter $j$.  More precisely, $\{\oL_j\}_{j\in J}$ is a family of uniformly sectorial operators if  there exists $\alpha >0$  such that
\begin{itemize}
\item for all $j \in J$ there exists $\sigma_j \in [0,1]$ such that
$$
\int_\Omega (V_j)_- |u|^2 \leq \alpha \int_\Omega |\nabla u|^2 + \sigma_j \int_\Omega (V_j)_+ |u|^2,
$$
for all $u \in \oV$;
\item $A -\alpha I_d \in \cA(\Omega)$. 
\end{itemize}

Given $V \in \cP_{\alpha,\sigma}(\Omega)$, it is clear that 
$$
V_n:=V_+- V_- \wedge n \in \cP_{\alpha,\sigma}(\Omega)
$$
for all $n \in\N$. Therefore, if $(A,V) \in \cA\cP(\Omega,\oV)$, the above properties are satisfied by the family of potentials $\{V_n\}_{n\in\N}$.  

Now, if we also try to truncate the positive part, we immediately notice that some problems come out. In fact, fixed $n \in \N$, it is not so clear at which altitude we have to cut the positive part of $V_n$ in order to obtain a family of potentials for which the existence of such uniform constant is guaranteed. 
If we first truncate the positive part  and then the negative one it is even worse. In both cases, the point is that we might need all $V_+$ in order to control either $V_- \wedge n$ or $V_-$.

These observations force us to work with potentials whose positive part is unbounded. On the other hand, as noted before, we are allowed to truncate the negative part.
We will thus first assume that our potentials have a bounded negative part. Once the bilinear embedding is established for this class of potentials, we will then extend the result to potentials with unbounded negative part by approximating them with their truncations, as done in \cite{CD-Potentials}.

In Section~\ref{s: alt proof} we will exhibit an alternative proof of \cite[Theorem~1.4]{CD-Potentials}, studying the monotonicity of $\cE$ directly working with possibly unbounded potentials. In addition to being of independent interest, this new proof represents an auxiliary step in establishing the bilinear embedding in the case where the potentials may have a nontrivial (bounded) negative part. In particular, it will enable us to 
apply the chain rule through Proposition~\ref{p : fucnt a tratti in sob}. In fact, 
we will prove that
\begin{equation}
\label{e : heat flow est pos 2}
\aligned
2 \, \Re \int_\Omega \biggl(   \partial_{\zeta}\cQ(u,v)  \oL^{A,V_+}u &+ \partial_{\eta}\cQ(u,v) \oL^{B,W_+}v \biggr) \\
\geq &  \liminf_{\nu \rightarrow 0} \int_\Omega  H_{\cQ * \varphi_\nu}^{(A,B)}[(u,v); (\nabla u,\nabla v)]  \\
& + 2 \, \int_\Omega   V_+  (\partial_{\zeta}\cQ)(u,v)\cdot u +  W_+ (\partial_{\eta}\cQ )(u,v)\cdot v,
\endaligned
\end{equation}
for all  $u \in \Dom(\gota_{A,V_+,\oV})$, $v \in \Dom(\gota_{B,W_+,\oW})$ such that $u, v, \oL^{A,V_+}u, \oL^{B,W_+}v \in (L^p \cap L^q)(\Omega)$. As shown in \cite[Sections~3.3.1 and 3.3.2]{CD-Potentials}, \eqref{e : heat flow est pos 2} implies
\begin{equation}
\label{e : heat flow est pos 3}
\aligned
2 \, \Re \int_\Omega \biggl(   &\partial_{\zeta}\cQ(u,v)  \oL^{A,V_+}u + \partial_{\eta}\cQ(u,v) \oL^{B,W_+}v \biggr) \\
&\geqsim \int_{\Omega \setminus\{u=0,v=0 \}} \tau(u,v) \left(|\nabla u|^2 + V_+ |u|^2 \right)+ \tau(u,v)^{-1} \left(|\nabla v|^2 + W_+ |v|^2\right),
\endaligned
\end{equation}
with $\tau(u,v)= \max\{|u|^{p-2}, |v|^{2-q}\}$. In particular, \eqref{e : heat flow est pos 3} and the last two estimates of \eqref{eq: N 5}  give
\begin{equation}
\label{eq: uv l1}
|v|^{2-q}|\nabla u|^2 \in L^1(\Omega),
\end{equation}
for all $u \in \Dom(\gota_{A,V_+,\oV})$, $v \in \Dom(\gota_{B,W_+,\oW})$ such that $u, v, \oL^{A,V_+}u, \oL^{B,W_+}v \in (L^p \cap L^q)(\Omega)$. 
 Therefore, if $V_-$ and $W_-$ are bounded, then Proposition~\ref{t : p grad est} and \eqref{eq: uv l1} guarantee that such $u, v$ fall into the assumptions of Proposition~\ref{p : fucnt a tratti in sob}; see Remark~\ref{r: V e Vp}.

\subsection{Potentials with bounded negative part}
We prove now the bilinear embedding assuming that $V_-, W_-$ are {\it essentially bounded}. In that case, $\Dom(\oL^{A,V}) = \Dom(\oL^{A,V_+})$ and for $u \in \Dom(\oL^{A,V})$ we have
\begin{equation}
\label{e : split op under bound}
\oL^{A,V} u = \oL^{A,V_+}u  - V_- u.
\end{equation}
The same holds true for $B, W$. 

\begin{remark}
\label{r: V e Vp}
Clearly, $\Dom(\gota_{A,V,\oV}) = \Dom(\gota_{A,V_+,\oV})$. Hence, from \eqref{e : split op under bound} it follows that
$$
u \in \Dom(\gota_{A,V,\oV}) \quad \text{ and } \quad  u, \oL^{A,V}u \in (L^p \cap L^q)(\Omega),
$$
if and only if 
$$
u \in \Dom(\gota_{A,V_+,\oV}) \quad \text{ and }\quad  u, \oL^{A,V_+}u \in (L^p \cap L^q)(\Omega),
$$
whenever $V_-$ is bounded.
\end{remark}
 
In Proposition~\ref{p: sempl prop} we will show how  we may deduce the bilinear estimate \eqref{eq: N bil}  from \eqref{e : heat flow est pos 2}. 
First, we will start with a reduction, in the same spirit as \cite[Section~6.1]{CD-Mixed} and  \cite[Proposition~7.2]{CDKS-Tril}.
\begin{proposition}
\label{p: sempl prop1}
Suppose that $A,B,V,W, \oV, \oW, p,q$ are as in the formulation of Theorem~\ref{t: N bil} and $V_-, W_- \in L^\infty(\Omega)$. Assume that
\begin{equation}
\label{eq: hfm 1}
\aligned
\int_\Omega  \sqrt{|\nabla u|^2+|V||u|^2}&\sqrt{|\nabla v|^2+|W||v|^2} \leqsim \, \Re \int_\Omega \left(  \partial_{\zeta}\cQ(u,v)  \oL^{A,V}u + \partial_{\eta}\cQ(u,v) \oL^{B,W}v \right),
\endaligned
\end{equation}
for all  $u \in \Dom(\gota_{A,V,\oV})$, $v \in \Dom(\gota_{B,W,\oW})$ such that $u, v, \oL^{A,V}u, \oL^{B,W}v \in (L^p \cap L^q)(\Omega)$. Then \eqref{eq: N bil} holds.
\end{proposition}
\begin{proof}
Let $f, g \in (L^p \cap L^q)(\Omega)$. Let $\cE$ be the flow defined in \eqref{eq: def flow} and define $\gamma : [0, \infty) \rightarrow \C^3$ by 
$$
\gamma_t =\gamma(t) := \left(T_t^{A,V,\oV}f, T_t^{B,W,\oV}g\right).
$$  
As in \cite[Section~6.1]{CD-Mixed}, we can show that  $\cE$ is well defined, continuous on $[0,\infty)$, differentiable on $(0,\infty)$ with a continuous derivative,
\begin{equation}
\label{e: hfm 2}
-\cE^\prime(t) = 2 \, \Re \int_\Omega \left(  \partial_{\zeta}\cQ(\gamma_t)  \oL^{A,V}T_t^{A,V,\oV}f + \partial_{\eta}\cQ(\gamma_t) \oL^{B,W}T_t^{B,W,\oV}g \right)
\end{equation}
and
\begin{equation}
\label{e: hfm 3}
-\int_0^\infty \cE^\prime (t) \, dt \leq \cE(0) \leqsim \|f\|_p+\|g\|_q.
\end{equation}
Analyticity of the semigroups both on $L^p(\Omega)$ and $L^q(\Omega)$ (see Corollary~\ref{c: N analytic sem}) yields $T_t^{A,V,\oV}f \in \Dom(\oL^{A,V}_p) \cap \Dom(\oL^{A,V}_q)$ and $T_t^{B,W,\oW}g \in \Dom(\oL^{B,W}_p) \cap \Dom(\oL^{B,W}_q)$. By consistency of the semigroups and H\"older inequality, we have
$$
\begin{array}{rcccl}
\Dom(\oL^{A,V}_p) \cap \Dom(\oL^{A,V}_q) &\subseteq &\Dom(\oL^{A,V}_2) &\subseteq& \Dom(\gota_{A,V,\oV}),\\
\Dom(\oL^{B,W}_p) \cap \Dom(\oL^{B,W}_q) &\subseteq &\Dom(\oL^{B,W}_2) &\subseteq & \Dom(\gota_{B,W,\oW}).
\end{array}
$$
Therefore, we may apply \eqref{eq: hfm 1} with $(u,v)=(T_t^{A,V,\oV}f,T_t^{B,W,\oW}g)$. Together with \eqref{e: hfm 2} and \eqref{e: hfm 3} we then obtain 
$$
\int_0^\infty \int_\Omega  \sqrt{\left|\nabla T_t^{A,V,\oV}f\right|^2+|V|\left|T_t^{A,V,\oV}f\right|^2}\sqrt{\left|\nabla T_t^{B,W,\oW}g\right|^2+|W|\left|T_t^{B,W,\oW}g\right|^2} \leqsim \|f\|_p^p +\|g\|_q^q.
$$
At this point, \eqref{eq: N bil} follows by replacing $f$ and $g$ with $sf$ and $s^{-1}g$ and minimizing the right-hand side with respect to $s >0$.
\end{proof}
\begin{proposition}
\label{p: sempl prop}
Suppose that $A,B,V,W, \oV, \oW, p,q$ are as in the formulation of Theorem~\ref{t: N bil} and $V_-, W_- \in L^\infty(\Omega)$. Assume that \eqref{e : heat flow est pos 2} is satisfied 
for all  $u \in \Dom(\gota_{A,V_+,\oV})$, $v \in \Dom(\gota_{B,W_+,\oW})$ such that $u, v, \oL^{A,V_+}u, \oL^{B,W_+}v \in (L^p \cap L^q)(\Omega)$.
Then \eqref{eq: N bil} holds.
\end{proposition}
\begin{proof}
By Proposition~\ref{p: sempl prop1} it suffices to prove \eqref{eq: hfm 1} for all $u \in \Dom(\gota_{A,V,\oV})$, $v \in \Dom(\gota_{B,W,\oW})$ such that $u, v, \oL^{A,V}u, \oL^{B,W}v \in (L^p \cap L^q)(\Omega)$.
Given such $u$ and $v$, denote
\begin{equation}
\label{eq: OQ}
\cO(\cQ)(u,v) := 2 \, \Re \int_\Omega \left(  \partial_{\zeta}\cQ(u,v)  \oL^{A,V}u + \partial_{\eta}\cQ(u,v) \oL^{B,W}v \right).
\end{equation}
By \eqref{e : split op under bound}, we get
$$
\aligned
\cO(\cQ)(u,v) =& \, 2\, \Re \int_\Omega  \left(  \partial_{\zeta}\cQ(u,v)  \oL^{A,V_+}u + \partial_{\eta}\cQ(u,v) \oL^{B,W_+}v \right)\\
& -  2 \int_\Omega \left[ V_- (\partial_{\zeta}\cQ)(u,v) \cdot u + W_- (\partial_{\eta}\cQ)(u,v) \cdot v \right].
\endaligned
$$
Hence, from  Remark~\ref{r: V e Vp} we can apply \eqref{e : heat flow est pos 2} and obtain
\begin{equation}
\label{e : flow deriv}
\cO(\cQ)(u,v) \geq I_1 +I_2-I_3,
\end{equation}
where
$$
\aligned
I_1 &:=  \liminf_{\nu \rightarrow 0} \int_\Omega  H_{\cQ * \varphi_\nu}^{(A,B)}[(u,v); (\nabla u,\nabla v)],  \\
I_2 &:=  2 \, \int_\Omega   \left[ V_+ (\partial_{\zeta}\cQ)(u,v) \cdot u + W_+(\partial_{\eta}\cQ)(u,v) \cdot v \right], \\
I_3 &:= 2 \int_\Omega \left[ V_- (\partial_{\zeta}\cQ)(u,v) \cdot u + W_- (\partial_{\eta}\cQ)(u,v) \cdot v \right].
\endaligned
$$
Estimating $I_1, I_2, I_3$ separately, we will prove first that
\begin{equation}
\label{e : first est of E}
\aligned
\cO(\cQ)(u,v)   \geqsim \int_{\Omega \setminus\{u=0,v=0 \}}   \tau(u,v) \left( |\nabla u|^2 +  V_+|u|^2 \right) + \tau^{-1}(u,v) \left(|\nabla v|^2 + W_ + |v|^2 \right),
\endaligned
\end{equation}
and then that
\begin{equation}
\label{e : second est of E}
\aligned
\cO(\cQ)(u,v)  &\geqsim \int_{\Omega \setminus\{u=0,v=0 \}} \tau(u,v)  \cdot V_- |u|^2 + \tau^{-1}(u,v) \cdot W_- |v|^2,
\endaligned
\end{equation}
which, along with the fact that for $w \in W^{1,2}(\Omega)$ we have $\nabla w =0$ almost everywhere on $\{w= 0\}$, imply \eqref{eq: hfm 1}. Recall that $\tau(u,v) = \max\{ |u|^{p-2}, |v|^{2-q}\}$.

To verify both of the above estimates, we will make use of Proposition~\ref{p : fucnt a tratti in sob}, whose assumptions are fulfilled by such functions $u$ and $v$ thanks to Proposition~\ref{t : p grad est}, \eqref{eq: uv l1} and Remark~\ref{r: V e Vp}.
\medskip

\framebox{Proof of \eqref{e : first est of E}}:

We start estimating $I_1$. Recall definition 
\eqref{e : def gp e bp}. 
Set $\omega=(u,v)$ and $\nabla \omega=(\nabla u,\nabla v)$ and denote
$$
G_p(u,v) = u \max\{|u|^{p/2-1},|v|^{1-q/2}\}.
$$
By Theorem~\ref{t : gen conv reg bell func} applied with 
$$
\aligned
\mu&=\alpha(V,\oV),\\
\gamma&= \alpha(W,\oW),
\endaligned
$$
 we obtain 
\begin{equation}
\label{e : first est for I1}
\aligned
I_1 \geq \liminf_{\nu \searrow 0} \int_\Omega& \widetilde{C} \biggl( (\tau * \varphi_\nu)(u,v) |\nabla u|^2 + (\tau^{-1} * \varphi_\nu)(u,v) |\nabla v|^2 \biggr) \\
&+\mu \frac{pq}{4} \left(H_{F_p }^{I_d}[\cdot;\nabla u]*\varphi_\nu\right) (u)	\\	
&+ \gamma [ q+  (2-q)\delta ] \frac{p}{4}\left( H_{F_q}^{I_d}[\cdot;\nabla v]*\varphi_\nu \right) (v)\\
&+ 2 \delta \mu (h_p[\, \cdot \, ; \nabla \omega] * \varphi_\nu)(\omega).
\endaligned
\end{equation}
Proposition~\ref{p : fucnt a tratti in sob} and \eqref{e : hp equal Fp} show that for a.e. $x \in \Omega \cap \{ |u|^p = |v|^q \}$, 
$$
 b_p[\omega;\nabla \omega] \mathds{1}_{\{v\ne 0\}} = g_p[u;\nabla u]  =  h_p[\omega;\nabla\omega] = |\nabla[ G_p(u,v)]|^2.
$$
Therefore, Proposition~\ref{p : point conv hp reg} gives that for almost all $x \in \Omega \setminus \{u=0, v=0\}$,
\begin{equation}
    \label{e : point conv hp reg}
    \lim_{\nu \rightarrow 0} (h_p[\,\cdot\,; \nabla \omega] * \varphi_\nu])(\omega) = |\nabla[G_p(u,v)]|^2. 
\end{equation}
Hence, by combining \eqref{e : first est for I1} and \eqref{e : point conv hp reg} with Fatou's Lemma and the facts that $F_p \in C^2(\C)$, $F_q \in C^2( \C \setminus \{ 0\})$ and $\nabla[G_p(u,v)]=0$ almost everywhere in $\{u=0,v=0\}$, we get
\begin{equation}
\label{e : est I1}
I_1 \geq J(\mu,\gamma) +\widetilde{C} \int_{\Omega \setminus \{u=0,v=0\}}  \tau(u,v) |\nabla u|^2 + \tau^{-1}(u,v) |\nabla v|^2,
\end{equation}
where
$$
\aligned
J(\mu,\gamma) :=  \int_{\Omega }\mu \frac{pq}{4} H_{F_p}^{I_d}[u,\nabla u] + \gamma [ q+  (2-q)\delta ]  \frac{p}{4} H_{F_q}^{I_d}[v,\nabla v] \mathds{1}_{\{v \ne 0\}} + 2 \delta \mu  |\nabla[ G_p(u,v)]|^2. 
\endaligned
$$

Let us now estimate $I_2$. By \cite[Theorem~3.1(ii)]{CD-Potentials},
\begin{equation}
\label{e : est I2}
I_2 \geqsim \int_\Omega \tau(u,v) V_+| u|^2 + \tau^{-1}(u,v) W_+ |v|^2.
\end{equation}

Finally we estimate $I_3$. Lemma~\ref{l : first ord est Bell} gives
\begin{equation}
\label{e : est I_3}
I_3 \leq \int_\Omega pV_-|u|^p + [q + (2-q)\delta] W_-|v|^q+ 2\delta V_- \left|G_p(u,v) \right|^2.
\end{equation}
Recall definition \eqref{eq: beta}. Let $\kappa \in (\sigma,1)$. By combining the fact that $V \in \cP(\Omega,\oV)$ and $W \in \cP(\Omega, \oW)$ with Proposition~\ref{t : p grad est}, Proposition~\ref{p : fucnt a tratti in sob} and Lemma~\ref{l : first ord est Bell} applied with $\sigma/\kappa \in (\sigma,1)$, we get
\begin{equation}
\label{e : est I3}
\aligned
I_3 \leq& \, J(\mu,\gamma) + \sigma \int_\Omega \left[ p V_+|u|^p  +[q+(2-q)\delta] W_+|v|^q +  2\delta V_+ |G_p(u,v)|^2 \right]\\
\leq& \,J(\mu,\gamma) +2 \int_\Omega \sigma \, V_+ (\partial_\zeta \cQ )(u,v) \cdot u + \kappa \,  W_+ (\partial_\eta \cQ )(u,v) \cdot v\\
\leq&\, J(\mu,\gamma) + \kappa I_2.
\endaligned
\end{equation}

The estimates  \eqref{e : est I1} and \eqref{e : est I3} yield, together with \eqref{e : flow deriv},
$$
\aligned
\cO(\cQ)(u,v) \geq (1-\kappa)I_2 + \widetilde{C} \int_{\Omega \setminus \{u=0,v=0\}}  \tau(u,v) |\nabla u|^2 + \tau^{-1}(u,v) |\nabla v|^2. \nonumber 
\endaligned
$$
Therefore, \eqref{e : est I2} and the fact that $\kappa \in (0,1)$ give \eqref{e : first est of E}. 
\medskip

\framebox{Proof of \eqref{e : second est of E}}:

By Proposition~\ref{p : basic prop}\ref{i : inv small sub} there exist $\mu > \alpha(V)$ and $\gamma > \alpha(W)$ 
such that $A - \mu(pq/4)I_d$, $B- \gamma(pq/4)I_d \in \cA_p(\Omega)$. Therefore, by repeating the same argument to get \eqref{e : est I1}, we obtain 
$$
I_1 \geq J(\mu,\gamma).
$$
Moreover, since $\sigma \in [0,1)$, such $\mu$ and $\gamma$ might be chosen such that 
\begin{equation}
\label{eq: alfabetagamma}
\aligned
1 &> \frac{\alpha(V)}{\mu} > \sigma, \\
1 &>\frac{\alpha(W)}{\gamma} > \kappa \cdot \frac{\alpha(W)}{\gamma} > \sigma,
\endaligned
\end{equation}
for some $\kappa \in (0,1)$.
Lemma~\ref{l : first ord est Bell} applied with such $\kappa$ gives
$$
\aligned
I_2  &\geq    \int_\Omega p V_+ |u|^p + \kappa[q+(2-q)\delta]  W_+ |v|^q + 2 \delta V_+ |G_p(u,v)|^2.
\endaligned
$$
Therefore, 
$$
\aligned
I_1+I_2 \geq& \,  J(\mu,\gamma) +  \int_\Omega p V_+ |u|^p + \kappa[q+(2-q)\delta]  W_+ |v|^q + 2 \delta V_+ |G_p(u,v)|^2 \\
=& \,  p \frac{\mu}{\alpha(V)}\left[ \alpha(V) \int_\Omega \frac{q}{4}H_{F_p}^{I_d}[u;\nabla u] +  \frac{\alpha(V)}{\mu} \int_\Omega V_+ |u|^p \right] \\
&+  [q+(2-q)\delta]  \frac{\gamma}{\alpha(W)} \left[ \alpha(W) \int_\Omega\frac{p}{4}H_{F_q}^{I_d}[v;\nabla v]  \mathds{1}_{\{v \ne 0\}} + \kappa \frac{\alpha(W)}{\gamma} \int_\Omega W_+ |v|^q \right]\\
& +  2\delta  \frac{\mu}{\alpha(V)}\left[\alpha(V) \int_\Omega |\nabla[ G_p(u,v)]|^2 +  \frac{\alpha(V)}{\mu}  \int_\Omega V_+ |G_p(u,v)|^2 \right].
\endaligned
$$
By combining  again the fact that $V,W \in \cP(\Omega)$ with Proposition~\ref{t : p grad est}, Proposition~\ref{p : fucnt a tratti in sob} and \eqref{eq: alfabetagamma}, 
we get
$$
\aligned
I_1+I_2 \geq& \, \int_\Omega p \frac{ \mu}{\alpha(V)} V_-| u|^p + [q+(2-q)\delta]\frac{ \gamma}{\alpha(W)} W_-|v|^q  +2\delta  \frac{ \mu}{\alpha(V)} V_- |G_p(u,v)|^2,
\endaligned
$$
which, together with \eqref{e : est I_3} and the facts that $\mu >\alpha(V)$ and $\gamma > \alpha(W)$, gives
$$
\aligned
I_1+I_2 -I_3 
&\geqsim \, I_3  \\
&\geqsim \, \int_\Omega  \tau(u,v)  \cdot V_- |u|^2 + \tau^{-1}(u,v) \cdot W_ - |v|^2,
\endaligned
$$
where in the last inequality we used \cite[Theorem~3.1(ii)]{CD-Potentials}.
\end{proof}

\subsection{Proof of \eqref{e : heat flow est pos 2}: alternative proof of \cite[Theorem~1.4]{CD-Potentials} for unbounded nonnegative potentials} \label{s: alt proof}\label{ss: proof of 58}
In this section we will prove \eqref{e : heat flow est pos 2} and, consequently, the bilinear embedding by means of Proposition~\ref{p: sempl prop}.

Recall notation \eqref{eq: OQ}. Since \eqref{e : heat flow est pos 2} only involves the positive part of the potentials $V$ and $W$, we will assume that $V_-=W_-=0$ and prove 
\begin{equation}
\label{e : low est}
\aligned
\cO(\cQ)(u,v) \geq& \,  \liminf_{\nu \rightarrow 0} \int_\Omega  H_{\cQ * \varphi_\nu}^{(A,B)}[(u,v); (\nabla u,\nabla v)]  \\
& + 2 \, \int_\Omega   V (\partial_{\zeta}\cQ)(u,v)\cdot u +  W (\partial_{\eta}\cQ )(u,v)\cdot v.
\endaligned
\end{equation}
for all  $u \in \Dom(\gota_{A,V,\oV})$, $v \in \Dom(\gota_{B,W,\oW})$ such that $u, v, \oL^{A,V}u, \oL^{B,W}v \in (L^p \cap L^q)(\Omega)$.

Such estimate was already proved in \cite[Sections~3.3]{CD-Potentials} for bounded potentials. 
In the right-hand side of \eqref{e : low est}, the matrices appear independently of the potentials. This observation suggests that, in the expression for $\cO(u,v)$, we should likewise seek to separate the matrices from the potentials.
Given the generality of our setting, where the potentials are unbounded, such separation cannot be achieved directly within the operators themselves. Instead, it is more convenient to work with sesquilinear forms, where matrices and potentials naturally appear as separate terms. To transition to this framework, an integration by parts argument, as formulated in \eqref{eq: ibp}, is required. To this end, we will employ the sequence  $(\cR_{n,\nu})_{n\in\N, \nu\in(0,1)}$, following the approach of Carbonaro and Dragi\v{c}evi\'c in the case where $V=W=0$  \cite{CD-Mixed}.  

Analogously to \cite[(37) and (38)]{CD-Mixed}, respectively, we would like to show 
\begin{equation}
\label{eq: OQ e OR}
\aligned
\cO(\cQ)(u,v) =&\, \lim_{\nu \rightarrow 0} \lim_{n \rightarrow \infty} \cO(\cR_{n,\nu})(u,v),\\
\cO(\cR_{n,\nu})(u,v) = &\, \int_\Omega  H_{\cR_{n,\nu}}^{(A,B)}[(u,v); (\nabla u,\nabla v)]  \\
& +  2\int_\Omega \,  V \,\Re \left( u\cdot (\partial_{\zeta}\cR_{n,\nu})(u,v) \right) +  W \,\Re \left( v \cdot (\partial_{\eta}\cR_{n,\nu})(u,v) \right).
\endaligned
\end{equation}
In this way, thanks to \cite[Proposition~13 \& Theorem~16]{CD-Mixed}, the fact that $\cQ \in C^1(\C^2)$ and Fatou's Lemma (twice), if we had
\begin{equation}
\label{eq: Fatou 2}
\aligned
\Re \left( \zeta \cdot \partial_{\zeta}\cR_{n,\nu}(\zeta,\eta) \right) &\geq 0, \\
\Re \left( \eta \cdot \partial_{\eta}\cR_{n,\nu}(\zeta,\eta) \right) &\geq 0, \\
\Re \left( \zeta \cdot \partial_{\zeta}(\cQ * \varphi_\nu)(\zeta,\eta) \right) &\geq 0, \\
\Re \left( \eta \cdot \partial_{\eta}(\cQ * \varphi_\nu)(\zeta,\eta) \right) &\geq 0, \\
\endaligned
\end{equation}
for all $n \in \N$, $\nu \in (0,1)$ and $\zeta,\eta \in \C$, we would deduce \eqref{e : low est} from \eqref{eq: OQ e OR}.
\medskip

So, here is the plan. First, we will recall the definition of $\cR_{n,\nu}$, then we will prove \eqref{eq: Fatou 2}, and finally we will demonstrate \eqref{e : low est}.

Since $V_-=W_-=0 \in L^\infty(\Omega)$, in view of Proposition~\ref{p: sempl prop} this method provides an alternative proof of \cite[Theorem~1.4]{CD-Potentials} for unbounded nonnegative potentials.

\subsubsection{The sequence  $\cR_{n,\nu}$}\label{s: seq Rnnu}
We recall now the construction of the sequence $\cR_{n,\nu}$. The interested reader should consult \cite{CD-Mixed} for more in-depth information about its genesis. 

Let $ p > 2$ and $A,B \in \cA_p(\Omega)$. By \cite[Corollary~5.15]{CD-DivForm} there exists $\varepsilon >0$ such that $\Delta_{p+\varepsilon}(A,B) >0$. For this particular $\epsilon>0$ and all $n\in\N_{+}$ define $f_{n}$ by
\begin{equation*}\label{eq: deffn}
f_{n}(t):=\begin{cases}
n^{-\epsilon}t^{p+\epsilon},& 0\leq t\leq n,\\
\frac{p+\epsilon}{2}n^{p-2}t^{2}+(1-\frac{p+\epsilon}{2})n^{p},&t\geq n.
\end{cases}
\end{equation*}
For every $k\in\N_{+}$, let  $\cF_{n}:\C^{k}\rightarrow\R_{+}$ be given by
$$
\cF_{n}(\omega):=f_{n}(|\omega|),\quad \omega\in\C^{k}.
$$
Let $K_{p+\epsilon}$ be the constant in \cite[(23)]{CD-Mixed} and define
\begin{equation*}
\cP_n(\zeta,\eta)
:=\cF_{n}(\zeta,\eta)+K_{p+\epsilon}\left(\cF_{n}(\zeta)+\cF_{n}(\eta)\right),\quad (\zeta,\eta)\in\C\times\C.
\end{equation*}
Let $q=p/(p-1)$ and $\cQ=\cQ_{p,\delta}$ denote the Nazarov-Treil Bellman function introduced in \eqref{eq: N Bellman} with $\delta>0$ chosen so that \cite[Theorem~6]{CD-DivForm} holds true.  
Fix a radial function $\f\in C_c^\infty(\R^4)$ such that $0\leq\f\leq1$, ${\rm supp}\, \f\subset B_{\R^{4}}(0,1)$ and $\int\f=1$. For our purposes, let us further assume that $\varphi$ is radially decreasing. Also, fix a radial function $\psi\in C^{\infty}_{c}(\C^{2})$ such that $\psi\geq 0$, $\psi=1$ on $\{|\omega|\leq 3\}$ and $\psi=0$ on $\{|\omega|>4\}$. For $\nu\in(0,1]$ and $n\in\N_{+}$, set $\varphi_\nu(\omega)=\nu^{-4}\varphi(\omega/\nu)$ and $\psi_{n}(\omega)=\psi(\omega/n)$. 

Recall notations \eqref{e : realis of compl fun} and \eqref{e : compl convol}. For every $n\in\N_{+}$ and all $\nu\in (0,1]$, define
\begin{equation}\label{eq: d Rnv}
\aligned
&\cQ_{n,\nu}&\hskip -6pt := &\hskip 2pt \psi_{n}\cdot (\cQ *\varphi_{\nu});\\
&\cR_{n,\nu}&\hskip -6pt :=&\hskip 2pt\cQ_{n,\nu}+C_{1}\nu^{q-2}(\cP_{n}*\varphi_{\nu}),
\endaligned
\end{equation}
where $C_{1}=C_{1}(p,A,B,\psi)>0$ is a constant not depending on $\nu$ which was fixed in \cite[Theorem~16]{CD-Mixed} to achieve the $(A,B)$-convexity of $\cR_{n,\nu}$ on $\C^2$ for all $n \in \N_+$ and $\nu \in (0,1)$. The constant $C_1$ will be adjusted later so that $\cR_{n,\nu}$ satisfies \eqref{eq: Fatou 2}; see Corollary \ref{c : pos int pot R}.

\subsubsection{Proof of \eqref{eq: Fatou 2}}\label{ss: proof of 78}
Outside the annulus $\{ 3n \leq |\omega| \leq 4n\}$, the $(A,B)$-convexity of $\cR_{n,\nu}$  follows directly from the $(A,B)$-convexity of $\cQ *\varphi_\nu$ and $\cP_{n} *\varphi_\nu$. Within the annulus, a lower bound of the generalized Hessian of  $\cP_{n} *\varphi_\nu$ suggests how large the constant $C_1$ must be chosen in order to compensate for the lack of convexity of $\cQ_{n,\nu}$ in this region; see \cite[Proposition~15(ii) \& Theorem~16]{CD-Mixed}. We proceed similarly to prove the first two inequalities in \eqref{eq: Fatou 2}: first, we establish the corresponding inequalities for $\cQ *\varphi_\nu$ and $\cP_{n} * \varphi_\nu$; next,  we derive a lower bound for the terms involving $\cP_{n} * \varphi_\nu$ within the annulus; finally, we choose $C_1$ sufficiently large to ensure the desired estimates.
\medskip

We will start with two lemmas. For all $a \in \R^2$ define the function $P_a: \R^2 \rightarrow \R$ as
\begin{equation*}
P_a(x) :=   \sk{a}{a-x}, \qquad x \in \R^2.
\end{equation*}
When $a \ne (0,0)$ we can define the reflection $R_a:\R^2 \rightarrow \R^2$  with respect the line $\{x \in \R^2 \, : \, P_{a}(x) =0\}$ as
\begin{equation*}
\aligned
R_a(x) &:=x +2P_a(x) \frac{a}{|a|^2}, \qquad x \in \R^2. 
\endaligned
\end{equation*}
\begin{lemma}
\label{l : prop P e R zeta}
For all $a \in \R^2\setminus\{(0,0)\}$ and $x \in \R^2$ we have
\begin{enumerate}[label=\textnormal{(\roman*)}]
\item \label{e : odd prop refl} $P_{a}(R_a(x)) = - P_{a}(x)$,
\item \label{e : even prop refl} $|a - R_a(x)| = |a-x|$, 
\item \label{i: mod refl} $|R_a(x)|^2 = |x|^2 + 4P_a(x)$.
\end{enumerate}
\end{lemma}
\begin{proof}
Item~\ref{e : odd prop refl} follows by a trivial computation.

To prove item~\ref{e : even prop refl} we observe that
$$
\aligned
|a - R_a(x)|^2 &= |a-x|^2 + \frac{4}{|a|^2}P_a^2(x) -\frac{4}{|a|^2}P_a(x)\sk{a-x}{a} \\
&= |a-x|^2 .
\endaligned
$$

Finally, we have
\begin{alignat*}{2}
|R_a(x)|^2 &= |x|^2 + \frac{4}{|a|^2}P_a^2(x) + \frac{4}{|a|^2}P_a(x)\sk{x}{a}\\
&=|x|^2 + \frac{4}{|a|^2}P_a^2(x) + \frac{4}{|a|^2}P_a(x)\sk{x-a}{a}+ 4P_a(x)\\
&= |x|^2 + 4 P_a(x).
\tag*{\qedhere}
\end{alignat*}
\end{proof}

The further assumption that the function $\varphi$ is radially decreasing will be used in the next lemma.
\begin{lemma}
\label{l : pos int pot}
Let $F : \R^2 \times \R^2 \rightarrow \R_+ \in L^1_{\rm loc}(\R^4)$ and $\varphi$ as above. Suppose that 
\begin{equation}
\label{e : rad funct}
F(x,y) = G(|x|,|y|)
\end{equation}
for some nonnegative function $G$ on $[0,\infty) \times [0,\infty)$ and all $x,y \in \R^2$. Then 
\begin{eqnarray*}
\int_{\R^4} P_a(x)  F(a-x,b-y) \varphi_\nu(x,y) \, \wrt (x,y)\geq 0,\\
\int_{\R^4}    P_b(y)   F(a-x,b-y)  \varphi_\nu(x,y) \, \wrt (x,y)\geq 0,\\
\end{eqnarray*}
for all $(a,b) \in \R^2 \times \R^2$ and $\nu \in (0,1)$.
\end{lemma}
\begin{proof} 
Fix $(a,b) \in \R^2 \times \R^2$ and $\nu \in (0,1)$. It suffices to prove the first inequality. If $a=(0,0)$ the assertion clearly holds. So suppose that $a \ne(0,0)$. 
Since the support of the integrand is contained in $B_{\R^4}(0,\nu)$, we have
$$
\aligned
 \int_{\R^4}&P_{a}(x)F(a-x,b-y)  \varphi_\nu(x,y)\, \wrt (x,y) \\
& = \int_{B_{\R^2}(0,\nu)}P_{a}(x) \int_{\{y \, : \, |y|^2 < \nu^2-|x|^2\}} F(a-x,b-y)  \varphi_\nu(x,y)\, \wrt y \, \wrt x.
\endaligned
$$
Set $\Theta_a := \{x \in B_{\R^2}(0,\nu) \, : \,P_{a}(x) < 0 \}$. If $\Theta_a = \emptyset$, we conclude. Otherwise, 
from Lemma~\ref{l : prop P e R zeta}\ref{i: mod refl} we have $R_a(\Theta_a) \subseteq B_{\R^2}(0,\nu)$ and hence
$$
B_{\R^2}(0,\nu) = \overline{\Theta_a} \, \sqcup \,  R_a(\Theta_a) \sqcup \, \left(B_{\R^2}(0,\nu) \setminus \left(\overline{\Theta_a} \cup R_a(\Theta_a)\right)\right).
$$
Therefore, by the nonnegativity of $F$ and $\varphi$ on $\R^4$ and of $P_a$ on $B_{\R^2}(0,\nu) \setminus (\overline{\Theta_a} \cup R_a(\Theta_a)$,
\begin{equation}
\label{e : int in refl reg}
\aligned 
\int_{\R^4}P_{a}(x)&  F(a-x,b-y)  \varphi_\nu(x,y) \, \wrt (x,y)  \\
\geq& \int_{\Theta_a}P_{a}(x) \int_{\{|y|^2 < \nu^2-|x|^2\}} F(a-x,b-y)  \varphi_\nu(x,y) \, \wrt y \, \wrt x  \\
& +  \int_{R_a(\Theta_a)}P_{a}(x) \int_{\{ |y|^2 < \nu^2-|x|^2\}} F(a-x,b-y)  \varphi_\nu(x,y) \, \wrt y \, \wrt x.
\endaligned
\end{equation} 
By combining Lemma~\ref{l : prop P e R zeta}\ref{e : odd prop refl} and a change of variable  in \eqref{e : int in refl reg} by means of $R_a$, we get
\begin{equation}
\aligned
\label{e : int one refl reg}
\int_{\R^4}P_{a}(x)& F(a-x,b-y)  \varphi_\nu(x,y) \, \wrt (x,y)  \\
\geq& \int_{\Theta_a}P_{a}(x) \int_{\{ |y|^2 < \nu^2-|x|^2\}}F(a-x,b-y)  \varphi_\nu(x,y) \, \wrt y \, \wrt x  \\
&-   \int_{\Theta_a}P_{a}(x ) \int_{ \{|y|^2 < \nu^2-|R_a(x)|^2\}} F\Bigl(a-R_a(x),b-y\Bigr) \varphi_\nu\Bigl(R_a(x),y\Bigr)\, \wrt y \, \wrt x . 
\endaligned
\end{equation}
 Lemma~\ref{l : prop P e R zeta}\ref{i: mod refl} gives $|R_a(x)| \leq |x|$ for any $x \in \Theta_a$. Thus, by also using  that $\varphi$ is radially decreasing, we obtain 
$$
\aligned
\{y \, : \, |y|^2 < \nu^2-|R_a(x)|^2\} &\supseteq \{y \, : \, |y|^2 < \nu^2-|x|^2\}, \\
\varphi_\nu(R_a(x),y) &\geq \varphi_\nu(x,y).
\endaligned
$$
Therefore, by merging these  with \eqref{e : rad funct},  Lemma~\ref{l : prop P e R zeta}\ref{e : even prop refl} and the nonnegativity of $F$ and $\varphi$, for every $x \in \Theta_a$ we have
\begin{equation}
\label{e : int rad decr}
\aligned
 \int_{ \{ |y|^2 < \nu^2-|R_a(x)|^2\}} &F(a-R_a(x),b-y) \varphi_\nu(R_a(x),y)\,  \wrt y \\
& \geq \int_{\{ |y|^2 < \nu^2-|x|^2\}} F(a-x, b-y) \varphi_\nu(x,y) \, \wrt y.
\endaligned
\end{equation}
We conclude by combining \eqref{e : int one refl reg} and \eqref{e : int rad decr} with the fact that $P_\zeta$ is negative on $\Theta_\zeta$.
\end{proof}

\begin{proposition}
\label{p : pos int pot Q e P}
Let $\nu \in (0,1)$. Then
\begin{enumerate}[label=\textnormal{(\roman*)}]
\item
\label{i : pos int pot Q}
 For all $\omega=(\zeta,\eta) \in \C \times \C$ we have
$$
\aligned
\Re\bigl( \zeta \cdot \partial_\zeta (\cQ* \varphi_\nu)(\omega) \bigr) &\geq 0, \\
\Re\bigl( \eta \cdot \partial_\eta (\cQ* \varphi_\nu)(\omega) \bigr) &\geq 0.
\endaligned
$$
\item
\label{i : pos int pot P}
For all $\omega=(\zeta,\eta) \in \C \times \C$ and $n \in \N_+$ we have
$$
\aligned
\Re\bigl(\zeta \cdot \partial_\zeta (\cP_n* \varphi_\nu)(\omega) \bigr) &\geq 0, \\
\Re\bigl(\eta \cdot \partial_\eta (\cP_n* \varphi_\nu)(\omega) \bigr) &\geq 0. 
\endaligned
$$
Moreover, for all $n \in \N_+$ and all $|\omega| \geq 2n$,
\begin{eqnarray*}
2 \, \Re \left( \sigma \cdot (\partial_\sigma \cP_n * \varphi_\nu)(\omega) \right) \geq (p+\varepsilon) n^{p-2} |\sigma|^2, \quad \sigma=\zeta,\eta.
\end{eqnarray*}
\end{enumerate}
\end{proposition}
\begin{proof}
An easy computation shows that
$$
\aligned
\partial_\zeta \cQ(\omega) &=\frac{\overline{\zeta}}{2} \left( p|\zeta|^{p-2} + 2 \delta
\begin{cases}
|\eta|^{2-q}, & \text{if }|\zeta|^p \leq |\eta|^q, \\
|\zeta|^{p-2}, & \text{if } |\zeta|^p \geq |\eta|^q
\end{cases}
\right), \\
\partial_\eta \cQ(\omega) &=\frac{\overline{\eta}}{2} \left( q|\eta|^{q-2} + (2-q)\delta
\begin{cases}
|\zeta|^2|\eta|^{-q}, & \text{if }|\zeta|^p \leq |\eta|^q, \\
|\eta|^{q-2}, & \text{if } |\zeta|^p \geq |\eta|^q
\end{cases}
\right), \\
\partial_\sigma \cP_n(\omega) &=\frac{\overline{\sigma}}{2} \biggl(g_n(|\omega|) + K_{p+\varepsilon} \, g_n(|\sigma|) \biggr), \quad \sigma=\zeta,\eta,
\endaligned
$$
where
\begin{equation}
\label{e : def of gn}
g_n(t) = (p+\varepsilon) 
\begin{cases}
n^{-\varepsilon} t^{p+\varepsilon-2}, & 0 \leq t \leq n,\\
n^{p-2}, &t \geq n.
\end{cases}
\end{equation}
Therefore, we conclude the proof of  \ref{i : pos int pot Q} and the first part of \ref{i : pos int pot P} by recalling the convention \eqref{e : compl convol} and by combining Lemma~\ref{l : pos int pot} with the identity
\begin{equation}
\label{eq: sp r sp c}
\Re\left( \zeta \cdot \overline{\zeta - \cV_{1}^{-1}(x)} \right) = P_{\cV_1(\zeta)}(x),
\end{equation}
which holds for all $\zeta \in \C$, $x \in \R^2$.

For the second part of \ref{i : pos int pot P}, we restrict ourselves to the case $\sigma = \zeta$, as the remaining case follows by an entirely analogous argument. We have
$$
\aligned
 2 \, \Re \biggl( \zeta \cdot (\partial_\zeta \cP_n &* \varphi_\nu)(\omega) \biggr)\\
 = \int_{B_{\R^4}(0,\nu)} &\Re\left(\zeta \cdot \overline{\zeta - \cV_1^{-1}(x)}\right) \times \\
&\times \left[g_n\left(\left|\omega - \cW_{2,1}^{-1}(x,y)\right|\right) + K_{p+\varepsilon} \, g_n\left(\left|\zeta - \cV_1^{-1}(x)\right|\right) \right] \varphi_\nu(x,y)\, \wrt (x,y).
\endaligned
$$
If we assume that $|\omega| > 2n$, then $|\omega - \cW_{2,1}^{-1}(x,y)| > 2n - \nu > n$. Therefore, by \eqref{e : def of gn} we get
$$
\aligned
 2 \, \Re \bigl( &\zeta \cdot (\partial_\zeta \cP_n * \varphi_\nu)(\omega) \bigr)\\
 =&\, (p+ \varepsilon) n^{p-2} \int_{\R^4} \Re\left(\zeta \cdot \overline{\zeta - \cV_1^{-1}(x)}\right) \varphi_\nu(x,y)\, \wrt (x,y)  \\
&+  K_{p+\varepsilon}  \int_{\R^4}   \Re\left(\zeta \cdot \overline{\zeta - \cV_1^{-1}(x)}\right) g_n\left(\left|\zeta - \cV_1^{-1}(x)\right|\right) \varphi_\nu(x,y) \, \wrt (x,y) \\
 = & \,  (p+\varepsilon) n^{p-2}  |\zeta|^2 -  (p+\varepsilon) n^{p-2}  \int_{\R^4}  \Re\left( \zeta \cdot \overline{\cV_1^{-1}(x,y)}\right) \varphi_\nu(x,y) \, \wrt (x,y)  \\ 
&+  K_{p+\varepsilon}  \int_{\R^4} \Re\left(\zeta \cdot \overline{\zeta - \cV_1^{-1}(x)}\right)  g_n\left(\left|\zeta - \cV_1^{-1}(x)\right|\right)  \varphi_\nu(x,y) \, \wrt (x,y).
\endaligned
$$

The first integral in the right-hand side of the last equality is zero since the integrand is odd for every $\zeta \in \C$. Finally, Lemma  \ref{l : pos int pot} and \eqref{eq: sp r sp c} imply that the second integral is nonnegative. Thus, we conclude.
\end{proof}
\begin{corollary}
\label{c : pos int pot R}
Let $p >2$. 
There exists $C_1>0$ such that 
$$
\cR_{n,\nu} = \psi_n \cdot  (Q *\varphi_\nu)+ C_1 \nu^{q-2}(\cP_n * \varphi_\nu)
$$
 satisfies
$$
\aligned
\Re\bigl(\zeta \cdot (\partial_\zeta \cR_{n,\nu})(\omega) \bigr) &\geq 0, \\
\Re\bigl(\eta \cdot (\partial_\eta \cR_{n,\nu})(\omega) \bigr) &\geq 0,
\endaligned
$$
for all $n \in \N$, $\nu \in (0,1)$ and $\omega=(\zeta,\eta)\in \C \times \C$.
\end{corollary}
\begin{proof}
The nonnegativity of both terms in the region $\{|\omega| < 3n \} \cup \{|\omega| > 4n\}$ follows, for any $C_1 >0$, from Proposition~\ref{p : pos int pot Q e P}\ref{i : pos int pot Q} and the first part of \ref{i : pos int pot P}.
Prove now the nonnegativity in the annulus $\{3n \leq |\omega| \leq 4n \}$. Since $\psi$ is even in each variable, we have 
\begin{equation*}
\partial_{\zeta_j} \psi_n(0,\eta)=0, \quad \partial_{\eta_j} \psi_n ( \zeta,0)=0,
\end{equation*}
for all $\zeta, \eta \in \C$ and $j \in \{1,2\}$. Therefore, the mean value theorem implies that there exists $C=C(\psi)>0$, independent of $n$, such that
\begin{equation}
\label{e : est of grad of psi}
\aligned
|\partial_\zeta \psi_n ( \zeta,\eta)| &\leq \frac{C}{n^2} |\zeta|,\\
|\partial_\eta \psi_n ( \zeta,\eta)| &\leq \frac{C}{n^2} |\eta|,
\endaligned
\end{equation}
for all $n \in \N_+$ and $\zeta,\eta \in \C$. Hence, by combining \eqref{e : est of grad of psi} with the first estimate of \cite[Lemma~14]{CD-Mixed} and by applying the product rule, we infer that there exists $C_0=C_0(p,\psi) >0$ such that
$$
\aligned
2 \, \Re\bigl(\zeta \cdot (\partial_\zeta \cQ_{n,\nu})(\omega) \bigr) &\geq 2 \psi_n   (\omega) \,\Re\biggl(\zeta \cdot (\partial_\zeta\cQ * \varphi_\nu)(\omega) \biggr) - C_0 n^{p-2} |\zeta|^2, \\
2 \, \Re\bigl( \eta \cdot (\partial_\eta \cQ_{n,\nu})(\omega) \bigr) &\geq 2 \psi_n  (\omega) \,\Re\biggl(\eta \cdot (\partial_\eta\cQ * \varphi_\nu)(\omega) \biggr) - C_0 n^{p-2} |\eta|^2,
\endaligned
$$
for every $\omega=(\zeta,\eta) \in \C \times \C$ with $|\omega| \leq 5n$ and all $n \in \N_+$ and $\nu \in (0,1)$. Moreover, by the nonnegativity of $\psi$ and by Proposition~\ref{p : pos int pot Q e P}\ref{i : pos int pot Q}, we get
\begin{equation}
\label{e: est in the amnu}
\aligned
2 \, \Re\bigl(\zeta \cdot (\partial_\zeta \cQ_{n,\nu})(\omega) \bigr) &\geq - C_0 n^{p-2} |\zeta|^2, \\
2 \, \Re\bigl( \eta \cdot (\partial_\eta \cQ_{n,\nu})(\omega) \bigr) &\geq - C_0 n^{p-2} |\eta|^2,
\endaligned
\end{equation}
for every $\omega=(\zeta,\eta) \in \C \times \C$ with $|\omega| \leq 5n$ and all $n \in \N_+$ and $\nu \in (0,1)$. In order to conclude, we choose $C_1$ large enough and combine \eqref{e: est in the amnu} with the second part of Proposition~\ref{p : pos int pot Q e P}\ref{i : pos int pot P}.
\end{proof}

\subsubsection{Proof of \eqref{e : low est}}\label{ss: proof of 76}
By using \cite[Theorem~16 (ii) and (iv), Lemma~14(ii)]{CD-Mixed}, the fact that $\cQ \in C^1(\C^2)$ and Lebesgue's dominated convergence theorem twice, we deduce that
\begin{equation}
\label{e : appr heat flow}
\aligned
\cO(\cQ)(u,v) = \lim_{\nu \rightarrow 0} \lim_{n \rightarrow + \infty} \cO(\cR_{n,\nu})(u,v).
\endaligned
\end{equation}
Combining \cite[Theorem~16 (i) and (v)]{CD-Mixed} with the mean value theorem, we get
$$
\aligned
|(\partial_{\zeta}\cR_{n,\nu})(\zeta,\eta)| &\leq C(n,\nu) |\zeta|, \\
|(\partial_{\eta}\cR_{n,\nu})(\zeta,\eta)| &\leq C(n,\nu) |\eta|,
\endaligned
$$
for any $\zeta,\eta \in \C$. These estimates, together with \cite[Lemma~19]{CD-Mixed} (applied under \AssBE), imply that 
$$
\aligned
(\partial_{\zeta}\cR_{n,\nu})(u,v) &\in \Dom(\gota_{A,V,\oV}), \\
(\partial_{\eta}\cR_{n,\nu})(u,v) &\in  \Dom(\gota_{B,W,\oW})
\endaligned
$$
for all $u \in \Dom(\gota_{A,V,\oV})$, $v \in \Dom(\gota_{B,W,\oW})$. Hence we can integrate by parts the integral on the right-hand side of \eqref{e : appr heat flow} and, by means of the chain rule for the composition of smooth functions with vector-valued Sobolev functions, deduce that
\begin{equation}
\label{e : int by parts R}
\aligned
\cO(\cR_{n,\nu})(u,v) = &\int_\Omega  H_{\cR_{n,\nu}}^{(A,B)}[(u,v); (\nabla u,\nabla v)]  \\
& +  2\int_\Omega \,  V \,\Re \left( u\cdot (\partial_{\zeta}\cR_{n,\nu})(u,v) \right) +  W \,\Re \left( v \cdot (\partial_{\eta}\cR_{n,\nu})(u,v) \right).
\endaligned
\end{equation}
By \cite[Theorem~16]{CD-Mixed} and Corollary~\ref{c : pos int pot R}, the integral on the right-hand side of \eqref{e : int by parts R} is nonnegative for all $n\in \N_+$. Hence, by Fatou's lemma and \cite[Theorem~16(ii)]{CD-Mixed},
\begin{equation}
\label{e : heat flow lim n}
\aligned
\lim_{n \rightarrow +\infty} &\cO(\cR_{n,\nu})(u,v) \\
\geq &\int_\Omega  H_{\cQ * \varphi_\nu}^{(A,B)}[(u,v); (\nabla u,\nabla v)]  \\
&+ 2\int_\Omega \,  V \,\Re \left( u\cdot (\partial_{\zeta}(\cQ * \varphi_\nu))(u,v) \right) +  W \,\Re \left( v \cdot (\partial_{\eta}(\cQ * \varphi_\nu))(u,v) \right).
\endaligned
\end{equation}
By Proposition \ref{p : pos int pot Q e P}\ref{i : pos int pot Q}, the integrand of the second integral on the right-hand side of \eqref{e : heat flow lim n} is nonnegative for all $\nu \in (0,1)$. Hence, by Fatou's lemma and the fact that $\cQ \in C^1(\C^2)$,
\begin{equation}
\label{e : heat flow lim nu}
\aligned 
 \liminf_{\nu\rightarrow 0}    &\int_\Omega \,  V \,\Re \bigl( u\cdot (\partial_{\zeta}(\cQ * \varphi_\nu))(u,v) \bigr) +  W \,\Re \bigl( v \cdot (\partial_{\eta}(\cQ * \varphi_\nu))(u,v) \bigr) \\
& \geq  \int_\Omega \,  V  (\partial_{\zeta}\cQ)(u,v)\cdot u +  W (\partial_{\eta}\cQ )(u,v)\cdot v.
\endaligned
\end{equation}
Therefore, combining \eqref{e : appr heat flow}, \eqref{e : heat flow lim n} and \eqref{e : heat flow lim nu}, we get \eqref{e : low est}.

\section{The general case: unbounded potentials}\label{s: unb neg pot}
In order to treat the general case with unbounded potentials, we will follow the argument used by Carbonaro and Dragi\v{c}evi\'c in \cite[Section~3.4]{CD-Potentials} when they proved \cite[Theorem~1.4]{CD-Potentials}. Like in their case, Theorem~\ref{t: N bil} will follow from the special case of potentials with bounded negative part, already proved in Section~\ref{s: proof b.e. neg bound}, once we prove the following approximation result.

Let $A \in \cA(\Omega)$, $\oV$ be a closed subspace of $W^{1,2}(\Omega)$ containing $W_0^{1,2}(\Omega)$ and  $U \in \cP_{\alpha,\sigma}(\Omega,\oV)$ such that 
\begin{equation}
\label{eq: AVlambda}
A - \alpha I_d \in \cA(\Omega).
\end{equation}
For each $n\in \N$ define
\begin{equation*}
U_n :=  U_+ - U_- \wedge n.
\end{equation*}
We also set $U_{\infty}=U$. Denote $\vartheta_0^* = \pi/2-\vartheta_0$, with $\vartheta_0$ being the angle defined in page \pageref{e : above boun a}.
\begin{theorem}
\label{t : conv semig}
For all $f \in L^2(\Omega)$  and all $z \in \bS_{\theta_0^*}$ we have  
$$
\begin{array}{rclccl}
\nabla T_z^{A, U_{n}}f &\rightarrow& \nabla T_z^{A,U} f& \quad &\text{in}& L^2(\Omega, \C^d), \\
 |U_n|^{1/2} T_z^{A, U_{n}}f &\rightarrow& |U|^{1/2} T_z^{A,U} f & \quad &\text{in} & L^2(\Omega)
\end{array}
$$
as $n \rightarrow \infty$.
\end{theorem}
The proof of Theorem~\ref{t : conv semig} relies on an adaptation of the argument employed by Carbonaro and Dragi\v{c}evi\'c to prove \cite[Theorem~3.6]{CD-Potentials}. As a first step, we establish a preliminary lemma, whose proof is based on an idea of Ouhabaz \cite{O-EssSp}. This lemma serves as a key ingredient in the proof of Proposition~\ref{p : converg resolv}, the counterpart of \cite[Proposition~3.9]{CD-Potentials}, from which Theorem~\ref{t : conv semig} can then be deduced from the standard representation of analytic semigroups by means of a Cauchy integral; we will omit the proof, see \cite[pp. 99]{CD-Potentials} for the detailed proof. We remark that in the statement of  \cite[Theorem~3.6]{CD-Potentials} the parameter $z$ is assumed to be in the interval $(0,\infty)$, whereas in Theorem~\ref{t : conv semig} it is allowed to lie in the sector $\bS_{\theta_0^*}$. This distinction, however, does not affect the proof: in the representation of the semigroups by means of a Cauchy integral, it suffices to choose $\delta >0$ and $\theta \in (0,\pi/2)$ such that $|\arg z| < \theta^* < \theta_0^*$ and $\gamma$ the positively oriented boundary of $\bS_\theta \cup \{\zeta \in \C: |\zeta| <\delta\}$.

\begin{notation}
Until the end of this chapter we will work with a single matrix function $A$. Therefore, in order to make the text more readable, we will from now omit $A$ in the notation for the operators and semigroups. For example, we will write $T_t^U$ instead of $T_t^{A,U}$ and $\oL^{U}$ instead of $\oL^{A,U}$.
\end{notation}

Clearly, for each $n \in \N \cup \{\infty\}$,
\begin{equation}
\label{e : dis sub unif}
\int_\Omega (U_-  \wedge n) |u|^2 \leq \alpha \int_{\Omega}|\nabla u|^2 + \sigma \int_\Omega U_+ |u|^2, \quad \forall u \in\oV. 
\end{equation}
It follows that the operators $\oL^{U_n}$, $n \in \N \cup \{\infty\}$, are uniformly sectorial of angle $\theta_0$ in the sense that
\begin{equation}
\label{e : unif sect est}
\|( \zeta -\oL^{U_n})^{-1} \|_2 \leq \frac{1}{{\rm dist}(\zeta, \overline{\bS}_{\theta_0})}, \quad \forall \zeta \in \C \setminus \overline{\bS}_{\theta_0}.
\end{equation}
The following lemma is modeled on \cite[Lemma~3.8]{CD-Potentials}, and we refer the reader to that paper for its proof. See also \cite{O-EssSp}. 
The only difference is that, instead of applying a monotone nondecreasing convergence theorem (see \cite[Theorem3.13a, p. 461]{Kat}), we use a monotone nonincreasing convergence theorem for sequences of symmetric sesquilinear forms (see \cite[Theorem3.11, p.~459]{Kat}), since in our case  the negative part of the potential has been truncated.
\begin{lemma}
\label{l : conv with real s}
For all $f \in L^2(\Omega)$ and all $s >0$ we have
\begin{equation*}
(s + \oL^{U_n})^{-1} f \rightarrow (s +\oL^{U})^{-1} f \,\, \text{ in } L^2(\Omega), \, \text{ as } n \rightarrow \infty.
\end{equation*}
\end{lemma}

Next proposition is modeled after \cite[Proposition~3.9]{CD-Potentials}.
\begin{proposition}
\label{p : converg resolv}
For all $ f \in L^2(\Omega)$ and all $\zeta \in \C \setminus \overline{\bS}_{\theta_0}$, we have
\begin{equation}
\label{eq: prop tronc}
\begin{array}{rclccl}
(\zeta - \oL^{U_n})^{-1} f &\rightarrow& (\zeta - \oL^{U})^{-1}f &\quad &\text{ in}& L^2(\Omega),\\
\nabla(\zeta - \oL^{U_n})^{-1} f &\rightarrow& \nabla(\zeta - \oL^U)^{-1}f & \quad &\text{ in}& L^2(\Omega,\C^d),\\
|U_n|^{1/2}(\zeta -\oL^{U_n})^{-1} f &\rightarrow& |U|^{1/2}(\zeta - \oL^U)^{-1}f & \quad &\text{ in}& L^2(\Omega),
\end{array}
\end{equation}
as $n \rightarrow \infty$.
\end{proposition}
\begin{proof}
Recall the notation $U_{\infty}=U$. Fix $f \in L^2(\Omega)$. For $n \in \N \cup \{\infty\}$ and $ \zeta \in \C \setminus \overline{\bS}_{\theta_0}$ set
$$
u_{n}(\zeta) := (\oL^{U_n} -\zeta)^{-1} f \in \Dom(\oL^{U_n}) \subseteq \oV \subseteq W^{1,2}(\Omega).
$$
By \eqref{eq: AVlambda} and \eqref{e : dis sub unif}, for every $n \in \N \cup \{\infty\}$ and $\zeta \in \C \setminus \overline{\bS}_{\theta_0}$ we have
\begin{eqnarray*}
&&\hskip -50 pt \lambda \| \nabla u_{n} \|_2^2 + (1-\sigma) \| U_+^{1/2} u_{n} \|_2^2 \\
& \leq&  \Re \, \int_\Omega [\sk{A\nabla u_{n}}{\nabla u_{n}} + U_+ |u_{n}|^2  - (U_n)_-| u_{n}|^2] \\
&= & \Re \, \int_{\Omega} (\oL^{U_n} u_{n} )\overline{u}_{n} \\
& =& \Re \, \int_\Omega f \overline{u}_{n} + (\Re \zeta) \, \int_\Omega |u_{n}|^2 \\
& \leq& \|f\|_2 \|u_{n}\|_2 + |\Re \zeta| \cdot \|u_{n}\|_2^2.
\end{eqnarray*}
Therefore, the uniform subcritical estimate \eqref{e : dis sub unif} and the uniform sectoriality estimate \eqref{e : unif sect est} give, for all $n \in \N \cup \{\infty\}$ and $\zeta \in \C \setminus \overline{\bS}_{\theta_0}$,
\begin{equation}
\label{e : bound seq}
\|u_{n}(\zeta)\|_2 +  \| \nabla u_{n}(\zeta) \|_2 + \| (U_n)_-^{1/2} u_{n}(\zeta) \|_2 \leq C_{\lambda, \alpha, \sigma, \theta_0}(\zeta) \|f\|_2,
\end{equation}
where $C_{\lambda, \alpha, \sigma,\theta_0}(\zeta) >0$ is continuous in $\zeta$.
\smallskip

Now temporarily fix $s >0$ and set
$$
\aligned
u_n & = u_{n}(-s), \\
u & =  u_{\infty}(-s).
\endaligned
$$
By \eqref{e : bound seq}, for all $n \in \N$ the sequence $\left(u_{n}\right)_{n \in \N}$ is bounded in $W^{1,2}(\Omega)$, hence it admits a weakly convergent subsequence. That is, there exist a subsequence of indices $(n_j)_{j \in \N}$ and function $\omega \in W^{1,2}(\Omega)$ such that
\begin{eqnarray*}
\label{eq: weak conv}
u_{n_k} &\rightharpoonup& \omega \quad \text{ in } W^{1,2}(\Omega), \nonumber\\
\end{eqnarray*}
as $k \rightarrow \infty$. Here the symbol $\rightharpoonup$ denotes weak convergence (that is, convergence in the weak topology). Lemma \ref{l : conv with real s} reads
\begin{equation}
\label{e : un tends u}
\lim_{m \rightarrow \infty} u_{n} = u \quad \text{ in } L^2(\Omega), \quad \forall s>0,
\end{equation}
which implies that $\omega=u$. Thus $u_{n_k} \rightharpoonup u$ in $W^{1,2}(\Omega)$.

Again by \eqref{e : bound seq}, the sequence \(\left((U_n)_-^{1/2} u_n\right)_{n \in \mathbb{N}}\) is bounded in \(L^2(\Omega)\). From \eqref{e : un tends u} and a standard theorem, we derive a subsequence \((u_{n_l})_{l \in \mathbb{N}}\) such that \(u_{n_l} \to u\) almost everywhere on \(\Omega\). 
Recall that \((U_n)_-^{1/2} \to U_-^{1/2}\) pointwise on \(\Omega\), just by the construction of \(U_n\). Hence, \((U_{n_l})_-^{1/2} u_{n_l} \to U_-^{1/2} u\) almost everywhere on \(\Omega\). 
Now a well-known theorem \cite[Theorem~13.44]{HS} gives $(U_{n_l})_-^{1/2} u_{n_l} \rightharpoonup U_-^{1/2} u$ in $L^2(\Omega)$.

Fix $\varphi \in C_c^\infty(\Omega)$. Since $U_+ \in L^1_{\rm{loc}}(\Omega)$, we have $U_+^{1/2}\varphi \in L^2(\Omega)$. Thus,  \eqref{e : un tends u} gives
\begin{equation*}
 \lim_{n \rightarrow \infty} \int_\Omega   U_+^{1/2}u_{n} \, \varphi =  \int_\Omega  U_+^{1/2}u  \, \varphi.
\end{equation*}
Hence, $ U_+^{1/2} u_{n} \rightharpoonup U_+^{1/2}  u$ in $L^2(\Omega)$.

So far,  we have proved that there exists a subsequence of indices $(n_k)_{k\in \N}$ such that in $L^2$ we have 
\begin{eqnarray}
\label{e : strong and weak conv}
\aligned
u_{n} &\rightarrow u, &\quad \nabla u_{n_k} &\rightharpoonup \nabla u, \\
(U_{n_k})_{-}^{1/2} u_{n_k} &\rightharpoonup U_-^{1/2} u, &\quad U_+^{1/2} u_{n} &\rightharpoonup U_+^{1/2} u.
\endaligned
\end{eqnarray}
We now show that the weak convergences in \eqref{e : strong and weak conv} are actually in the normed topology of $L^2(\Omega)$.
By \eqref{eq: AVlambda} and \eqref{e : dis sub unif}, 
\begin{eqnarray*}
J_{n_k} & :=&  s \|u_{n_k} -u \|_2^2 + \lambda \|\nabla u_{n_k} - \nabla u \|_2^2 + (1-\sigma) \left\| U_+^{1/2} u_{n_k} - U_+^{1/2} u \right\|_2^2 \\
& \leq &  s\| u_{n_k} \|^2 + s \|u\|^2 -2s \Re  \,\int_\Omega u_{n_k} \overline{u} \\
&\,& + \, \, \Re \int_\Omega \sk{A(\nabla u_{n_k} - \nabla u)}{\nabla u_{n_k} - \nabla u}  \\
&\,& +\, \,\int_\Omega U_+ |u_{n_k} - u|^2 -(U_{n_k})_{-}  |u_{n_k} - u |^2 \\
&=& I^0_{n_k} +  I^1_{n_k} +  I^2_{n_k} +  I^3_{n_k},
\end{eqnarray*}
where
\begin{eqnarray*}
I^0_{n_k} &=& s \|u\|_2^2 + \Re\,\int_\Omega \sk{A\nabla u}{u} + U_+ |u|^2 -(U_{n_k})_{-}|u|^2, \\
I^1_{n_k} &=&  s \|u_{n_k}\|_2^2 + \Re\,\int_\Omega \sk{A\nabla u_{n_k}}{u_{n_k}} + U_+ |u_{n_k}|^2 -(U_{n_k})_{-}|u_{n_k}|^2, \\
I^2_{n_k} &=&  -2s \Re \,\int_\Omega u_{n_k} \overline{u} - 2\Re\,\int_\Omega  U_+ u_{n_k} \overline{u} + 2\Re\, \int_\Omega (U_{n_k})_{-}^{1/2} u_{n_k}  U_-^{1/2}\overline{u},  \\
I^3_{n_k} &=& - \Re \left(\int_\Omega \sk{A \nabla u_{n_k}}{\nabla u} + \sk{A \nabla u}{u_{n_k}} \right), \\
I^4_{n_k} &=&2\Re\, \int_\Omega (U_{n_k})_{-}^{1/2} u_{n_k} \cdot \left((U_{n_k})_{-}^{1/2}- U_-^{1/2}\right)\overline{u}.
\end{eqnarray*}
Sending $k \rightarrow \infty$, we obtain
$$
\hskip -170pt I^0_{n_k} \rightarrow \Re \, \displaystyle \int_\Omega \left( \left( s+ \oL^U\right) u\right) \overline{u} = \Re \, \int_\Omega f \overline{u},
$$
because $(U_{n_k})_{-} |u|^2 \leq U_- |u|^2 \in L^1(\Omega)$  and $u \in \Dom(\oL^U)$;

$$
\hskip -70pt I^1_{n_k} = \Re \, \displaystyle \int_\Omega \left(\left(s+\oL^{U_{n_k}} \right) u_{n_k} \right) \overline{u}_{n_k} = \Re \, \int_\Omega f \overline{u}_{n_k} \rightarrow \Re \, \int_\Omega f \overline{u},
$$
because $u_{n_k} \rightarrow u$ in $L^2(\Omega)$;

\begin{eqnarray*}
 \hskip -42pt I^2_{n_k} = -2 \Re \, \left(s \displaystyle \int_\Omega u_{n_k} \overline{u} -\int_\Omega  U_+ u_{n_k}  \overline{u} +  \int_\Omega (U_{n_k})_{-}^{1/2} u_{n_k} U_-^{1/2} \overline{u}, \right)\\
 \rightarrow -2 s \|u\|_2^2 -2 \|U_+^{1/2}  u\|_2^2+2 \|U_-^{1/2}  u\|_2^2,
\end{eqnarray*}
 by  \eqref{e : strong and weak conv}, since $u \in \Dom(\gota)$ implies $U_+^{1/2} \overline{u}, \,U_-^{1/2} \overline{u} \in L^2(\Omega)$;

$$
\hskip -235pt I^3_{n_k,} \rightarrow -2 \Re\, \displaystyle \int_\Omega \sk{A \nabla u}{u},
$$
by \eqref{e : strong and weak conv} again, since $A \in \cA(\Omega)$ implies $|A \nabla u|, |A^* \nabla u| \leqsim |\nabla u| \in L^2(\Omega)$; and finally

$$
\hskip -70pt |I^4_{n_k}| \leq 2 \left(  \displaystyle \int_\Omega \left(U_-^{1/2} - (U_{n_k})_{-}^{1/2} \right)^2 |u|^2\right)^{1/2} \| (U_{n_k})_{-}^{1/2}u_{n_k}\|_2 \rightarrow 0,
$$
by the Cauchy-Schwarz inequality, \eqref{e : bound seq} and Lebesgue's dominated convergence theorem, since $ \left(U_-^{1/2}- (U_{n_k})_{-}^{1/2} \right)^2 |u|^2 \leq U_- |u|^2 \in L^1(\Omega)$.
\bigskip

Therefore,  using that $u \in \Dom(\oL^U)$, we obtain, as $k \rightarrow \infty$,
\begin{eqnarray*}
I^0_{n_k} + I^1_{n_k} &\rightarrow& 2 \Re \, \displaystyle\int_\Omega f \overline{u}, \\
I^2_{n_k} + I^3_{n_k} &\rightarrow & -2 \Re \, \displaystyle\int_\Omega ((s + \oL^U) u) \overline{u} = -2 \Re \, \int_\Omega f \overline{u}.
\end{eqnarray*}
It follows that  $J_{n_k} \rightarrow 0$ as $k \rightarrow \infty$, so
\begin{equation}
\label{e : strong  conv 2}
\nabla u_{n_k} \rightarrow \nabla u \quad \text{and} \quad U_+^{1/2} u_{n_k} \rightarrow U_+^{1/2} u \qquad \text{in } L^2(\Omega),
\end{equation}
as desired. Moreover, by \eqref{e : dis sub unif} we get
\begin{eqnarray*}
\| (U_{n_k})_{-}^{1/2} u_{n_k} - U_-^{1/2}u \|_2^2 &\leqsim& \left\| (U_{n_k})_{-}^{1/2} u_{n_k} - (U_{n_k})_{-}^{1/2} u\right\|_2^2 +\left \|\left( (U_{n_k})_{-}^{1/2} - U_-^{1/2}\right) u \right\|_2^2 \\
 & \leq& \alpha\| \nabla u_{n_k} -\nabla u \|_2^2 + \sigma \left\| U_+^{1/2} u_{n_k} - U_+^{1/2} u\right\|_2^2\\
&\, & + \, \,\left \|\left( (U_{n_k})_{-}^{1/2} - U_-^{1/2}\right) u \right\|_2^2,
\end{eqnarray*}
which, together with \eqref{e : strong  conv 2} and Lebesgue's dominated convergence theorem, implies that
\begin{equation}
\label{e : strong conv neg}
(U_{n_k})_{-}^{1/2} u_{n_k} \rightarrow U_-^{1/2} u \qquad \text{in } L^2(\Omega),
\end{equation}
as $k \rightarrow \infty$.

By repeating verbatim the argument following \eqref{e : strong and weak conv}, we may prove that every {\it subsequence} of $(u_{n})_{n \in \N}$ has its own subsequences for which \eqref{e : strong  conv 2} and \eqref{e : strong conv neg} hold. Therefore, by a standard convergence argument involving subsequences, \eqref{eq: prop tronc} holds for all $\zeta = -s$, $s >0$.

For the validity of \eqref{eq: prop tronc} for all $\zeta \in (\C \setminus \overline{\bS}_{\theta_0}) \setminus (-\infty, 0)$, we refer the reader to the final part of the proof of \cite[Proposition~3.9]{CD-Potentials}.
\end{proof}

\section{Maximal regularity and functional calculus: proof of Theorem~\ref{t: N principal}}\label{s: max funct}
We follow here the approach of Carbonaro and Dragi\v{c}evi\'c in \cite[Section~7]{CD-Mixed}, used to prove \cite[Theorem~3]{CD-Mixed}. For this, a bilinear estimate with complex time, analogous to \cite[(42)]{CD-Mixed}, is required; it will be established in the next subsection (see \eqref{eq: bilineq compl}).

\subsection{Bilinear embedding with complex time}\label{s: complex time}
Let $\theta, \phi \in (-\pi/2,\pi/2)$ be such that $(e^{i\theta}A, (\cos\theta)V) \in \cA_p(\Omega,\oV)$ and  $(e^{i\phi}B, (\cos\phi)W) \in \cA_p(\Omega,\oW)$. We will prove that 
\begin{equation}\label{eq: bilineq compl}
\int^{\infty}_{0}\!\int_{\Omega}\sqrt{\mod{\nabla T^{A,V}_{te^{i\theta}}f}^2 +  \mod{V} \mod{T^{A,V}_{te^{ i\theta}}f}^2 }\sqrt{\mod{\nabla T^{B,W}_{te^{i\phi}}g}^2+  \mod{W} \mod{T^{B,W}_{te^{ i\phi}}g}^2} \wrt x \wrt t \leq C \norm{f}{p}\norm{g}{q}, 
\end{equation}
for all $f,g \in (L^p \cap L^q)(\Omega)$.
\smallskip

First, assume that $V$ and $W$ have bounded negative part.
We follow the argument in Section~\ref{s: proof b.e. neg bound}, summarized as follows: 
\begin{itemize}
\item
Define $\gamma_{\theta,\phi} : [0, \infty) \rightarrow \C^2$ by 
$$
\gamma_{\theta,\phi}(t) := \left(T_{e^{i\theta}t}^{A,V}f, T_{e^{i\phi} t}^{B,W}g\right)
$$ 
and 
$\cE_{\theta,\phi} : [0,\infty) \rightarrow  [0,\infty)$ by
$$
\cE_{\theta,\phi}(t)= \int_\Omega \cQ\left(T_{e^{i\theta}t}^{A,V}f, T_{e^{i\phi} t}^{B,W}g\right) \, , \,\, t>0.
$$
We have
$$
\aligned
\hskip 38.3 pt -\cE_\theta^\prime(t) &= 2 \, \Re \int_\Omega \biggl( e^{i\theta}  \partial_{\zeta}\cQ(\gamma_{\theta,\phi}(t)) \oL^{A,V} T_{e^{i\theta}t}^{A,V}f + e^{i\phi}\partial_{\eta}\cQ(\gamma_{\theta,\phi}(t))  \oL^{B, W} T_{e^{i\phi}t}^{B,W}g  \biggr).
\endaligned
$$
As in Proposition~\ref{p: sempl prop1}, it suffices to show 
\begin{equation}
\label{eq: first sempl compl}
\aligned
2 \, \Re \int_\Omega \biggl( e^{i\theta}  \partial_{\zeta}\cQ(u,v) & \oL^{A,V}u + e^{i\phi}\partial_{\eta}\cQ(u,v) \oL^{B, W}v \biggr) \\
\geq &  \int_\Omega  \sqrt{|\nabla u|^2+|V||u|^2}\sqrt{|\nabla v|^2+|W||v|^2}
\endaligned
\end{equation}
for all  $u \in \Dom(\gota_{A,V,\oV})$, $v \in \Dom(\gota_{B,W,\oW})$ such that $u, v, \oL^{A,V}u, \oL^{B,W}v \in (L^p \cap L^q)(\Omega)$.
\item 
 Recall definition \eqref{eq: d Rnv}. As done in Section~\ref{ss: proof of 78}, we can prove that there exists $C_1=C_1(\theta,\phi,p)>0$ in the definition of $\cR_{n,\nu}$ such that
$$
\aligned
\Re\bigl(e^{i\vartheta}\zeta \cdot (\partial_\zeta \cR_{n,\nu})(\omega) \bigr) &\geq 0, \\
\Re\bigl(e^{i\phi} \eta \cdot (\partial_\eta \cR_{n,\nu})(\omega) \bigr) &\geq 0, \\
\Re \left(e^{i\vartheta} \zeta \cdot \partial_{\zeta}(\cQ * \varphi_\nu)(\omega) \right) &\geq 0, \\
\Re \left(e^{i\phi} \eta \cdot \partial_{\eta}(\cQ * \varphi_\nu)(\omega) \right) &\geq 0, \\
\endaligned
$$
for all $n \in \N$, $\nu \in (0,1)$ and $\omega=(\zeta,\eta) \in \C^2$.
\item
Consequently (see Section~\ref{ss: proof of 76}), since $e^{i\theta}A, e^{i\phi}B \in \cA_p(\Omega)$  we obtain
\begin{equation}
\label{e : heat flow est pos 2 compl}
\aligned
\hskip 40 pt 2 \, \Re \int_\Omega \biggl( e^{i\theta}  \partial_{\zeta}&\cQ(u,v)   \oL^{A,V_+}u + e^{i\phi}\partial_{\eta}\cQ(u,v) \oL^{B, W_+}v \biggr) \\
\geq &\,  \liminf_{\nu \rightarrow 0} \int_\Omega  H_{\cQ * \varphi_\nu}^{(e^{i\theta}A,e^{i\phi}B)}[(u,v); (\nabla u,\nabla v)]  \\
& + 2 \, \int_\Omega ( \cos\theta) V_+   (\partial_{\zeta}\cQ)(u,v)\cdot u +  ( \cos\phi) W_+  (\partial_{\eta}\cQ )(u,v)\cdot v,
\endaligned
\end{equation}
 for all $u \in \Dom(\gota_{A,V_+,\oV})$, $v \in \Dom(\gota_{B,W_+,\oW})$ such that $u, v, \oL^{A,V_+}u, \oL^{B,W_+}v \in (L^p \cap L^q)(\Omega)$
\item
As described in the proof of Proposition~\ref{p: sempl prop}, from \eqref{e : heat flow est pos 2 compl} and the fact that $(e^{i\theta}A, (\cos\theta)V) \in \cA_p(\Omega,\oV)$ and  $(e^{i\phi}B, (\cos\phi)W) \in \cA_p(\Omega,\oW)$ we deduce \eqref{eq: first sempl compl}, which in turn implies \eqref{eq: bilineq compl}. 
\end{itemize}
\smallskip

Finally, the bilinear estimate \eqref{eq: bilineq compl} in the general case is obtained by combining Theorem~\ref{t : conv semig} with the previously established  estimate for potentials with bounded negative part.

\subsection{Proof of Theorem~\ref{t: N principal}}
The following result is modeled after \cite[Proposition 20]{CD-Mixed}. See \cite[Sections 7.1 and 7.2]{CD-Mixed} for the necessary terminology and references.
\begin{proposition}\label{p: N D-V bis} 
Suppose that $\oV$ satisfies \eqref{eq: inv P} and \eqref{eq: inv N}. Choose $p>1$ and $(A,V) \in \cA\cP_p(\Omega,\oV)$. Let $-\oL_p$ be the generator of $(T_{t})_{t>0}$ on $L^{p}(\Omega)$. If $\omega_{H^{\infty}}(\oL_{p})<\pi/2$, then $\oL_{p}$ has parabolic maximal regularity.
 \end{proposition} 

We are ready now to prove Theorem~\ref{t: N principal}. Without loss of generality we suppose $p\ge2$. In light of Proposition~\ref{p: N D-V bis} it suffices to show that 
$$
(A,V) \in \cA\cP_p(\Omega,\oV) \implies \omega_{H^{\infty}}(\oL_{p})<\pi/2.
$$ 

Observe that $\oL_2^{A^{*},V} = \left( \oL^{A,V}\right)_2^*$, so $T^{A^{*},V,\oV}_{t}=\left(T^{A,V,\oV}_{t}\right)^{*}$ for all $t> 0$.  \\
Set $T_t = T^{A,V,\oV}_{t}$ and $T_t^*=T^{A^{*},V,\oV}_{t}$ for all $t>0$. 

By Proposition~\ref{p : basic prop}\ref{i : inv small rot},\ref{i : inv adj}, there exists $\theta\in (0,\pi/2)$ such that $(e^{\pm i\theta}A, (\cos\theta)V)$, \\$(e^{\mp i\theta}A^{*}, (\cos\theta)V) \in \cA\cP_p(\Omega,\oV)$.
Moreover, for every $r \in [q,p]$ both $(T_{t})_{t>0}$ and $(T^{*}_{t})_{t>0}$ are analytic (and contractive) in $L^r(\Omega)$ in the cone $\bS_{\theta}$; see Corollary \ref{c: N analytic sem}.
\smallskip

From \eqref{eq: bilineq compl} there exists $C>0$ such that   
\begin{equation}\label{eq: bilineq*}
\int^{\infty}_{0}\!\int_{\Omega}\sqrt{\mod{\nabla T_{te^{\pm i\theta}}f}^2 +  |V| \mod{T_{te^{\pm i\theta}}f}^2 }\sqrt{\mod{\nabla T^{*}_{te^{\mp i\theta}}g}^2+  |W| \mod{T^{*}_{te^{\mp i\theta}}g}^2} \wrt x \wrt t \leq C \norm{f}{p}\norm{g}{q}, 
\end{equation}
for all $f,g\in (L^{p}\cap L^{q})(\Omega)$.

It follows from \eqref{eq: bilineq*} and the inequality
\begin{equation*}
\aligned
\biggl|\int_{\Omega}\oL_{2}&T_{te^{\pm i\theta}}f\,\,\overline{T^{*}_{te^{\mp i\theta}}g }\wrt x \biggr|\\
&\leqsim   \int_{\Omega}\sqrt{\mod{\nabla T_{te^{\pm i\theta}}f}^2 + | V| \mod{T_{te^{\pm i\theta}}f}^2 }\sqrt{\mod{\nabla T^{*}_{te^{\mp i\theta}}g}^2+  |V| \mod{T^{*}_{te^{\mp i\theta}}g}^2} \wrt x,
\endaligned
\end{equation*}
that
$$
\int^{\infty}_{0}\!\mod{\int_{\Omega}\oL_{p}T_{2te^{\pm i\theta}}f\, \overline{g }\wrt x}\wrt t \, \leqsim \, \norm{f}{p}\norm{g}{q},
$$
for all $f,g\in (L^{p}\cap L^{q})(\Omega)$. Analyticity of $(T_{t})_{t>0}$ in $L^{p}(\Omega)$, Fatou's lemma and a density argument show that
\begin{equation}\label{eq: N CDMY}
\int^{\infty}_{0}\!\mod{\int_{\Omega}\oL_{p}T_{te^{\pm i\theta}}f\, \overline{g }\wrt x}\wrt t \, \leqsim \, \norm{f}{p}\norm{g}{q},
\end{equation}
for all $f\in L^{p}(\Omega)$ and all $g\in L^{q}(\Omega)$. 

We now apply \cite[Theorem~4.6 and Example~4.8]{CDMY} to the dual subpair $\sk{\ovR(\oL_{p})}{\ovR(\oL^{*}_{q})}$ and the dual operators $(\oL_{p})_{\vert\vert}$, $(\oL^{*}_{q})_{\vert\vert}$ \cite[p.~64]{CDMY}, and deduce from \eqref{eq: N CDMY} that $\omega_{H^{\infty}}(\oL_{p})\leq \pi/2-\theta$.

\section{Strongly subcritical potentials}\label{s: ex ss pot}
This section is devoted to presenting examples of strongly subcritical potentials associated with different choices of $\oV$.
Given the definition, it is natural to begin with Hardy-type inequalities, which provide canonical examples of subcritical potentials on $W^{1,2}_0(\Omega)$ and, more generally, on $W^{1,2}_D(\Omega)$.
In contrast, Hardy's inequality does not yield such examples on the full space $W^{1,2}(\Omega)$, where a different line of argument is required.
\subsection{Hardy's inequality on domain}
For every closed $D \subseteq \partial\Omega$ we define the function ${\rm dist}_D ={\rm dist}(\cdot,D)$ on $\Omega$. In the special case when $D = \partial\Omega$, we simply write ${\rm dist}_\Omega = {\rm dist}_{\partial \Omega}$. The classical $p$-Hardy inequality on $\Omega$ takes the form
\begin{equation}
\label{eq: HI}
\int_\Omega \left|\frac{u}{{\rm dist}_\Omega}\right|^p  \leq c \int_\Omega |\nabla u|^p, \qquad u \in W_0^{1,p}(\Omega),
\end{equation}
where $p \in (1,\infty)$ and $c=c(d,p)>0$. This inequality was first investigated in the one-dimensional setting by Hardy (see \cite[Sect. 33]{HLP52} and the references therein).
Ne\v{c}as \cite{Necas} subsequently extended the $p$-Hardy inequality to higher dimensions,  proving that \eqref{eq: HI} holds for every  $p \in (1,\infty)$ on any bounded Lipschitz domain $\Omega \subset \R^d$, with a constant $c = c(\Omega,d,p) > 0$.
Later developments showed that a domain $\Omega \subset \R^d$ satisfies the $p$-Hardy inequality under the weaker assumption that the complement of $\Omega$ is uniformly $p$-fat \cite{Ancona86, Lewis88, Wannebo90, KM97}. As a consequence, by taking $p=2$,  for such domains $\Omega$ one obtains
$$
-({\rm dist}_\Omega)^{-2} \in \cP\left(\Omega, W_0^{1,2}(\Omega)\right).
$$
Without imposing any geometric restriction on $\Omega \subset \R^d$, the space $W_0^{1,p}(\Omega)$ is  the largest subspace of $W^{1,p}(\Omega)$ on which Hardy's inequality \eqref{eq: HI} holds. More precisely, if $u \in W^{1,p}(\Omega)$ and $u/ {\rm dist}_\Omega \in L^p(\Omega)$, then necessarily $u \in W_0^{1,p}(\Omega)$ \cite[p. 223]{EE87}. An even stronger statement is true: it suffices to assume that $u / {\rm dist}_\Omega$ belongs to the weak  $L^{p}(\Omega)$ \cite{KM97}.
In particular, for every $\oV$ satisfying  $W_0^{1,2}(\Omega)\subsetneq \oV \subseteq W^{1,2}(\Omega)$ and for any domain $\Omega \subset \R^d$, we have the strict inclusion
$$
\cP(\Omega, \oV) \subsetneq \cP\left(\Omega, W^{1,2}_0(\Omega)\right).
$$

\medskip

More recently, Egert, Haller-Dintelmann and Rehberg developed a geometric framework for Hardy's inequality on bounded domains $\Omega$, in the setting where functions vanish only on a closed portion $D$ of the boundary, i.e., when they belong to $W_D^{1,2}(\Omega)$ \cite{EHDR}. We refer to their work for the underlying geometric definitions.
For every $p \in (1,\infty)$ they proved in \cite[Theorem~3.1]{EHDR} the existence of a constant $c>0$ such that 
\begin{equation}
\label{eq: HI mix}
\int_\Omega  \left| \frac{u}{{\rm dist}_D} \right|^p  \leq c \int_\Omega  |\nabla u|^p, 
\qquad u \in W_D^{1,p}(\Omega),
\end{equation}
provided the following conditions are satisfied: 
\begin{enumerate}
\item The set $D$ is $l$-thick for some $l \in (d-p,d)$.
\item The space $W^{1,p}_D(\Omega)$ admits an equivalent norm given by $\| \nabla \cdot \|_{L^p(\Omega)}$.
\item There exists a continuous linear extension operator $E: W^{1,p}(\Omega) \to W_{D}^{1,p}(\R^d)$.
\end{enumerate}

In particular, by taking $p=2$ one has 
\[
-({\rm dist}_D)^{-2} \in \cP\left(\Omega, W_D^{1,2}(\Omega)\right).
\]
Conditions (i) and (ii) are automatically fulfilled if, for every $x \in \overline{\partial\Omega \setminus D}$, there exists an open neighborhood $U_x$ such that $\Omega \cap U_x$ is a $W^{1,2}$-extension domain \cite[Theorem~3.2]{EHDR}. 

Moreover, if in addition $D$ is porous, then $W_D^{1,2}(\Omega)$ is the largest subspace of $W^{1,2}(\Omega)$ on which Hardy's inequality \eqref{eq: HI mix} holds. More precisely, if $u \in W^{1,2}(\Omega)$ and $u/{\rm dist}_D \in L^2(\Omega)$, then necessarily $u \in W_D^{1,2}(\Omega)$ \cite{EHDR}.  As a consequence, it follows that 
\[
-({\rm dist}_D)^{-2} \notin \cP\left(\Omega, W^{1,2}(\Omega)\right),
\]
and thus we obtain the strict inclusion
\[
\cP\left(\Omega, W^{1,2}(\Omega)\right) \subsetneq \cP\left(\Omega, W^{1,2}_D(\Omega)\right).
\]

\medskip

So far, we have provided examples of subcritical potentials for $W_0^{1,2}(\Omega)$ and for $W_D^{1,2}(\Omega)$. What remains is to give an example of a potential belonging to $\cP(\Omega, W^{1,2}(\Omega))$, and hence to $\cP(\Omega, \widetilde{W_D}^{1,2}(\Omega))$. In the next section we shall address this, applying an argument we learned from \cite{Magniez}. This approach will also allow us to construct further examples within the classes $\cP(\Omega, W_0^{1,2}(\Omega))$ and $\cP(\Omega, W_D^{1,2}(\Omega))$.

\subsection{Potentials on homogeneous domain}
Let $\Omega \subseteq \R^d$ be open and $\oV=\oV(\Omega)$ be a closed subspace of $W^{1,2}(\Omega)$ containing $W^{1,2}_0(\Omega)$. Denote by $\Delta_\oV = \oL^{I,0,\oV}$ the Laplacian on $L^2(\Omega)$ with domain $\Dom(\Delta_\oV) \subseteq \oV$.  For all $x \in \Omega$ and $r>0$, define
$$
v(x,r) := |\Omega \cap B(x,r)|.
$$

The aim of this section is to establish a sufficient condition for strong subcriticality, thereby providing examples of potentials in $\cP(\Omega, \oV)$, and in particular in $\cP(\Omega, W^{1,2}(\Omega))$. The argument is not new: we shall follow and adapt the method of \cite[Section~5]{Magniez}, proving, under suitable assumptions on $\Omega$, that the finiteness of
\begin{equation}
\label{eq: vol V}
\|V_-^{1/2}\|_{vol} := 
\int_0^1 \left\|\frac{V_-^{1/2}}{v(\cdot,\sqrt{t})^{\frac{1}{r_1}}}\right\|_{r_1} \frac{\wrt t}{\sqrt{t}} + 
 \int_1^\infty \left\|\frac{V_-^{1/2}}{v(\cdot,\sqrt{t})^{\frac{1}{r_2}}}\right\|_{r_2} \frac{\wrt t}{\sqrt{t}},
\end{equation}
for some $r_1, r_2 >2$, is sufficient for $V$ to be a strongly subcritical potential on $\Omega$.  The quantity \eqref{eq: vol V} has been introduced by Assad and Ouhabaz in \cite{AO} for studying the boundedness on $L^p$ of Riesz transforms of Schr\"{o}dinger operators on complete Riemannian manifolds. 
It subsequently  appeared in \cite{Magniez} in the context of $L^p$-boundedness of Riesz transforms of the Hodge-de Rham Laplacian on complete Riemannian manifolds. There, the negative part  $R_-$ of the Ricci curvature plays the role of $V_-$ and it has been proved to be subcritical whenever  $\|R_-\|_{vol}$ is small enough \cite[Proposition~5.9]{Magniez}. In both papers, two structural assumptions on the manifold are required: the volume doubling property and Gaussian upper estimates for the heat kernel of the Laplace-Beltrami operator. To replicate the argument of \cite{Magniez} in our setting, we impose analogous conditions on the pair $(\Omega, \oV(\Omega))$:  we assume that
\begin{enumerate}[label=\textnormal{(\roman*)}]
\item \label{i: doub} there exists $C>0$ such that $v(x,2r)\leq C v(x,r)$ for all $x \in \Omega$ and $r>0$;
\item \label{i: gauss} the semigroup $(e^{-t\Delta_\oV})_{t>0}$ has a Gaussian (upper) bound, that is, there exist $k_t(x,y) \in L^\infty(\Omega \times \Omega)$ and $C,c >0$ satisfying 
$$
|k_t(x,y)| \leq  \frac{C e^{-c\frac{|x-y|^2}{t}}}{v(x,\sqrt{t})}, \quad {\rm a.e.} \,\, x, y, \, \forall t>0,
$$
such that
$$
e^{-t\Delta_\oV}f (x) = \int_\Omega k_t(x,y) f(y) \wrt y,
$$
for almost all $x \in \Omega$, all $t>0$ and $f \in L^2(\Omega)$.
\end{enumerate}

Assumptions \ref{i: doub} and \ref{i: gauss} allow us to apply \cite[Proposition~2.9]{AO} for $\Delta_\oV$ and obtain
\begin{equation}
\label{eq: pq est}
\| v(\cdot, \sqrt{t})^{\frac{1}{p}-\frac{1}{r}} e^{-t\Delta_\oV}\|_{p-r} \leq C, \quad \forall \, 1 <p \leq r < \infty,
\end{equation}
where $C$ is a nonnegative constant depending on $p$, $r$ and on the constants appearing in \ref{i: doub} and \ref{i: gauss}. For every $\varepsilon>0$ we have the domination $|e^{-t(\Delta_\oV +\varepsilon)} f| \leq e^{-t\Delta_\oV} |f|$ for all $f \in C_c^\infty(\Omega)$. Therefore, \eqref{eq: pq est} yields
\begin{equation}
\label{eq: pq est schr}
\| v(\cdot, \sqrt{t})^{\frac{1}{p}-\frac{1}{r}} e^{-t(\Delta_\oV + \varepsilon)}\|_{p-r} \leq C, \quad \forall \, 1 <p \leq r < \infty,
\end{equation}
with $C$ as in \eqref{eq: pq est}, thus not depending on $\varepsilon$.
Such estimate is the key  ingredient for the next lemma, modeled after \cite[Lemma~5.4]{Magniez}.
\begin{lemma}
\label{l: est vol lap}
Assume that \ref{i: doub} and \ref{i: gauss} are satisfied. Let $V \in L^1_{\rm loc}(\Omega)$ be nonnegative. Then there exists a constant $C\geq 0$, depending on the constants appearing in \ref{i: doub} and \ref{i: gauss}, such that
$$
\|V^{1/2} (\Delta_\oV+\varepsilon)^{-1/2} \|_{2-2} \leq C \| V^{1/2}\|_{vol},
$$
for all $\varepsilon >0$.
\end{lemma}
\begin{proof}
Set $H=\Delta_\oV + \varepsilon$.  Writing 
$$
H^{-1/2}= \frac{1}{2\sqrt{\pi}} \int_0^\infty e^{-tH} \frac{\wrt t}{\sqrt{t}},
$$
 and using the H\"{o}lder inequality, we obtain 
$$
\aligned
\|V&^{1/2}H^{-1/2}\|_{2-2}\\
\leq&\,\, C  \int_0^1 \left\|\frac{V^{1/2}}{v(\cdot,\sqrt{t})^{\frac{1}{r_1}}}v(\cdot,\sqrt{t})^{\frac{1}{r_1}} e^{-tH}  \right\|_{2-2} \frac{\wrt t}{\sqrt{t}} + \int_1^\infty \left\|\frac{V^{1/2}}{v(\cdot,\sqrt{t})^{\frac{1}{r_2}}}v(\cdot,\sqrt{t})^{\frac{1}{r_2}} e^{-tH}  \right\|_{2-2} \frac{\wrt t}{\sqrt{t}} \\
\leq&\,\, C \int_0^1 \left\|\frac{V^{1/2}}{v(\cdot,\sqrt{t})^{\frac{1}{r_1}}}\right\|_{r_1} \left\|v(\cdot,\sqrt{t})^{\frac{1}{r_1}} e^{-tH}  \right\|_{2-\frac{2r_1}{r_1-2}} \frac{\wrt t}{\sqrt{t}} \\
 &+ C \int_1^\infty \left\|\frac{V^{1/2}}{v(\cdot,\sqrt{t})^{\frac{1}{r_2}}}\right\|_{r_2} \left\|v(\cdot,\sqrt{t})^{\frac{1}{r_2}} e^{-tH}  \right\|_{2-\frac{2r_2}{r_2-2}} \frac{\wrt t}{\sqrt{t}}.
\endaligned
$$
Since for $i=1,2$ we have 
$$
\frac{1}{r_i}=\frac{1}{2} -\frac{r_i-2}{2r_i},
$$
 we conclude by invoking \eqref{eq: pq est schr}.
\end{proof}

The following corollary is modeled after \cite[Proposition~5.8]{Magniez}.
\begin{corollary}
\label{c: est vol}
Assume that \ref{i: doub} and \ref{i: gauss} are satisfied. Let $V \in L^1_{\rm loc}(\Omega)$ be nonnegative such that $\|V^{1/2}\|_{vol}< \infty$. Then there exists a nonnegative constant $\alpha$ such that $-V \in \cP_{\alpha,0}(\Omega,\oV)$. The constant $\alpha$ can be chosen equal to $C\|V^{1/2}\|_{vol}$, with $C\geq 0$ depending on the constants appearing in \ref{i: doub} and \ref{i: gauss}.

In particular, for every nonnegative $W \in L^1_{\rm loc}(\Omega)$, the potential $W-V$ belongs to $\cP_{\alpha,\sigma}(\Omega, \oV)$ for all $\sigma \in [0,1)$.
\end{corollary}
\begin{proof}
Let $\varepsilon >0$ and set $H=\Delta_\oV + \varepsilon$. 
We have 
$$
\aligned
\int_\Omega V |u|^2 = \|V^{1/2}u\|_2^2 &=  \|V^{1/2}H^{-1/2} H^{1/2} u\|_2^2 \\
&\leq  \|V^{1/2}H^{-1/2}\|_{2-2}^2 \|H^{1/2}u\|_2^2 \\
&=  \|V^{1/2}H^{-1/2}\|_{2-2}^2 \sk{H u}{u} \\
& =  \|V^{1/2}H^{-1/2}\|_{2-2}^2  \left(\int_\Omega |\nabla u|^2 + \varepsilon |u|^2\right).
\endaligned
$$
Applying Lemma~\ref{l: est vol lap} and sending $\varepsilon \rightarrow 0$ yield the claim.  

The last assertion follows from the facts that $-V \in \cP_{\alpha,0}(\Omega,\oV)$ and  $(W-V)_- \leq V$ for every nonnegative $W \in L^1_{\rm loc}(\Omega)$.
\end{proof}
\begin{remark}
\label{r: pol vol}
\begin{enumerate}[label=\textnormal{(\roman*)}]
\item
\label{i: pol vol}
If the volume on $\Omega$ is polynomial, that is,  $ c r^d \leq  v(\cdot,r) \leq C r^d$, then $\|V_-^{1/2}\|_{vol} < \infty$ if and only if $V_- \in L^{d/2-\eta} \cap L^{d/2+\eta}$ for some $\eta >0$.
\item
If $\|V_-^{1/2}\|_{vol}<\infty$ then $\Omega$ is unbounded.
\end{enumerate}
\end{remark}

An immediate consequence of Corollary~\ref{c: est vol} is the following (compare with \cite[Proposition~5.9]{Magniez}).
\begin{corollary}
Assume that \ref{i: doub} and \ref{i: gauss} are satisfied. Suppose that there exists $W \in L^{1}_{\rm loc}(\Omega)$ nonnegative such that $\|W^{1/2}\|_{vol}<\infty$. Then, 
for all $A \in \cA(\Omega)$ there exists $V \in \cP(\Omega,\oV)$ such that $(A,V) \in \cA\cP(\Omega,\oV)$.
\end{corollary}

Let us return to the assumptions \ref{i: doub} and \ref{i: gauss} and discuss situations in which they are satisfied. It is well known that the semigroup associated with the Dirichlet Laplacian $\Delta_{W_0^{1,2}(\Omega)}$ admits a Gaussian upper bound for every open set 
$\Omega \subseteq \R^d$ \cite{D}. Furthermore, the semigroup associated with $\Delta_\oV$ 
also enjoys a Gaussian upper bound provided that the space $\oV$ satisfies \eqref{eq: inv P} and the 
following two conditions:
\begin{enumerate}[label=(\alph*)]
\item \label{i: hse} 
$\oV$ enjoys the homogeneous Sobolev embedding property; see Definition~\ref{d: emb prop};
\item \label{i: mult exp inv} $u \in \oV \;\Rightarrow\; e^\psi u \in \oV$ for every real-valued 
function $\psi \in C^\infty(\R^d)$ such that both $\psi$ and $\nabla\psi$ are 
bounded on $\R^d$.
\end{enumerate}
For details, see \cite[Chapter~6.3]{O}. Condition (a)  always holds when $\oV=W^{1,2}_0(\Omega)$ \cite[Theorem~2.4.1]{Ziemer}. 
On the other hand, condition (b) and \eqref{eq: inv P} are always verified when $\oV$ falls into any of the special cases \ref{i: D}-\ref{i: cM} of Section~\ref{s: boundary}; see for example \cite[Lemma~4(ii)]{Egert20}.

In some cases the heat kernel can be written explicitly. For instance, when $\Omega = \R_+^d := \{x \in \R^d \, : \, x_d >0\}$ the heat kernel $k_t$ of the Neumann Laplacian $\Delta_{W^{1,2}(\R_+^d)}$ is  given by
$$
k_t(x,y) = \frac{1}{(4\pi t)^{d/2}}e^{-\frac{|x-y|^2}{4t}} + \frac{1}{(4\pi t)^{d/2}} e^{-\frac{|x-y^\prime|^2}{4t}},
$$
where $y^\prime = (y_1, \dots, y_{d-1}, -y_d)$ if $y=(y_1,\dots,y_d)$ \cite{GSC}. Hence, $(\R_+^d, W^{1,2}(\R_+^d))$ satisfies \ref{i: gauss}. Moreover, the volume on $\R_+^d $ is polynomial, so  $\R_+^d $ also satisfies \ref{i: doub}.  Therefore, by Corollary~\ref{c: est vol} and Remark~\ref{r: pol vol}\ref{i: pol vol}, we obtain
$$
\emptyset \ne  (L^{d/2-\eta} \cap L^{d/2+\eta})(\R_+^d, (-\infty,0]) \subseteq \cP(\R_+^d, W^{1,2}(\R_+^d)),
$$
for some $\eta>0$.

\subsection*{Acknowledgements}
 The author was partially supported by the MIUR Excellence Department Project awarded to Dipartimento di Matematica, Universit\`a di Genova, CUPD33C23001110001, and the ``National Group for Mathematical Analysis, Probability and their Applications'' (GNAMPA-INdAM). He was partially supported by Progetto GNAMPA 2024: ``$L^p$ estimates for singular integrals in nondoubling settings" (CUP E53C23001670001).

He would like to thank  Andrea Carbonaro and Oliver Dragi\v{c}evi\'c for their support and guidance during the writing of this paper.

%

\bibliographystyle{amsxport}
\bibliography{biblio_mixed}

\end{document}